\documentclass[12pt]{article}
\usepackage{amsmath,amscd,amsthm,amssymb,amsfonts,mathtools}
\usepackage{setspace}
\usepackage{ytableau}
\usepackage{mathrsfs}
\usepackage{graphics}
\usepackage[new]{old-arrows}
\usepackage{hyperref}
\usepackage{epsfig}
\usepackage[utf8]{inputenc}
\usepackage[T1]{fontenc}
\usepackage[version=4]{mhchem}
\usepackage{xfrac} 
\usepackage{stmaryrd}
\usepackage{url}
\usepackage{tikz}
\usepackage{mathdots}
\usepackage{yhmath}
\usepackage{cancel}
\usepackage{color}
\usepackage{siunitx}
\usepackage{array}
\usepackage{authblk}
\usepackage{multirow}
\usepackage{gensymb}
\usepackage{enumitem}
\usepackage{tabularx}
\usepackage{extarrows}
\usepackage{booktabs}
\usetikzlibrary{decorations.markings}
\usetikzlibrary{fadings}
\usetikzlibrary{patterns}
\usetikzlibrary{shadows.blur}
\usetikzlibrary{shapes}
\usepackage{cleveref}
\usepackage{graphicx}
\usepackage{xcolor}
\usepackage{float}
\usepackage[top=1.5in,bottom=1in,left=1.25in,right=1.25in]{geometry}
\usepackage{tikz-cd}
\usepackage{braket}
\usepackage{custom_macros}
\usepackage{adjustbox}
\usetikzlibrary{positioning,arrows.meta}

\usepackage[normalem]{ulem}

\usepackage[framemethod=TikZ]{mdframed} 

\newmdenv[
  skipabove=10pt,
  skipbelow=10pt,
  backgroundcolor=blue!3,
  linecolor=blue!70!black,
  linewidth=1pt,
  roundcorner=10pt,
  leftmargin=-5pt,
  rightmargin=-5pt,
  innerleftmargin=10pt,
  innerrightmargin=10pt,
  innertopmargin=10pt,
  innerbottommargin=10pt
]{theorembox}

\DeclareMathOperator{\Hom}{Hom}

\DeclareMathOperator{\gl}{\mathfrak{gl}}

\DeclareMathOperator{\End}{End}

\DeclareMathOperator{\Spec}{Spec}

\DeclareMathOperator{\depth}{depth}
\DeclareMathOperator{\codim}{codim}

\def\CV{{\cal V}}

\newcommand{\bc}{\mathsf{bc}}

\newcommand{\jet}{J_{\infty}}
\newcommand{\CVs}{\CV^{\sharp}}

\newcommand{\JBRST}{\mathcal{J}}

\numberwithin{equation}{section}

\title{Chiralization of Quiver Varieties\bigskip}

\author[1]{Ioana Coman\thanks{icomanl@ed.ac.uk}}
\author[2]{Myungbo Shim\thanks{mbshim@tsinghua.edu.cn}}
\author[3,4,5]{Masahito Yamazaki\thanks{masahito.yamazaki@ipmu.jp}}
\author[6,7]{Yehao Zhou\thanks{yehao.zhou@simis.cn}}

\affil[1]{School of Mathematics and Maxwell Institute for Mathematical Sciences, University of Edinburgh, Edinburgh EH9 3FD, UK}
\affil[2]{Yau Mathematical Sciences Center, Tsinghua University, Beijing 100084, China}
\affil[3]{Kavli Institute for the Physics and Mathematics of the Universe (WPI), University of Tokyo, Kashiwa, Chiba 277-8583, Japan}
\affil[4]{Graduate School of Physics, University of Tokyo, Tokyo 113-0033, Japan}
\affil[5]{Trans-Scale Quantum Science Institute, University of Tokyo, Tokyo 113-0033, Japan}
\affil[6]{Center for Mathematics and Interdisciplinary Sciences, Fudan University, Shanghai 200433, China}
\affil[7]{Shanghai Institute for Mathematics and Interdisciplinary Sciences (SIMIS), Shanghai 200433, China}

\date{}

\begin{document}

\maketitle
\setstretch{1.2}

\begin{abstract}

Given a quiver $Q$ with gauge dimension $\mathbf v$ and framing dimension $\mathbf w$, one can define the extended quiver variety $\widetilde{\mathcal M}(\mathbf v,\mathbf w)$, which is a smooth family of deformations of the Nakajima quiver variety $\mathcal M(\mathbf v,\mathbf w)$. In this paper we discuss two vertex algebras which chiralize the geometry $\widetilde{\mathcal M}(\mathbf v,\mathbf w)$. We construct a sheaf of $\hbar$-adic vertex superalgebras $\mathscr D^{\mathrm{ch}}_{\widetilde{\mathcal M}(\mathbf v,\mathbf w),\hbar}$ on $\widetilde{\mathcal M}(\mathbf v,\mathbf w)$ which quantizes the jet bundle of $\widetilde{\mathcal M}(\mathbf v,\mathbf w)$, and define a vertex algebra $\mathsf D^{\mathrm{ch}}(\widetilde{\mathcal M}(\mathbf v,\mathbf w))$ to be the $\hbar=1$ specialization of the $\mathbb C^{\times}$-finite part of the vector space of global sections $\Gamma(\widetilde{\mathcal M}(\mathbf v,\mathbf w), \mathscr D^{\mathrm{ch}}_{\widetilde{\mathcal M}(\mathbf v,\mathbf w),\hbar})$.
We define another vertex superalgebra $\mathcal V(\mathbf v,\mathbf w)$
by BRST reduction of the tensor product of the $\beta\gamma bc$-system and Heisenberg VOA associated to the quiver $Q$, and show that there exists a natural vertex superalgebra map from $\mathcal V(\mathbf v,\mathbf w)$ to $\mathsf D^{\mathrm{ch}}(\widetilde{\mathcal M}(\mathbf v,\mathbf w))$. Under certain technical assumptions, we prove that the negative degree BRST cohomologies of the tensor product of $\beta\gamma bc$-systems and Heisenberg VOA associated to the quiver $Q$ are zero, and under stronger assumptions, that the aforementioned vertex superalgebra map is injective. 

Physically, the vertex superalgebra $\mathcal V(\mathbf v,\mathbf w)$ is closely related to the boundary VOA of the H-twisted 3D $\mathcal N=4$ quiver gauge theory associated to the quiver $Q$ with gauge and framing
dimension vectors $\mathbf v$ and $\mathbf w$.

\end{abstract}

{\tableofcontents}

\newpage 

\section{Introduction}\label{sec:intro}

Since the axiomatic definition of lattice vertex algebras by R. Borcherds \cite{Borcherds} in his work on the Monstrous Moonshine conjecture, and that of vertex operator algebras (VOAs) by I. Frenkel, J. Lepowsky and A. Meurman \cite{MR0996026}, 
vertex algebra research has established a wealth of relations to different areas of mathematics and thereby proven to be immensely prolific. 
Among these we find the representation theory of Lie algebras, geometry --- notably algebraic geometry of curves, moduli spaces and singular varieties; Poisson geometry and quantization of singular varieties; derived geometry and homological algebra involving BRST cohomology --- all related to the present work, as well as category theory and number theory (through vertex algebra modules), integrable systems, low-dimensional topology and mathematical physics applications in string and quantum field theory. While there exists no single universal construction of vertex algebras, a toolkit of complementary constructions has instead been developed over the past decades, with each approach adapted to the particular origin -- geometric, algebraic or physical -- of the vertex algebra. 

\medskip
 
\begin{figure}[ht] \centering
\begin{tikzpicture}[
    node distance=7cm,
    every node/.style={
        draw,
        rectangle,
        rounded corners=2pt,
        align=center,
        inner sep=6pt,
        font=\small,
        text width=6.7cm
    },
    dbl/.style={
        <->,
        dashed,
        thin
    }
]

\node (top) at (0,3.2)
{
Representation theory\\
of Lie algebras\\[2pt]
{\small (affine Kac--Moody algebras)} \\[2pt]
Free boson / fermion; Heisenberg VOA 
};

\node (left) at (-5,-2.5)
{
Chiral quantization\\
of Poisson schemes
};

\node (right) at (5,-2.5)
{
Physical (S)QFT\\
origin of vertex algebras
};

\node[text width=5cm] (center) at (0,0.2)
{
Symmetry-based reductions\\[2pt]
{\small (BRST, cosets, orbifolds)}
};

\draw[dbl] ([xshift=-1pt]top.south west) -- ([xshift=-1pt]left.north);
\draw[dbl] ([xshift=1pt]top.south east) -- ([xshift=1pt]right.north);
\draw[dbl] (left) -- (right);

\draw[dbl] (center) -- (top);
\draw[dbl] (center) -- (left);
\draw[dbl] (center) -- (right);

\end{tikzpicture}
\caption{Interplay within a network of relations of vertex algebra constructions from representation theory, chiral quantization and symmetry-based reductions, as well as physical realizations.} \label{fig:quiver-voa-diagram} 
\end{figure}
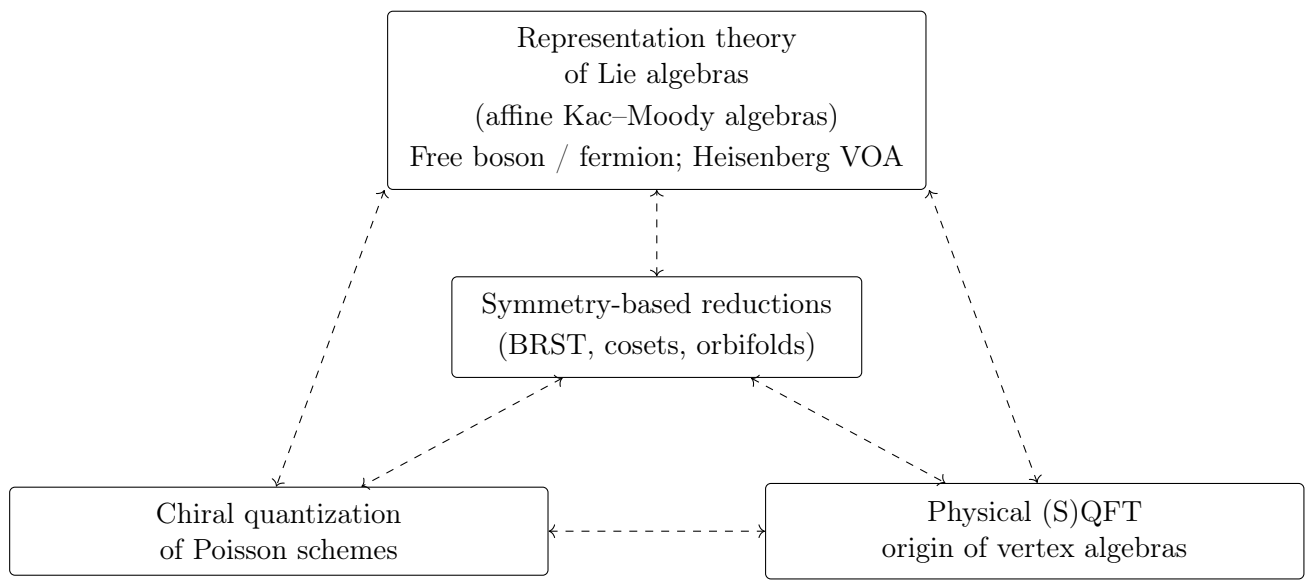

\medskip

A simplified network containing a subset of these relations is presented in Figure \ref{fig:quiver-voa-diagram}. The top node contains elementary algebraic constructions which will represent fundamental building blocks of the vertex algebras constructed in this paper --- specifically the free field realizations of symplectic boson or fermion VOAs, as well as Heisenberg algebras. At one level higher in complexity, one encounters affine Kac–Moody or current-algebra constructions, whose underlying vector spaces are vacuum modules of affine Lie algebras and which often admit realizations inside free-field VOAs. In contrast, the central node hosts more sophisticated constructions which start from a \emph{parent} vertex algebra defined using the fundamental building blocks, and which are based on reductions through orbifolds, coset / commutants, or via BRST / quantum Hamiltonian reductions. BRST reduction, in particular, implements a derived semi-infinite quotient of a parent vertex algebra by a current algebra using a ghost system and BRST differential, and its cohomology realizes the corresponding reduced vertex algebra, such as affine $\mathcal{W}$-algebras \cite{Feigin:1990pn}, \cite{kac2003quantum}, \cite{arakawa2015localization}. \cite{arakawa2017introduction}, \cite{MR1849359} and \cite{arakawa2024arc} are excellent recent resources with pedagogical presentations.

\medskip 

A complementary constructive approach towards vertex algebras comes from geometry, from the left-most node in Figure \ref{fig:quiver-voa-diagram} and the chiral quantization of affine Poisson varieties or Poisson schemes. Such a classical geometry represents the associated variety of the vertex algebra in question, where the arc space of the associated variety captures geometrical aspects of the vertex algebra through the associated graded algebra. The monograph \cite{arakawa2024arc} is currently the most comprehensive and up to date exposition emphasizing this geometrical aspect of vertex algebras. The relation between left-most and central nodes of Figure \ref{fig:quiver-voa-diagram}, which has been explored for instance in \cite{arakawa2015localization, kuwabara2017vertex, arakawa2023hilbert}, will play a key role in the work presented in this paper.

\medskip

The third perspective in the right-most node in Figure \ref{fig:quiver-voa-diagram} is from the physical origin of vertex algebras as distinguished algebras of operators inside quantum field theories in three dimensions or higher. An established example is the 3D/2D correspondence between Chern-Simons gauge theory with compact gauge group at integer level, placed on a three-dimensional manifold with boundary, and Wess-Zumino-Witten algebra supported on the boundary \cite{Witten:1988ze, Elitzur:1989nr}. A modern analog of this correspondence is realized within $\mathcal{N}=4$ supersymmetric quantum field theories in three dimensions on manifolds with a transverse holomorphic foliation structure, whose topologically twisted sector supports boundary vertex algebra \cite{Gaiotto:2016wcv, Costello:2018fnz}. These new correspondences have been studied intensely in recent years, in particular in the context of 3D $\mathcal{N}=4$ Abelian gauge theories \cite{Beem:2023dub, Ferrari:2024dst, Ballin:2023rmt}, the so-called rank-0 theories \cite{Gang:2023rei, Ferrari:2023fez} which have zero-dimensional moduli spaces of vacua, 3D $\mathcal{N}=4$ Chern-Simons-matter theories \cite{Creutzig:2021ext} which admit, upon a topological twist, holomorphic boundary conditions supporting logarithmic vertex operator algebras, and a systematic study of 3D $\mathcal{N}=4$ non-Abelian quiver gauge theories \cite{Coman:2023xcq}. From this physical perspective, the role of the classical Poisson geometry from Figure \ref{fig:quiver-voa-diagram}, which is the associated variety of the vertex algebra under consideration and related to this via chiral quantization, is played by (a branch of) the moduli space of vacua of three-dimensional theory. This type of analysis has in turn uncovered surprising vertex algebra dualities which had been motivated physically and related to 3D mirror symmetry \cite{Creutzig:2024ljv, Ballin:2023rmt}, and established the existence of a braided tensor structure on categories of line operators in the three-dimensional quantum field theory, which end on the vertex algebra-supporting boundary, and which are viewed as non-semisimple (derived bounded) categories of modules for the boundary vertex algebras \cite{Creutzig:2021ext, Ballin:2023rmt}. Furthermore, the chiral quantization construction of boundary vertex algebras \cite{Coman:2023xcq} has allowed for their detailed local analysis, and the study of relations among VOAs via embeddings and free-field realizations.

\medskip

In the following subsection, we will give a brief introduction to the precise context in which we construct and analyze vertex algebras in this paper, specifically via the \emph{global} and respectively \emph{local} chiral quantizations of certain Poisson schemes defined from representation spaces of quivers. We will first state in a somewhat informal manner the key mathematical results of the paper --- which will ultimately characterize the precise relation between vertex algebras resulting from the local and global chiral quantization procedures --- together with the logic flow which connects the core statements and the intuition which lies behind them, deferring the presentation of the rigorous statements and their proofs to the subsequent sections.

\subsection{Core mathematical results}

\cite{Coman:2023xcq} initiated a systematic study of vertex operator algebras constructed as chiralizations, or chiral quantizations \cite{arakawa2024arc}, of quiver varieties. In this context, features of this classical geometry are reflected in the resulting vertex algebra structure, with the quiver variety recovered as the associated variety of the VOA. The input data for the chiralization procedure is comprised of a quiver $Q$ and dimension vectors $\mathbf v, \, \mathbf w$, which define an associated Cartan matrix $\mathsf C$ and reductive Lie group $\mathrm G$. 

\medskip

The data $(Q,\mathbf v,\mathbf w)$ then determines the representation space $T^*\mathrm{Rep}(Q,\mathbf v,\mathbf w)$ and the extended representation space $\mathfrak{R}(Q)=T^*\mathrm{Rep}(Q,\mathbf v,\mathbf w)\times\mathcal{Z}$ of the quiver, on which the group $\mathrm G$ has a Hamiltonian action with moment map $\mu:T^*\mathrm{Rep}(Q,\mathbf v,\mathbf w)\to\mathfrak g^\ast$ and extended moment map $\widetilde\mu:\mathfrak{R}\to\mathfrak g^\ast$, where $\mathfrak g^\ast$ is the dual Lie algebra of $\mathrm G$ and $\mathcal{Z}$ is a particular linear subspace of $\mathfrak g^\ast$. We will assume throughout the paper that the quiver data satisfies a technical \emph{balance} condition, specifically that each component of $\mathbf w-\mathsf{C}\mathbf v$ is a non-negative integer. The extended Nakajima quiver variety $\widetilde{\mathcal{M}}(Q)$ is then defined as a Hamiltonian reduction $\widetilde{\mathcal{M}}(Q)=\widetilde\mu^{-1}(0)^{\mathrm{st}}\sslash  \mathrm{G}$ from stable quiver representations. For notational ease, we will henceforth drop the explicit quiver dependence when referring to the extended representation space $\mathfrak{R}$ of the quiver and Nakajima quiver variety $\widetilde{\mathcal{M}}$. 

\medskip 

The chiralization procedure naturally promotes the extended representation space $\mathfrak{R}$ of the quiver to a vertex superalgebra $\mathcal{D}^{\mathrm{ch}}(\mathfrak{R})$ given by the tensor product of a collection of $\beta\gamma\mathsf{bc}$ systems and Heisenberg algebras, with a chiral counterpart $\widetilde\mu^{\mathrm{ch}}$ of the moment map $\widetilde\mu: V^{-\kappa_{\mathfrak{g}}}\to \mathcal{D}^{\mathrm{ch}}(\mathfrak{R})$ where $\kappa_{\mathfrak{g}}$ is the Killing form on $\mathfrak{g}$. The quantum counterpart of the Hamiltonian reduction defining the Nakajima quiver variety $\widetilde{\mathcal{M}}$ is a BRST reduction procedure which computes BRST cohomology, with two vertex superalgebras thus constructed in \cite{Coman:2023xcq}:

\medskip

\begin{itemize}
    \item $\mathcal{V}=H^{\infty/2+0}_{\mathrm{BRST}}(\mathfrak{g},\mathcal{D}^{\mathrm{ch}}(\mathfrak{R}))$, defined by a \emph{global} BRST reduction from the vertex superalgebra $\mathcal{D}^{\mathrm{ch}}(\mathfrak{R})$ which is the chiralization of the extended representation space $\mathfrak{R}$ of the quiver;
    \item $\mathsf D^{\mathrm{ch}}(\widetilde{\mathcal M})$ constructed from the \emph{local} BRST reduction which defines a sheaf of $\hbar$-adic vertex superalgebras $\mathscr D^{\mathrm{ch}}_{\widetilde{\mathcal M},\hbar}=\mathscr{H}^{\infty/2+0}_{\hbar\mathrm{BRST}}(\mathfrak{g},\mathscr D^{\mathrm{ch}}_{\mathfrak{R},\hbar})$ on $\widetilde{\mathcal M}$ from a sheaf of $\hbar$-adic vertex superalgebras $\mathscr D^{\mathrm{ch}}_{\mathfrak{R},\hbar}$ on $\mathfrak{R}$, where $\Gamma(\mathfrak{R},\mathscr D^{\mathrm{ch}}_{\mathfrak{R},\hbar})=\mathcal{D}^{\mathrm{ch}}(\mathfrak{R})_\hbar$ is the $\hbar$-adic completion of $\mathcal{D}^{\mathrm{ch}}(\mathfrak{R})$. The sheaf $\mathscr D^{\mathrm{ch}}_{\widetilde{\mathcal M},\hbar}$ quantizes the jet bundle $J_\infty\widetilde{\mathcal M}$, and the $\hbar=1$ specialization of the $\mathbb C^{\times}$-finite part of its vector space of global sections $\Gamma(\widetilde{\mathcal M}, \mathscr D^{\mathrm{ch}}_{\widetilde{\mathcal M},\hbar})$ defines $\mathsf D^{\mathrm{ch}}(\widetilde{\mathcal M})$.
\end{itemize}

\medskip

At the semi-classical level, we therefore work with the infinite jets --- or arc spaces --- of the extended representation space of the quiver and the Nakajima quiver variety, denoted by $J_\infty\mathfrak{R}$ and $J_\infty\widetilde{\mathcal{M}}$ respectively, and their structure sheaves $\mathcal{O}_{J_\infty\mathfrak{R}}$ and $\mathcal{O}_{J_\infty\widetilde{\mathcal{M}}}$. More precisely, there are four perspectives to consider: \emph{global} vs.\ \emph{local}, and \emph{before} vs.\ \emph{after} BRST reduction, as illustrated in four quadrants of Figure \ref{table:global-VS-local-semiclassical}, with natural relations among the quadrants. 

\bigskip

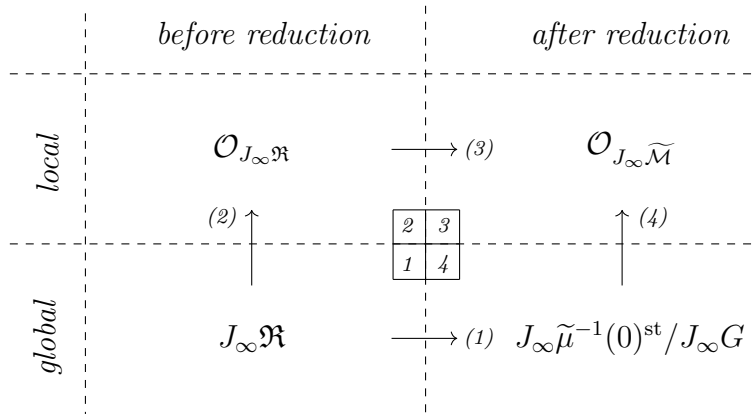
\begin{figure}[h]
  \centering
\begin{center}
\begin{tikzpicture}[
  label/.style={font=\itshape}
]
\node[] (A) at (0,0) {$\mathcal{O}_{J_\infty \mathfrak R}$};
\node[] (B) at (5,0) {$\mathcal{O}_{J_\infty\widetilde{\mathcal{M}}}$};
\node[] (C) at (-0.0,-2.5) {$J_\infty\mathfrak R$};
\node[] (D) at (5,-2.5) {$J_\infty\widetilde\mu^{-1}(0)^{\mathrm{st}}/J_\infty G$};

\draw[dashed] (2.3,1.9) -- (2.3,-3.5);   
\draw[dashed] (-2.2,1.8) -- (-2.2,-3.5);   
\draw[dashed] (-3.2,-1.25) -- (6.7,-1.25); 
\draw[dashed] (-3.2,1.0) -- (6.7,1.0); 

\draw[<-] (0.0,-0.8) -- (0.0,-1.8);   
\draw[<-] (4.90,-0.8) -- (4.90,-1.8);   
\draw[->] (1.85,-2.5) -- (2.72,-2.5);   
\draw[->] (1.85,0.0) -- (2.72,0.0);   

\node[label] at (0.15,1.5) {before reduction};
\node[label] at (5,1.5) {after reduction};

\node[label, rotate=90] at (-2.7,-0.2) {local};
\node[label, rotate=90] at (-2.7,-2.5) {global};

\node[label] at (3.0,0.0) {\scriptsize (3)};
\node[label] at (3.0,-2.5) {\scriptsize (1)};
\node[label] at (5.3,-0.90) {\scriptsize (4)};
\node[label] at (-0.4,-0.90) {\scriptsize (2)}; 

\node[label] at (2.53,-1.52) {\scriptsize 4};
\node[label] at (2.05,-1.52) {\scriptsize 1};
\node[label] at (2.53,-1.0) {\scriptsize 3};
\node[label] at (2.05,-1.0) {\scriptsize 2}; 

\draw[solid] (1.87,-1.72) -- (1.87,-0.8);   
\draw[solid] (2.75,-1.72) -- (2.75,-0.8);   
\draw[solid] (2.3,-1.72) -- (2.3,-0.8);   
\draw[solid] (1.87,-0.8) -- (2.75,-0.8);   
\draw[solid] (1.87,-1.72) -- (2.75,-1.72);   
\draw[solid] (1.87,-1.25) -- (2.75,-1.25);   

\end{tikzpicture}
\end{center}
\caption{Semiclassical descriptions of the vertex superalgebras associated to the quiver data $(Q,\mathbf v,\mathbf w)$, at the global vs. local level, before vs. after the Hamiltonian reduction procedure, with natural relations connecting the quadrants.} 
\label{table:global-VS-local-semiclassical}
\end{figure}

\bigskip

Quantizing the picture from Figure \ref{table:global-VS-local-semiclassical} lifts this to its chiral analogue, as outlined in Figure \ref{table:global-VS-local-quantization}. Chiralization furthermore lifts functions from $\mathbb{C}[\widetilde{\mathcal{M}}]$ to BRST invariant currents generating the quiver vertex superalgebra $\mathcal{V}$.

\bigskip

\begin{figure}[h]
  \centering
\begin{center}
\begin{tikzpicture}[
  label/.style={font=\itshape}
]

\node[] (A) at (0,0) {$\hbar$-adic VOA $\mathscr D^{\mathrm{ch}}_{\mathfrak R,\hbar}$};
\node[] (B) at (6.1,0) {$\mathscr{D}^{ch}_{\widetilde{\mathcal{M}},\hbar}=\mathscr H^{\infty/2+0}_{\hbar\mathrm{BRST}}(\mathfrak{g},\mathscr D^{\mathrm{ch}}_{\mathfrak R,\hbar})$};
\node[] (C) at (-0.0,-2.5) {VOA $\mathcal{D}^{\mathrm{ch}}(\mathfrak{R})$};
\node[] (D) at (6.1,-2.5) {$\mathcal{V}=H^{\infty/2+n}_{\mathrm{BRST}}(\mathfrak{g},\mathcal{D}^{\mathrm{ch}}(\mathfrak{R}))$};

\draw[dashed] (2.3,1.9) -- (2.3,-3.5);   
\draw[dashed] (-2.2,1.8) -- (-2.2,-3.5);   
\draw[dashed] (-3.2,-1.25) -- (8.9,-1.25); 
\draw[dashed] (-3.2,1.0) -- (8.9,1.0); 

\draw[<-] (0.0,-0.8) -- (0.0,-1.8);   
\draw[<-] (4.50,-0.8) -- (4.50,-1.8);   
\draw[->] (1.85,-2.5) -- (2.72,-2.5);   
\draw[->] (1.85,0.0) -- (2.72,0.0);   

\node[label] at (0.15,1.5) {before reduction};
\node[label] at (5.35,1.5) {after reduction};

\node[label, rotate=90] at (-2.7,-0.2) {local};
\node[label, rotate=90] at (-2.7,-2.5) {global};

\node[label] at (3.0,0.0) {\scriptsize (3)};
\node[label] at (3.0,-2.5) {\scriptsize (1)};
\node[label] at (4.9,-0.90) {\scriptsize (4)};
\node[label] at (-0.4,-0.90) {\scriptsize (2)}; 

\node[label] at (2.53,-1.52) {\scriptsize 4};
\node[label] at (2.05,-1.52) {\scriptsize 1};
\node[label] at (2.53,-1.0) {\scriptsize 3};
\node[label] at (2.05,-1.0) {\scriptsize 2}; 

\draw[solid] (1.87,-1.72) -- (1.87,-0.8);   
\draw[solid] (2.75,-1.72) -- (2.75,-0.8);   
\draw[solid] (2.3,-1.72) -- (2.3,-0.8);   
\draw[solid] (1.87,-0.8) -- (2.75,-0.8);   
\draw[solid] (1.87,-1.72) -- (2.75,-1.72);   
\draw[solid] (1.87,-1.25) -- (2.75,-1.25);   

\end{tikzpicture}
\end{center}
\caption{Global (resp.\ local) descriptions of the VOAs associated to the quiver data $(Q,\mathbf v,\mathbf w)$ before and after BRST reduction.}
\label{table:global-VS-local-quantization}
\end{figure}
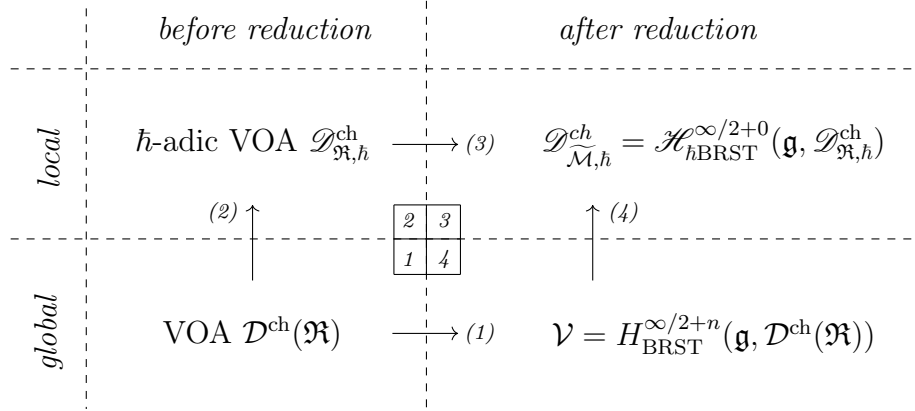

\bigskip

From their construction, there exists a natural vertex superalgebra map $\mathcal V \to \mathsf D^{\mathrm{ch}}(\widetilde{\mathcal M})$.
One of the main mathematical results of this paper provides the following insights into the nature of this map for families of quivers, under certain technical assumptions on the quiver data $(Q,\mathbf v, \mathbf w)$. 

\begin{theorembox}
        \begin{theorem}\label{thm:vanishing and embedding-Intro} 
        For quiver data $(Q,\mathbf v, \mathbf w)$ chosen such that each component of $\mathbf w-\mathsf C\mathbf v$ is a non-negative integer, under some additional conditions, we have
        \begin{enumerate}
            \item $H^{\infty/2+n}_{\mathrm{BRST}}(\mathfrak{g},\mathcal{D}^{\mathrm{ch}}(\mathfrak{R}))$ vanishes for $n<0$,
            \item the natural vertex superalgebra map $\CV\to \mathsf D^{\mathrm{ch}}(\widetilde\CM)$ is injective.
        \end{enumerate}
        \end{theorem}
\end{theorembox}

The precise statement can be found in Theorem \ref{thm:vanishing and embedding} in the main body of this paper, providing a proof of Conjecture \ref{thm:vanishing and embedding} of \cite{Coman:2023xcq}. 

\bigskip 

The key to proving point $1.$ is the jet-flatness of the moment map $\widetilde\mu$, which is primarily a condition about infinitesimal (i.e.\ jet) behavior and which we will refer to as property $(\widetilde{P}_2)$. A flat morphism of schemes $f:X\to B$ is a morphism which preserves regular sequences, where the induced local ring map $f^\ast_x:\mathcal{O}_{B,f(x)}\to \mathcal{O}_{X,x}$ at any point $x\in X$ makes $\mathcal{O}_{X,x}$ a flat $\mathcal{O}_{B, f(x)}$-module. Geometrically, there are no sudden jumps in the dimension of the fibers $X_p = f^{-1}(p)$ for $p\in B$. In the present context, the jet-flatness of the moment map is the flatness of $J_n\widetilde\mu,\,\forall n\in\mathbb{Z}_{\geq 0}$, which is equivalent to $\{\partial^n\widetilde\mu^*(a_i)\}_{1\le i\le \dim \mathfrak{g}}$ being a regular sequence, where $\{a_i\}_{1\le i\le \dim \mathfrak{g}}$ is a basis of $\mathfrak{g}$ and $J_n\widetilde\mu^*:\mathcal O_{J_n\mathfrak{g}^*}\to \mathcal O_{J_n\mathfrak{R}}$ is the pull-back map between function rings. Following \cite{arakawa2015localization}, the regularity of the sequence $\{\partial^n \widetilde\mu^*(a_i)\}_{1\le i\le \dim \mathfrak{g}},\,\forall n\in\mathbb{Z}_{\geq 0}$ determines in Theorems \ref{lem:vanishingHigherDegCoh-SemiClassical} and \ref{thm:vanishingHigherDegCoh} the vanishing of cohomology $H^{\infty/2+n}_{\mathrm{BRST}}(\mathfrak{g},\mathcal{O}_{J_\infty\mathfrak{R}})$ and $H^{\infty/2+n}_{\mathrm{BRST}}(\mathfrak{g},\mathcal{D}^{\mathrm{ch}}(\mathfrak{R}))$ in negative degree.

\bigskip 

Point $2.$ of Theorem \ref{thm:vanishing and embedding-Intro}, the injectivity of the vertex superalgebra map $\CV\to \mathsf D^{\mathrm{ch}}(\widetilde\CM)$ (see Theorem \ref{thm:VOAinjectivity} and Theorem \ref{thm:VOA-module-injectivity} for the corresponding natural map of vertex superalgebra modules $\CV_\lambda\to \mathsf D^{\mathrm{ch}}_\lambda(\widetilde\CM)$), holds under stronger conditions on the moment map and the singular structure of $\widetilde\mu^{-1}(0)$, which we will refer to as property $(\widetilde{P}_1)$. Property $(\widetilde{P}_1)$ requires in particular that the extended moment map $\widetilde\mu$ be flat, and that $\widetilde\mu^{-1}(0)$ be reduced, irreducible and with rational singularities, with property $(\widetilde{P}_1)$ implying property $(\widetilde{P}_2)$ by Proposition \ref{prop: various jet flatness}. In particular, the requirement of property $(\widetilde{P}_1)$ for $\widetilde\mu^{-1}(0)$ to have rational singularities implies that the $n$-jet $J_n\widetilde\mu^{-1}(0)$ is reduced and irreducible for all $n\in\mathbb{Z}_{\geq 0}$, which then implies the vanishing of local cohomology sheaves on $J_n\widetilde\mu^{-1}(0)$ in degree zero; this will be referred to as $J_n\widetilde\mu^{-1}(0)$ having the property "$(S_1)$". In turn, having property $(S_1)$ leads to the following semiclassical result. 

\bigskip

Semiclassically, if the quiver data $(Q,\mathbf v, \mathbf w)$ is such that the conditions of property $(\widetilde{P}_1)$ are satisfied, then there are natural injective maps at the level of function rings.

\bigskip 

        \begin{theorembox}
        \begin{theorem}\label{thm:semiclassical-injectivity-Intro}
        If $\widetilde\mu$ is flat, and $\widetilde\mu^{-1}(0)$ is reduced and irreducible and has rational singularities, then, for all $n\in \mathbb Z_{\ge 0}$, the following natural maps are injective 
        \begin{align*}
            \mathbb C[J_n\widetilde\mu^{-1}(0)]^{J_n\mathrm{G}}\longrightarrow \Gamma(J_n\widetilde{\mathcal M},\mathcal O_{J_n\widetilde{\mathcal M}})~.
        \end{align*}
        \end{theorem}
        \end{theorembox}

\medskip 

The full statement can be found in Theorem \ref{thm:semiclassical-injectivity} in the main body of the paper, and is the key input of the proof of Theorem \ref{thm:VOAinjectivity} on the injectivity of the vertex superalgebra map $\CV\to \mathsf D^{\mathrm{ch}}(\widetilde\CM)$.

\bigskip 

Given a particular set of quiver data $(Q,\mathbf v, \mathbf w)$, the essential properties required to prove the two points of Theorem \ref{thm:vanishing and embedding-Intro} are $(\widetilde{P}_2)$ (jet-flatness of $\widetilde\mu$) and respectively $(\widetilde{P}_1)$ (flatness of $\widetilde\mu$, and $\widetilde\mu^{-1}(0)$ being reduced, irreducible and with rational singularities). These properties must therefore be translated into algebraic conditions on the quiver data $(Q,\mathbf v, \mathbf w)$, which are precisely those provided by Theorems \ref{thm: property P} - \ref{thm: property P tilde}. These conditions are based on: 
\begin{itemize}
    \item the flatness of the moment map, which is determined by a combinatorial criterion due to \cite{crawley2001geometry}, and which implies that $\widetilde\mu^{-1}(0)$ is reduced and irreducible (Propositions \ref{prop: flat moment map} and \ref{prop: flat moment map 2}); 
    \item dimension criteria on the scheme of $m$-jets passing through the singular locus of $\widetilde\mu^{-1}(0)$ due to \cite{mustata2001jet}, and which in turn implies the rationality of singularities of $\widetilde\mu^{-1}(0)$; 
    \item the stratification of the zero locus of the moment map by representation type "$\tau$", with a distinguished most singular stratum; we develop a generalization of the étale slice construction of \cite{crawley2003normality} in Appendix \ref{sec:étaleSlice}, which defines for a representation $x\in\widetilde{\mu}^{-1}(0)\sslash \mathrm{G}$ of type $\tau$ an auxiliary quiver $(Q_\tau,\mathbf{v}_\tau,\mathbf{w}_\tau)$, where a neighborhood $U\subset \widetilde{\mu}^{-1}(0)\sslash  \mathrm{G}$ of $x$ is isomorphic in étale topology to a neighborhood $U_{(x),\tau}$ of $0\in\hat{\widetilde\mu}^{-1}(0)\sslash  G_x$, where $\hat{\widetilde\mu}^{-1}(0)$ is an extension of $\widetilde\mu^{-1}_\tau(0)$ for the auxiliary quiver, and where the auxiliary quiver also defines $G_x$ (we refer to Appendix \ref{sec:étaleSlice} for further detail). The dimension criteria from the second bullet point are therefore applied to the auxiliary quivers in order to demonstrate the jet-flatness of the moment map, and rationality of the singular points of $\widetilde{\mu}^{-1}(0)$.
\end{itemize}
Having established the algebraic conditions of Theorem \ref{thm: property P tilde} which imply the properties $(\widetilde{P}_2)$ and $(\widetilde{P}_1)$, we then prove in Section \ref{sec:examples} that these conditions are indeed satisfied for particular families of quiver data $(Q, \mathbf v, \mathbf w)$, which in turn provide the final input element in the proof of Theorem \ref{thm:vanishing and embedding-Intro}.

\subsection{Structure of the paper}

In Section \ref{sec:GeomQuiverVarieties} we review key statements pertaining to the geometry of extended quiver (super)-varieties. Propositions \ref{prop: stable locus}-\ref{prop: flat moment map 2} in particular address the flatness of the moment map $\mu$, equivalent conditions on the quiver and dimension data and their implications for the reducedness and irreducibility of preimages under the moment map, as well as natural maps between function rings. Proposition \ref{prop: flat extended moment map} then lists equivalent conditions for the flatness of extended moment maps. This section also summarizes the construction of the simple locus in the representation space of a quiver, with Propositions \ref{prop: simple locus}-\ref{prop: simple locus criterion} providing numerical conditions for the regular locus $\mu^{-1}(0)^{\mathrm{reg}}$ to be non-empty and implications for the reducedness and irreducibility of $\mu^{-1}(0)$, and flatness of the moment map. This covers the \emph{classical} aspects of the chiralization procedure of quiver varieties. 

Section \ref{sec:JetFlatness} then covers the \emph{semi-classical} aspects of this procedure. In particular, Proposition \ref{prop: various jet flatness} introduces the properties $(\widetilde{P}_1)$ --- that the extended moment map $\widetilde\mu$ be flat, and that $\widetilde\mu^{-1}(0)$ be reduced, irreducible and with rational singularities --- and $(\widetilde{P}_2)$ (jet-flatness of $\widetilde\mu$), as well as their non-extended counterparts $({P}_1)-({P}_2)$, and demonstrates how property $(\widetilde{P}_1)$ implies the others. Two key theorems in this section are Theorems \ref{thm: property P}-\ref{thm: property P tilde}, which provide conditions (dimension criteria) for the jet-flatness of quiver moment maps, and for $\mu^{-1}(0)$ and respectively $\widetilde\mu^{-1}(0)$ to be reduced, irreducible and with rational singularities. Examples of quivers for which these conditions are satisfied are then discussed in Section \ref{sec:examples}. The remaining key theorems in this section are Theorems \ref{thm:semiclassical-injectivity}-\ref{thm:semiclassical-injectivity3}, which assert the injectivity of natural maps between function rings on jet spaces, when the quiver and dimension data is such that properties introduced in Proposition \ref{prop: various jet flatness} hold.   

Sections \ref{sec:ChiralizationQuiverRep}-\ref{sec:VSA-globalsections} are then focused on the \emph{chiral quantum} aspects of the chiralization procedure. Specifically, Sections \ref{sec:ChiralizationQuiverRep}-\ref{sec:QuiverSheavesVSA} review the construction of quiver vertex superalgebras from \cite{Coman:2023xcq} as chiralizations of quiver varieties via the \emph{global} and respectively \emph{local} BRST reduction procedures, which we have outlined in Figure \ref{table:global-VS-local-quantization}, and which have the semi-classical description presented in Figure \ref{table:global-VS-local-semiclassical}. Section \ref{sec:VSA-globalsections} then contains the main statements relating the quiver vertex superalgebras constructed by the global and respectively local chiralization procedures. Theorems \ref{lem:vanishingHigherDegCoh-SemiClassical}-\ref{thm:vanishingHigherDegCoh} demonstrate the vanishing of BRST cohomology in negative degrees, under the condition of jet-flatness of the moment map. Theorem \ref{thm:VOAinjectivity} then proves the injectivity of natural maps between the quiver vertex superalgebras constructed in Sections \ref{sec:ChiralizationQuiverRep}-\ref{sec:QuiverSheavesVSA}, under the conditions that the extended moment map $\widetilde{\mu}$ is flat and $\widetilde{\mu}^{-1}(0)$ is reduced, irreducible and with rational singularities (property $(\widetilde{P}_1)$) --- with Theorem \ref{thm:VOA-module-injectivity} providing the analogous statement for modules of the quiver vertex superalgebra. Having demonstrated these results, Theorem \ref{thm:vanishing and embedding} then proves their applicability to specific classes of quivers and dimension data.

The appendices contain the derivation of technical results, and some relevant vertex algebra background. In particular, Appendix \ref{sec:étaleSlice} reviews and generalizes the étale slice construction of quiver varieties from \cite{crawley2003normality}, while Appendix \ref{appendix:VA-PVA} collects vertex algebra and vertex Poisson algebra definitions.

\subsection*{Acknowledgments}
I.C.\ is supported by the Horizon Europe MSCA grant No.101204790, \emph{QFT2VOA-DuStRel}. M.S.\ is supported in part by the Beijing Natural Science Foundation (BJNSF) Grant No.~IS24010 and by the Shuimu Scholar Program of Tsinghua University. M.Y.\ is supported in part by the JSPS Grant-in-Aid for Scientific Research (Grant No.~23K25865) and by JST, Japan (CREST Grant No.~JPMJCR26XA, Moonshot R\&D Grant No.~JPMJMS2061). During preparation of this paper we learned of the independent, related work \cite{Ferrari:2026} by Tomoyuki Arakawa, Andrea Ferrari and Sven M{\"o}ller, and coordinated simultaneous submissions. We thank the authors for their correspondence.

\section{Geometry of Quiver Varieties}\label{sec:GeomQuiverVarieties}

In this section we summarize the classical geometries of (extended) quiver varieties.

\subsection{Quiver varieties}\label{sec:GeomQuiverVarieties-intro}
A quiver $Q$ is an oriented graph which consists of a finite set of nodes $Q_0$ and a finite set of arrows $Q_1$ together with two maps $\mathsf h,\mathsf t: Q_1\to Q_0$ sending an arrow to its head and tail, respectively. We define the adjacency matrix $\mathsf Q$ to be
\begin{align}
    \mathsf Q_{ij}:=\#(a\in Q_1:\mathsf h(a)=j,\mathsf t(a)=i)~,
\end{align}
and correspondingly, define the Cartan matrix $\mathsf C$ to be
\begin{align}
    \mathsf C_{ij}:=2\delta_{ij}-\mathsf Q_{ij}-\mathsf Q_{ji}~.
\end{align}
The space of representations of $Q$ with gauge dimension vector $\mathbf v$ and framing dimension vector $\mathbf w$ is defined to be
\begin{align}
    \mathrm{Rep}(\mathbf v,\mathbf w):=\bigoplus_{i\in Q_0}\Hom(W_i,V_i)\oplus \bigoplus_{a\in Q_1}\Hom(V_{\mathsf t(a)},V_{\mathsf h(a)}) ~,
\end{align}
where $V=\bigoplus_{i\in Q_0}V_i$ and $W=\bigoplus_{i\in Q_0}W_i$ are $Q_0$-graded vector spaces of dimensions $\dim V_i=v_i$, with $(v_i)_{i\in Q_0}=\mathbf v$, and $\dim W_i=w_i$, with $(w_i)_{i\in Q_0}=\mathbf w$ respectively. The space of representations $\mathrm{Rep}(\mathbf v,\mathbf w)$ naturally possesses an action of 
$$\GL(\mathbf v):=\prod_{i\in Q_0}\GL( v_i) ~,$$
and this action induces a Hamiltonian $\GL(\mathbf v)$-action on the cotangent bundle $T^*\mathrm{Rep}(\mathbf v,\mathbf w)$. Note that $T^*\mathrm{Rep}(\mathbf v,\mathbf w)$ is naturally isomorphic to the space of representations of the doubled quiver $\overline{Q}$, which has the same set of nodes as $Q$, while the set of arrows is $\overline{Q}_1=Q_1\sqcup Q^*_1$, where $Q^*_1$ is the orientation reverse of $Q_1$. Explicitly, the cotangent bundle $T^*\mathrm{Rep}(\mathbf v,\mathbf w)$ of the space of representations is
\begin{multline}
    T^*\mathrm{Rep}(\mathbf v,\mathbf w)=\bigoplus_{i\in Q_0}\Hom(W_i,V_i)\oplus \bigoplus_{i\in Q_0}\Hom(V_i,W_i) \\
    \oplus\bigoplus_{a\in Q_1}\Hom(V_{\mathsf t(a)},V_{\mathsf h(a)})\oplus\bigoplus_{a\in Q_1}\Hom(V_{\mathsf h(a)},V_{\mathsf t(a)})~.
\end{multline}
We introduce the notation for the elements in the direct summands in $T^*\mathrm{Rep}(\mathbf v,\mathbf w)$ as follows
\begin{center}
\begin{tabular}{ | c | c | c | c |} 
 \hline
 $\Hom(V_{\mathsf t(a)},V_{\mathsf h(a)})$ & $\Hom(V_{\mathsf h(a)},V_{\mathsf t(a)})$ & $\Hom(W_i,V_i)$ & $\Hom(V_i,W_i)$  \\  [1ex] 
 \hline
$x_a$ & $y_a$ & $\alpha_i$ & $\beta_i$  \\ [0.5ex] 
 \hline
\end{tabular}
\end{center}
The moment map 
\begin{equation} 
\mu:T^*\mathrm{Rep}(\mathbf v,\mathbf w)\to \mathfrak{gl}(\mathbf v)^*
\end{equation}
for the Hamiltonian $\GL(\mathbf v)$-action on $T^*\mathrm{Rep}(\mathbf v,\mathbf w)$ is
\begin{align}\label{moment map} 
    \mu(x,y,\alpha,\beta)=\sum_{a\in Q_1}[x_a,y_a]+\sum_{i\in Q_0}\alpha_i\beta_i ~,
\end{align}
with pullback map between the function rings 
\begin{equation}
    \mu^\ast :\mathcal{O}_{\gl(\mathbf v)^\ast} \to \Gamma(T^*\mathrm{Rep}(\mathbf v,\mathbf w),\mathcal{O}_{T^*\mathrm{Rep}(\mathbf v,\mathbf w)}) ~.
\end{equation}
Inside $\mathfrak{gl}(\mathbf v)^*=\bigoplus_{i\in Q_0}\mathfrak{gl}( v_i)^*$ there is a linear subspace $\mathcal Z$ which is the span of scalar matrices 
\begin{align}\label{the center Z}
    \mathcal Z:=\bigoplus_{i\in Q_0} \mathbb C\cdot\mathrm{Id}_{V_i} ~.
\end{align}

\medskip

\begin{defn}
The affine quiver variety ${\mathcal M}^0(\mathbf v,\mathbf w)$ and its deformation $\widetilde{\mathcal M}^0(\mathbf v,\mathbf w)$ are defined as the Hamiltonian reductions
\begin{align}\label{def:affinequivervariety}
    {\mathcal M}^0(\mathbf v,\mathbf w):=\mu^{-1}(0)\sslash \GL(\mathbf v),\quad \widetilde{\mathcal M}^0(\mathbf v,\mathbf w):=\mu^{-1}(\mathcal Z)\sslash \GL(\mathbf v)~.
\end{align}
\end{defn}

\medskip

\begin{remark}
It is known that if the moment map $\mu$ is flat then both ${\mathcal M}^0(\mathbf v,\mathbf w)$ and $\widetilde{\mathcal M}^0(\mathbf v,\mathbf w)$ are reduced \cite[Theorem B]{zhou2022reducedness}. It does not seem to be known whether the scheme ${\mathcal M}^0(\mathbf v,\mathbf w)$ is reduced in general. On the other hand it is known that ${\mathcal M}^0(\mathbf v,\mathbf w)_{\mathrm{red}}$ is an irreducible symplectic singularity \cite[Theorem 1.2]{bellamy2016symplectic}. Here $X_{\mathrm{red}}$ denotes the reduced subscheme of a scheme $X$. 
\end{remark}

Note that ${\mathcal M}^0(\mathbf v,\mathbf w)$ is a closed subscheme of $\widetilde{\mathcal M}^0(\mathbf v,\mathbf w)$. In fact, since $\GL(\mathbf v)$ acts on $\mathcal Z$ trivially, the natural projection $\mu^{-1}(\mathcal Z)\to \mathcal Z$ descends to a morphism 
\begin{equation} \label{p0morphism-Mtilde0-to-Z}
p_0: \widetilde{\mathcal M}^0(\mathbf v,\mathbf w)\to \mathcal Z~,
\end{equation} 
and moreover ${\mathcal M}^0(\mathbf v,\mathbf w)\cong p_0^{-1}(0)$. Alternatively, one can define the extended moment map 
\begin{equation}\label{extended moment map}
\begin{split}
    \widetilde\mu: T^*\mathrm{Rep}(\mathbf v,\mathbf w)\times\mathcal Z &\longrightarrow \gl(\mathbf v)^*\\
    ((x,y,\alpha,\beta),z)&\mapsto \mu(x,y,\alpha,\beta)+ z~,
\end{split}
\end{equation}
and similarly the pullback map
\begin{equation}
    \tilde\mu^\ast:\mathcal{O}_{\gl(\mathbf v)^*} \longrightarrow \Gamma(T^*\mathrm{Rep}(\mathbf v,\mathbf w)\times\mathcal Z, \mathcal{O}_{T^*\mathrm{Rep}(\mathbf v,\mathbf w)\times\mathcal Z}) ~.
\end{equation}
Then we obviously have
\begin{align*}
    \mu^{-1}(\mathcal Z)\cong \widetilde\mu^{-1}(0) ~.
\end{align*}

\medskip 

\begin{defn}
A representation $(x,y,\alpha,\beta)\in T^*\mathrm{Rep}(\mathbf v,\mathbf w)$ is called stable if there is no proper $Q_0$-graded linear subspace $S_i\subset V_i$ such that 
\begin{align*}
    \alpha_i(W_i)\subset S_i,\quad x_a(S_{\mathsf t(a)})\subset S_{\mathsf h(a)},\quad y_a(S_{\mathsf {h}(a)})\subset S_{\mathsf {t}(a)}~.
\end{align*}
A representation $(x,y,\alpha,\beta)\in T^*\mathrm{Rep}(\mathbf v,\mathbf w)$ is called costable if there is no proper $Q_0$-graded linear subspace $T_i\subset V_i$ such that 
\begin{align*}
    \beta_i(T_i)=0,\quad x_a(T_{\mathsf t(a)})\subset T_{\mathsf h(a)},\quad y_a(T_{\mathsf {h}(a)})\subset T_{\mathsf {t}(a)} ~.
\end{align*}
\end{defn}
Denote by $T^*\mathrm{Rep}(\mathbf v,\mathbf w)^{\mathrm{st}}$ (resp. $T^*\mathrm{Rep}(\mathbf v,\mathbf w)^{\mathrm{cst}}$) the locus of stable (resp. costable) representations, and for every $\GL(\mathbf v)$-invariant subset $S\subset T^*\mathrm{Rep}(\mathbf v,\mathbf w)$ we set
\begin{align}
    S^{\mathrm{st}}:=S\cap T^*\mathrm{Rep}(\mathbf v,\mathbf w)^{\mathrm{st}}, \quad S^{\mathrm{cst}}:=S\cap T^*\mathrm{Rep}(\mathbf v,\mathbf w)^{\mathrm{cst}}.
\end{align}

\medskip

\begin{defn} \label{def:extendedNakajimaQuiverVariety}
The Nakajima quiver variety ${\mathcal M}(\mathbf v,\mathbf w)$ and its deformation $\widetilde{\mathcal M}(\mathbf v,\mathbf w)$ are defined as the Hamiltonian reductions
\begin{align}
    {\mathcal M}(\mathbf v,\mathbf w):=\mu^{-1}(0)^{\mathrm{st}}\sslash \GL(\mathbf v),\quad \widetilde{\mathcal M}(\mathbf v,\mathbf w):=\mu^{-1}(\mathcal Z)^{\mathrm{st}}\sslash \GL(\mathbf v)~.
\end{align}
\end{defn}
Note that ${\mathcal M}(\mathbf v,\mathbf w)$ is a closed subscheme of $\widetilde{\mathcal M}(\mathbf v,\mathbf w)$. In fact, if $\mu^{-1}(\mathcal Z)^{\mathrm{st}}$ is nonempty, then the natural projection $\mu^{-1}(\mathcal Z)^{\mathrm{st}}\to \mathcal Z$ is smooth and the action of $\GL(\mathbf v)$ on $\mu^{-1}(\mathcal Z)^{\mathrm{st}}$ is free. After passing to the quotient we get a smooth morphism $p: \widetilde{\mathcal M}(\mathbf v,\mathbf w)\to \mathcal Z$, and ${\mathcal M}(\mathbf v,\mathbf w)=p^{-1}(0)$. Since ${\mathcal M}(\mathbf v,\mathbf w)$ is a good quotient, its dimension is the expected one:
\begin{align}
    \dim {\mathcal M}(\mathbf v,\mathbf w)=\mathbf v\cdot(2\mathbf w-\mathsf C\mathbf v)~.
\end{align}

\medskip

\begin{proposition}\label{prop: stable locus}
The following statements are equivalent to each other:
\begin{enumerate}
    \item $\mu^{-1}(0)^{\mathrm{st}}$ is nonempty.
    \item $\mu^{-1}(\mathcal Z)^{\mathrm{st}}$ is nonempty.
    \item There exists a free $\GL(\mathbf v)$-orbit in $\mu^{-1}(\mathcal Z)$.
    \item The map $p: \widetilde{\mathcal{M}}(\mathbf v, \mathbf w)\to \mathcal{Z}$ is surjective.
    \item The moment map $\mu:T^*\mathrm{Rep}(\mathbf v,\mathbf w)\to \mathfrak{gl}(\mathbf v)^*$ is surjective.  
\end{enumerate}
\end{proposition}

\medskip

\begin{proof}
Tautologically, we have $1\Rightarrow 2$. Since $\GL(\mathbf v)$-action on stable locus is free, we have $2\Rightarrow 3$. The implication $3\Rightarrow 4$ is proven in \cite[Proposition 2.2.2]{maulik2019quantum}. Assuming $4$, then ${\mathcal{M}}(\mathbf v, \mathbf w)=p^{-1}(0)$ is nonempty, this implies $1$. 

$1\Rightarrow 5$. Assuming $1$, then $\mu$ is smooth along $\mu^{-1}(0)^{\mathrm{st}}$, so the image of $\mu$ contains an open subscheme of $\mathfrak{gl}(\mathbf v)^*$. Since $\mu$ is equivariant with respect to the $\mathbb C^\times$ action that scales every arrow in the quiver with weight one and scales $\mathfrak{gl}(\mathbf v)^*$ with weight two, the image of $\mu$ is $\mathbb C^\times$-invariant, therefore must be the whole $\mathfrak{gl}(\mathbf v)^*$. This implies $5$.

$5\Rightarrow 2$. Assuming $5$, then for every generic $\lambda\in \mathcal Z$ the preimage $\mu^{-1}(\lambda)$ is nonempty. According to the proof of \cite[Proposition 2.2.2]{maulik2019quantum}, every point in $\mu^{-1}(\lambda)$ for a generic $\lambda$ is stable, therefore $\mu^{-1}(\mathcal Z)^{\mathrm{st}}$ is nonempty. This implies $2$.
\end{proof}

\bigskip 

From now on we assume that the statements in Proposition \ref{prop: stable locus} hold, in particular the quiver varieties ${\mathcal M}(\mathbf v,\mathbf w)$ and $\widetilde{\mathcal M}(\mathbf v,\mathbf w)$ are nonempty. The GIT construction gives a projective morphism
\begin{align}\label{proj-morph-Mtilde-to-Mtilde0}
    \pi: \widetilde{\mathcal M}(\mathbf v,\mathbf w)\to \widetilde{\mathcal M}^0(\mathbf v,\mathbf w)~,
\end{align}
and moreover $p=p_0\circ\pi$. For a point $\lambda\in \mathcal Z$, define ${\mathcal M}_\lambda(\mathbf v,\mathbf w):=p^{-1}(\lambda)$ and ${\mathcal M}^0_\lambda(\mathbf v,\mathbf w):=p_0^{-1}(\lambda)$, then the restriction of $\pi$ to the fibers of $p$ at $\lambda$ gives a projective morphism 
\begin{equation}\label{lambdafiber-proj-morph-M-to-M0} 
\pi_\lambda: {\mathcal M}_\lambda(\mathbf v,\mathbf w)\to {\mathcal M}^0_\lambda(\mathbf v,\mathbf w) ~.
\end{equation}
For a generic $\lambda\in \mathcal Z$, it is known that $\pi_\lambda$ is an isomorphism; moreover, it is also known that for all $\lambda\in \mathcal Z$, $\dim\mathrm{Im}(\pi_\lambda)=\dim {\mathcal M}_\lambda(\mathbf v,\mathbf w)$. The proof of both of these results can be found in \cite[Proposition 3.6]{losevlecture}. As a corollary, we see that
\begin{align}
    \dim {\mathcal M}^0_\lambda(\mathbf v,\mathbf w)\ge \dim {\mathcal M}_\lambda(\mathbf v,\mathbf w)=\mathbf v\cdot(2\mathbf w-\mathsf C\mathbf v)~.
\end{align}
The following is well-known to experts, see \cite[Theorem 1.1, Theorem 1.3]{crawley2001geometry} for a proof.

\medskip 

\begin{proposition}\label{prop: flat moment map}
The following statements are equivalent to each other:
\begin{enumerate}
    \item $\mu:T^*\mathrm{Rep}(\mathbf v,\mathbf w)\to \mathfrak{gl}(\mathbf v)^*$ is flat.
    \item $\dim \mu^{-1}(0)=\mathbf v\cdot(\mathbf v+2\mathbf w-\mathsf C\mathbf v)$. 
    \item $\dim \mathcal M^0(\mathbf v,\mathbf w)=\mathbf v\cdot(2\mathbf w-\mathsf C\mathbf v)$.
    \item $\pi_0: {\mathcal M}(\mathbf v,\mathbf w)\to {\mathcal M}^0(\mathbf v,\mathbf w)$ is surjective. 
    \item $\pi: \widetilde{\mathcal M}(\mathbf v,\mathbf w)\to \widetilde{\mathcal M}^0(\mathbf v,\mathbf w)$ is surjective.
    \item For every decomposition $\mathbf v=\sum_{l=0}^k\mathbf v^{(l)}$ where $k\ge 1$, and $\mathbf v^{(l)}\in \mathbb Z^{Q_0}_{\ge 0}$, $l=0,1,\cdots,k$, and $\mathbf v^{(s)}\neq 0$, $s=1,\cdots,k$, the following inequality holds: 
    \begin{align}\label{key inequality}
        \boxed{ ~~ \mathbf v\cdot(2\mathbf w-\mathsf C\mathbf v)\ge \mathbf v^{(0)}\cdot(2\mathbf w-\mathsf C\mathbf v^{(0)})+\sum_{l=1}^k(2-\mathbf v^{(l)}\cdot\mathsf C\mathbf v^{(l)})~.~}
    \end{align}
\end{enumerate}
\end{proposition}

\begin{remark}
If $\mu:T^*\mathrm{Rep}(\mathbf v,\mathbf w)\to \mathfrak{gl}(\mathbf v)^*$ is flat, then the image of $\mu$ is a $\mathbb C^{\times}$-invariant open subset of $\mathfrak{gl}(\mathbf v)^*$ which contains the origin, where the $\mathbb C^{\times}$-action scales $\mathfrak{gl}(\mathbf v)^*$ with weight $2$, therefore it follows that $\mu$ is surjective. In particular, the flatness of $\mu$ implies that both $\mu^{-1}(0)^{\mathrm{st}}$ and $\mu^{-1}(\mathcal Z)^{\mathrm{st}}$ are nonempty by Proposition \ref{prop: stable locus}.
\end{remark}

\begin{proposition}\label{prop: flat moment map 2}
The equivalent statements in Proposition \ref{prop: flat moment map} imply the following.
\begin{itemize}
    \item[(1)] $\mu^{-1}(\mathcal Z)$ is reduced and irreducible.
    \item[(2)] Both $\widetilde{\mathcal M}^0(\mathbf v,\mathbf w)$ and ${\mathcal M}^0_\lambda(\mathbf v,\mathbf w)$ are reduced and irreducible and normal for all $\lambda\in \mathcal Z$.
    \item[(3)] $\pi: \widetilde{\mathcal M}(\mathbf v,\mathbf w)\to \widetilde{\mathcal M}^0(\mathbf v,\mathbf w)$ and $\pi_\lambda: {\mathcal M}_\lambda(\mathbf v,\mathbf w)\to {\mathcal M}^0_\lambda(\mathbf v,\mathbf w)$ are resolutions of singularities for all $\lambda\in \mathcal Z$.
    \item[(4)] The natural maps
    \begin{align*}
    \mathbb C[\mathcal M^0_\lambda(\mathbf v,\mathbf w)]\longrightarrow \Gamma(\mathcal M_{\lambda}(\mathbf v,\mathbf w),\mathcal O_{\mathcal M_\lambda(\mathbf v,\mathbf w)}),\quad \mathbb C[\widetilde{\mathcal M}^0(\mathbf v,\mathbf w)]\longrightarrow \Gamma(\widetilde{\mathcal M}(\mathbf v,\mathbf w),\mathcal O_{\widetilde{\mathcal M}(\mathbf v,\mathbf w)})~,
\end{align*}
are isomorphisms.
\end{itemize}
\end{proposition}

\medskip 

\begin{proof}
Since $\mu: \mu^{-1}(\mathcal Z)\to \mathcal Z$ is flat, in order to prove that $\mu^{-1}(\mathcal Z)$ is reduced and irreducible, it suffices to prove that $\mu^{-1}(\lambda)$ is reduced and irreducible for generic $\lambda\in \mathcal Z$, which is implied by \cite[Theorem 1.2]{crawley2001geometry}. This proves (1).

Since $\mathbb C[\widetilde{\mathcal M}^0(\mathbf v,\mathbf w)]=\mathbb C[\mu^{-1}(\mathcal Z)]^{\GL(\mathbf v)}$, we see that $\widetilde{\mathcal M}^0(\mathbf v,\mathbf w)$ is reduced and irreducible. The fact that ${\mathcal M}^0_\lambda(\mathbf v,\mathbf w)$ is irreducible for all $\lambda\in \mathcal Z$ is due to \cite[Corollary 1.4]{crawley2002decomposition}; the reducedness of ${\mathcal M}^0_\lambda(\mathbf v,\mathbf w)$ is shown in \cite[Theorem B]{zhou2022reducedness}, and the normality of ${\mathcal M}^0_\lambda(\mathbf v,\mathbf w)$ is proven in \cite[Theorem 1.1]{crawley2003normality}. Note that $p_0: \widetilde{\mathcal M}^0(\mathbf v,\mathbf w)\to \mathcal Z$ is flat since $\mathbb C[\widetilde{\mathcal M}^0(\mathbf v,\mathbf w)]$ is a direct summand of $\mathbb C[\mu^{-1}(\mathcal Z)]$ as a $\mathbb C[\mathcal Z]$-module, then the normality follows from the flatness of $p_0$, and the smoothness of $\mathcal Z$, and the normality of the fibers ${\mathcal M}^0_\lambda(\mathbf v,\mathbf w)$. This proves (2).

$\pi_\lambda: {\mathcal M}_\lambda(\mathbf v,\mathbf w)\to {\mathcal M}^0_\lambda(\mathbf v,\mathbf w)$ is birational by \cite[Corollary 2.16]{zhou2022reducedness}, therefore it is a resolution of singularities. Since $\pi_\lambda$ is an isomorphism for generic $\lambda\in \mathcal Z$, we conclude that $\pi:\widetilde{\mathcal M}(\mathbf v,\mathbf w)\to \widetilde{\mathcal M}^0(\mathbf v,\mathbf w)$ is birational, therefore it is a resolution of singularities. This proves (3). 

(4) follows from the normality of $\widetilde{\mathcal M}^0(\mathbf v,\mathbf w)$ and ${\mathcal M}^0_\lambda(\mathbf v,\mathbf w)$ and (3).
\end{proof}

\medskip

\begin{remark}
In fact $\pi_\lambda: {\mathcal M}_\lambda(\mathbf v,\mathbf w)\to {\mathcal M}^0_\lambda(\mathbf v,\mathbf w)$ \eqref{lambdafiber-proj-morph-M-to-M0} is a symplectic resolution according to \cite[Section 8]{crawley2003normality}. Using \cite[Lemma 8.5]{crawley2003normality}, we see that $\pi_\lambda$ induces an isomorphism  $$\pi_\lambda^{-1}({\mathcal M}^0_\lambda(\mathbf v,\mathbf w)^{\mathrm{sm}}) \cong {\mathcal M}^0_\lambda(\mathbf v,\mathbf w)^{\mathrm{sm}}~,$$ where $X^{\mathrm{sm}}$ denotes the smooth locus of a variety $X$. Since $p_0$ \eqref{p0morphism-Mtilde0-to-Z} is flat, the smooth locus of $\widetilde{\mathcal M}{^0}(\mathbf v,\mathbf w)$ is the union $\bigcup_{\lambda\in \mathcal Z}{\mathcal M}^0_\lambda(\mathbf v,\mathbf w)^{\mathrm{sm}}$, then it follows that $\pi$ induces an isomorphism between $\pi^{-1}(\widetilde{\mathcal M}^0(\mathbf v,\mathbf w)^{\mathrm{sm}})$ and $\widetilde{\mathcal M}^0(\mathbf v,\mathbf w)^{\mathrm{sm}}$. 
\end{remark}

\begin{proposition}\label{prop: flat extended moment map}
The following statements are equivalent to each other:
\begin{enumerate}
    \item $\widetilde\mu:T^*\mathrm{Rep}(\mathbf v,\mathbf w)\times \mathcal Z\to \mathfrak{gl}(\mathbf v)^*$ is flat.
    \item $\dim \widetilde\mu^{-1}(0)=\mathbf v\cdot(\mathbf v+2\mathbf w-\mathsf C\mathbf v)+|Q_0|$.
    \item For every decomposition $\mathbf v=\sum_{l=0}^k\mathbf v^{(l)}$ where $k\ge 1$, and $\mathbf v^{(l)}\in \mathbb Z^{Q_0}_{\ge 0}$, $l=0,1,\cdots,k$, and $\mathbf v^{(s)}\neq 0$ if $s>0$, the following inequality holds:
    \begin{align}\label{key inequality tilde}
        \boxed{~~ \mathbf v\cdot(2\mathbf w-\mathsf C\mathbf v)+2|Q_0|\ge \mathbf v^{(0)}\cdot(2\mathbf w-\mathsf C\mathbf v^{(0)})+\sum_{l=1}^k(2-\mathbf v^{(l)}\cdot\mathsf C\mathbf v^{(l)})+2d~,~}
    \end{align}
    where $d=\dim\{\lambda\in \mathcal Z:\lambda\cdot \mathbf v^{(l)}=0,\, l=1,\cdots,k\}$.
\end{enumerate}
\end{proposition}

\medskip 

\begin{proof}
$\widetilde\mu$ is flat if and only if its fibers are of the expected dimension, which is $$\dim T^*\mathrm{Rep}(\mathbf v,\mathbf w)\times \mathcal Z-\dim \mathfrak{gl}(\mathbf v)^*=\mathbf v\cdot(\mathbf v+2\mathbf w-\mathsf C\mathbf v)+|Q_0|~.$$ Note that $\widetilde\mu$ is $\mathbb C^{\times}$-equivariant, where $\mathbb C^{\times}$ scales $T^*\mathrm{Rep}(\mathbf v,\mathbf w)$, $\mathcal Z$, and $\mathfrak{gl}(\mathbf v)^*$ with weights $1$, $2$, and $2$ respectively. Since the $\mathbb C^{\times}$ contracts $T^*\mathrm{Rep}(\mathbf v,\mathbf w)\times \mathcal Z$ to the origin, therefore $\widetilde\mu$ is flat if and only if $\widetilde\mu^{-1}(0)$ is of the expected dimension. This proves the equivalence between statements 1 and 2.

Next we prove the equivalence between 2 and 3. According to \cite[Theorem 4.4]{crawley2001geometry}, let $\lambda\in \mathcal Z$, then 
\begin{align}
    \dim\mu^{-1}(\lambda)=\mathbf v\cdot(\mathbf v+\mathbf w-\frac{1}{2}\mathsf C\mathbf v)+m ~,
\end{align}
where $$m=\mathrm{max}\left(\mathbf v^{(0)}\cdot(\mathbf w-\frac{1}{2}\mathsf C\mathbf v^{(0)})+\sum_{l={{1}}}^k(1-\frac{1}{2}\mathbf v^{(l)}\cdot\mathsf C\mathbf v^{(l)})\right)~,$$
$\mathbf v=\sum_{l=0}^k\mathbf v^{(l)}$ is a decomposition such that $k\ge 0$, and $\mathbf v^{(l)}\in \mathbb Z^{Q_0}_{\ge 0}$, $l=0,\cdots,k$, and $\mathbf v^{(s)}\in \lambda^{\perp}\setminus \{0\}$ if $s>0$. Here $\lambda^{\perp}:=\{z\in \mathcal Z:\lambda\cdot z=0\}$. Consider the hyperplane arrangement on $\mathcal Z$ given by the hyperplanes $\{\mathbf u^{\perp}:\mathbf u\in \mathbb Z^{Q_0}_{\ge 0},\mathbf v-\mathbf u\in \mathbb Z^{Q_0}_{\ge 0}\}$; this gives a stratification $\mathcal Z=\bigsqcup_{\sigma}\mathcal Z_{\sigma}$. For every strata $\mathcal Z_{\sigma}$, there is a unified choice of decomposition $\mathbf v=\sum_{l=0}^k\mathbf v^{(l)}$ that maximizes $\mathbf v^{(0)}\cdot\mathbf w-1+\sum_{l=0}^k(1-\frac{1}{2}\mathbf v^{(l)}\cdot\mathsf C\mathbf v^{(l)})$ for every point $\lambda$ in $\mathcal Z_{\sigma}$, and we denote the maximal value by $m_{\sigma}$. Then the dimension of $\widetilde\mu^{-1}(0)\cong\mu^{-1}(\mathcal Z)$ can be written as
\begin{align}\label{dim of extended shell}
    \dim\widetilde\mu^{-1}(0)=\underset{\sigma}{\max}\left(\mathbf v\cdot(\mathbf v+\mathbf w-\frac{1}{2}\mathsf C\mathbf v)+m_{\sigma}+\dim\mathcal Z_{\sigma}\right)~.
\end{align}
Evidently $\underset{\sigma}{\max}(m_{\sigma}+\dim\mathcal Z_{\sigma})$ is the maximal value of $$\mathbf v^{(0)}\cdot\mathbf w-1+\sum_{l=0}^k(1-\frac{1}{2}\mathbf v^{(l)}\cdot\mathsf C\mathbf v^{(l)})+\dim\left(\bigcap_{l=1}^k(\mathbf v^{(l)})^\perp\right)~,$$
among all decompositions $\mathbf v=\sum_{l=0}^k\mathbf v^{(l)}$ such that $k\ge 0$, and $\mathbf v^{(l)}\in \mathbb Z^{Q_0}_{\ge 0}$, $l=0,\cdots,k$, and $\mathbf v^{(s)}\neq 0$ if $s>0$. Since we always have $$\dim \widetilde\mu^{-1}(0)\ge \dim T^*\mathrm{Rep}(\mathbf v,\mathbf w)\times \mathcal Z-\dim \mathfrak{gl}(\mathbf v)^*=\mathbf v\cdot(\mathbf v+2\mathbf w-\mathsf C\mathbf v)+|Q_0|~,$$ the equality holds if and only if $\dim \widetilde\mu^{-1}(0)\le\mathbf v\cdot(\mathbf v+2\mathbf w-\mathsf C\mathbf v)+|Q_0|$, and by the dimension formula \eqref{dim of extended shell} such inequality is equivalent to that the following inequality
\begin{gather} \label{key inequality tilde V2}
\begin{split}
    & \mathbf v\cdot(\mathbf v+2\mathbf w-\mathsf C\mathbf v)+|Q_0| \ge \\ & \qquad \mathbf v\cdot(\mathbf v+\mathbf w-\frac{1}{2}\mathsf C\mathbf v)+\mathbf v^{(0)}\cdot\mathbf w-1+\sum_{l=0}^k(1-\frac{1}{2}\mathbf v^{(l)}\cdot\mathsf C\mathbf v^{(l)})+\dim\left(\bigcap_{l=1}^k(\mathbf v^{(l)})^\perp\right)
\end{split}
\end{gather}
holds for all allowed decompositions $\mathbf v=\sum_{l=0}^k\mathbf v^{(l)}$. A simple computation shows that this is equivalent to  statement (3).
\end{proof}

\medskip 

\begin{remark}\label{rmk: mu flat implies tilde mu flat}
It is easy to see that the inequality \eqref{key inequality} implies the inequality \eqref{key inequality tilde}, so the flatness of $\mu$ implies the flatness of $\widetilde\mu$. This can also be derived from Lemma \ref{lem: jet flatness for homogeneous map} (2) by taking $g=\mathrm{pr}_{\mathcal Z}$ and $$V=T^*\mathrm{Rep}(\mathbf v,\mathbf w)\times\mathcal Z~, \qquad U=\gl(\mathbf v)^*~, \qquad f=\mu\circ\mathrm{pr}_{T^*\mathrm{Rep}(\mathbf v,\mathbf w)}~.$$ 
\end{remark}

\subsection{Simple locus}
We define the simple locus $T^*\mathrm{Rep}(\mathbf v,\mathbf w)^{\mathrm{reg}}$ to be the intersection of stable and costable loci, i.e.
\begin{align}
    T^*\mathrm{Rep}(\mathbf v,\mathbf w)^{\mathrm{reg}}:=T^*\mathrm{Rep}(\mathbf v,\mathbf w)^{\mathrm{st}}\bigcap T^*\mathrm{Rep}(\mathbf v,\mathbf w)^{\mathrm{cst}} ~.
\end{align}

\medskip 

\begin{remark}\label{rmk: CB trick}
The terminology \textit{simple locus} is derived from the following fact. Consider the unframed quiver $Q^+$ associated to the framed quiver $Q$ defined as $Q^+_0:=Q_0\sqcup \{\infty\}$ with the adjacency matrix $\mathsf Q^+$ defined as
\begin{align*}
    \mathsf Q^+_{ij}=\begin{cases}
        \mathsf Q_{ij}, & i,j\neq \infty ~,\\
         w_j, & i=\infty, j\neq \infty ~,\\
        0, & j=\infty ~.
    \end{cases}
\end{align*}
Let $ \mathbf v^+\in \mathbb Z^{Q^+_0}$ be the following
\begin{align*}
     v^+_i:=\begin{cases}
         v_i, & i\neq \infty ~,\\
        1, & i=\infty ~.
    \end{cases}
\end{align*}
There is a canonical isomorphism $T^*\mathrm{Rep}(\mathbf v,\mathbf w)\cong \mathrm{Rep}(\overline{Q}^+,\mathbf v^+)$, where the latter is the space of representations of the doubled quiver $\overline{Q}^+$ with dimension vector $\mathbf v^+$ (see the end of Section 1 in \cite{crawley2001geometry}). Then a representation in $T^*\mathrm{Rep}(\mathbf v,\mathbf w)$ is in the simple locus if and only if its corresponding $\overline{Q}^+$-representation is simple.
\end{remark}

For every $\GL(\mathbf v)$-invariant subset $S\subset T^*\mathrm{Rep}(\mathbf v,\mathbf w)$ we set
\begin{align}
    S^{\mathrm{reg}}:=S\cap T^*\mathrm{Rep}(\mathbf v,\mathbf w)^{\mathrm{reg}} ~.
\end{align}
$\mu^{-1}(0)^{\mathrm{reg}}$ is an open subset (possibly empty) in $\mu^{-1}(0)$, and the $\GL(\mathbf v)$-action on $\mu^{-1}(0)^{\mathrm{reg}}$ is free. Moreover, every $\GL(\mathbf v)$-orbit in $\mu^{-1}(0)^{\mathrm{reg}}$ is closed in $\mu^{-1}(0)$, thus $\mu^{-1}(0)^{\mathrm{reg}}$ descends to an open subset $\mathcal M^0(\mathbf v,\mathbf w)^{\mathrm{reg}}$ along the quotient map $\mu^{-1}(0)\to \mu^{-1}(0)\sslash \GL(\mathbf v)$. Since $\mu^{-1}(0)^{\mathrm{reg}}\subset \mu^{-1}(0)^{\mathrm{st}}$, $\mu^{-1}(0)^{\mathrm{reg}}$ also descends to an open subset $\mathcal M(\mathbf v,\mathbf w)^{\mathrm{reg}}$ along the quotient map $\mu^{-1}(0)^{\mathrm{st}}\to \mu^{-1}(0)^{\mathrm{st}}\sslash \GL(\mathbf v)$. Nakajima shows that the natural projection
\begin{align*}
    \pi|_{\mathcal M(\mathbf v,\mathbf w)^{\mathrm{reg}}}: \mathcal M(\mathbf v,\mathbf w)^{\mathrm{reg}}\to \mathcal M^0(\mathbf v,\mathbf w)^{\mathrm{reg}}
\end{align*}
is an isomorphism \cite[Proposition 3.24]{nakajima1998quiver}.

\medskip

\begin{proposition}\label{prop: simple locus}
If $\mu^{-1}(0)^{\mathrm{reg}}$ is nonempty, then
\begin{itemize}
    \item[(1)] $\mathbf w-\mathsf C\mathbf v\in \mathbb Z^{Q_0}_{\ge 0}$, i.e.\  $ w_i-\sum_j\mathsf C_{ij} v_j\ge 0$ for all $i\in Q_0$.
    \item[(2)] $\mu^{-1}(0)$ is reduced and irreducible, and $\dim \mu^{-1}(0)=\mathbf v\cdot(\mathbf v+2\mathbf w-\mathsf C\mathbf v)$. In particular $\mu$ is flat.
\end{itemize}
\end{proposition}
\begin{proof}
The first statement is in \cite[Lemma 4.7]{nakajima1998quiver}, the second statement is a special case of \cite[Theorem 1.2 and Theorem 1.1]{crawley2001geometry}.
\end{proof}
It follows from Proposition \ref{prop: simple locus} and Proposition \ref{prop: flat moment map 2} that if $\mu^{-1}(0)^{\mathrm{reg}}$ is nonempty, then $\mu^{-1}(\mathcal Z)$ is reduced and irreducible. 

\medskip 

Crawley-Boevey proved the following criterion for the non-emptiness of $\mu^{-1}(0)^{\mathrm{reg}}$ \cite[Theorem 1.2]{crawley2001geometry}.

\medskip

\begin{proposition}\label{prop: simple locus criterion}
$\mu^{-1}(0)^{\mathrm{reg}}$ is nonempty if and only if for every decomposition $\mathbf v=\sum_{l=0}^k\mathbf v^{(l)}$ where $k\ge 1$, and $\mathbf v^{(l)}\in \mathbb Z^{Q_0}_{\ge 0}$, $l=0,1,\cdots,k$, and $\mathbf v^{(s)}\neq 0$ if $s>0$, the inequality \eqref{key inequality} is strict, i.e.
    \begin{align}\label{key strict inequality}
        \boxed{ ~~ \mathbf v\cdot(2\mathbf w-\mathsf C\mathbf v)> \mathbf v^{(0)}\cdot(2\mathbf w-\mathsf C\mathbf v^{(0)})+\sum_{l=1}^k(2-\mathbf v^{(l)}\cdot\mathsf C\mathbf v^{(l)})~.~}
    \end{align}
\end{proposition}

\medskip

\begin{remark}\label{rmk: KM quiver simple locus}
Assume that $Q$ is a Kac-Moody quiver (i.e.\  there is no edge loop), and also assume that $\mathbf w\cdot \lambda\ge 2$ for every imaginary root $\lambda$ (cf. definition in \cite[p.262]{crawley2001geometry}), then the computation in \cite[Proposition 10.5]{nakajima1998quiver} shows that \eqref{key strict inequality} holds if and only if $\mathbf w-\mathsf C\mathbf v\in \mathbb Z^{Q_0}_{\ge 0}$.
\end{remark}

\subsection{Quiver super-varieties}\label{sec:QSV}
Let $Q$ be a quiver and let $\mathbf v,\mathbf w$ be the dimension vectors of gauge nodes and framings respectively. Let $U=\bigoplus_{i\in Q_0} U_i$ be a $Q_0$-graded vector space with dimension vector $\mathbf u$, then we consider an enrichment of the original quiver such that the framing vector space $W$ is replaced by the super vector space $W\oplus \Pi U$. Define the superspace
\begin{align}
    \mathrm{Rep}(\mathbf v,\mathbf w|\mathbf u):=\mathrm{Rep}(\mathbf v,\mathbf w)\oplus\bigoplus_{i\in Q_0}\Pi\Hom(U_i,V_i)~,
\end{align}
which is a representation of $\GL(\mathbf v)$. Its cotangent bundle is identified with the following 
\begin{align}
    T^*\mathrm{Rep}(\mathbf v,\mathbf w|\mathbf u)=T^*\mathrm{Rep}(\mathbf v,\mathbf w)\oplus \bigoplus_{i\in Q_0}\Pi\Hom(U_i,V_i)\bigoplus_{i\in Q_0}\Pi\Hom(V_i,U_i)~.
\end{align}
We introduce the notation for the elements in the odd summands in $T^*\mathrm{Rep}(\mathbf v,\mathbf w|\mathbf u)$ as follows
\begin{center}
\begin{tabular}{ | c | c |} 
  \hline
  $\Hom(U_i,V_i)$ & $\Hom(V_i,U_i)$  \\  
 \hline
 $\psi_i$ & $\phi_i$  \\ [0.5ex] 
 \hline
\end{tabular}
\end{center}
The moment map for the Hamiltonian $\GL(\mathbf v)$-action on $T^*\mathrm{Rep}(\mathbf v,\mathbf w|\mathbf u)$ is
\begin{gather}\label{super moment map}
\begin{split}
    & \mu_{\mathsf s}:T^*\mathrm{Rep}(\mathbf v,\mathbf w|\mathbf u)\to \mathfrak{gl}(\mathbf v)^* ~, \\
    & \mu_{\mathsf s}(x,y,\alpha,\beta,\psi,\phi)=\sum_{a\in Q_1}[x_a,y_a]+\sum_{i\in Q_0}\alpha_i\beta_i+\sum_{i\in Q_0}\psi_i\phi_i ~,
\end{split}
\end{gather}
with pullback map
\begin{equation}
    \mu_{\mathsf s}^\ast : \mathcal{O}_{\mathfrak{gl}(\mathbf v)^*} \to \Gamma(T^*\mathrm{Rep}(\mathbf v,\mathbf w|\mathbf u), \mathcal{O}_{T^*\mathrm{Rep}(\mathbf v,\mathbf w|\mathbf u)}) ~.
\end{equation}
Similarly, we define the extended moment map 
\begin{equation} \label{eqdef:extendedsupermomentmap}
\widetilde\mu_{\mathsf s}: \mathfrak{R}\to \mathfrak{gl}(\mathbf v)^* ~,
\end{equation} 
where we introduce for future reference the shorthand notation\footnote{Remark that $\mathfrak{R}\big|_{\mathbf u=0}=T^*\mathrm{Rep}(\mathbf v,\mathbf w)\times  \mathcal Z$ is the extended representation space relevant in Section \ref{sec:GeomQuiverVarieties-intro}.} 
\begin{equation} \label{def:prequotient}
\mathfrak{R}:=T^*\mathrm{Rep}(\mathbf v,\mathbf w | \mathbf u)\times  \mathcal Z ~,
\end{equation}
as
\begin{align}\label{extended super moment map}
    \widetilde\mu_{\mathsf s}((x,y,\alpha,\beta,\psi,\phi),z)=\mu_{\mathsf s}(x,y,\alpha,\beta,\psi,\phi)+z ~,
\end{align}
so $\widetilde\mu_{\mathsf s}^{-1}(0)\cong \mu_{\mathsf s}^{-1}(\mathcal Z)$. Given the extended moment map \ref{eqdef:extendedsupermomentmap}, we also have the pullback map
\begin{equation}
    \tilde\mu_{\mathsf s}^\ast : \mathcal{O}_{\mathfrak{gl}(\mathbf v)^*} \to \Gamma( \mathfrak{R} , \mathcal{O}_{ \mathfrak{R}}) ~.
\end{equation}
Note that the underlying bosonic scheme of $\mu^{-1}_{\mathsf s}(0)$ is $\mu^{-1}(0)$, and the underlying bosonic scheme of $\mu^{-1}_{\mathsf s}(\mathcal Z)$ is $\mu^{-1}(\mathcal Z)$. Since a super-scheme has the same topology as its underlying bosonic scheme, we define the stable locus $\mu^{-1}_{\mathsf s}(0)^{\mathrm{st}}$ to be the open subset $\mu^{-1}(0)^{\mathrm{st}}$ endowed with the super-scheme structure of $\mu^{-1}_{\mathsf s}(0)$. Similarly we define $\mu^{-1}_{\mathsf s}(\mathcal Z)^{\mathrm{st}}$ to be the open subset $\mu^{-1}(\mathcal Z)^{\mathrm{st}}$ endowed with the super-scheme structure of $\mu^{-1}_{\mathsf s}(\mathcal Z)$.

\medskip 

\begin{defn}
The affine quiver super-variety ${\mathcal M}^0(\mathbf v,\mathbf w|\mathbf u)$ and its deformation $\widetilde{\mathcal M}^0(\mathbf v,\mathbf w|\mathbf u)$ are defined as the Hamiltonian reductions
\begin{align}
    {\mathcal M}^0(\mathbf v,\mathbf w|\mathbf u):=\mu^{-1}_{\mathsf s}(0)\sslash \GL(\mathbf v),\quad \widetilde{\mathcal M}^0(\mathbf v,\mathbf w|\mathbf u):=\mu^{-1}_{\mathsf s}(\mathcal Z)\sslash \GL(\mathbf v)~.
\end{align}
The Nakajima quiver super-variety ${\mathcal M}(\mathbf v,\mathbf w|\mathbf u)$ and its deformation $\widetilde{\mathcal M}(\mathbf v,\mathbf w|\mathbf u)$ are defined as the Hamiltonian reductions
\begin{align}\label{eqdef:ExtNakajimaQuiverSuperVariety}
    {\mathcal M}(\mathbf v,\mathbf w|\mathbf u):=\mu^{-1}_{\mathsf s}(0)^{\mathrm{st}}\sslash \GL(\mathbf v),\quad \widetilde{\mathcal M}(\mathbf v,\mathbf w|\mathbf u):=\mu^{-1}_{\mathsf s}(\mathcal Z)^{\mathrm{st}}\sslash \GL(\mathbf v)~.
\end{align}
\end{defn}

Note that the restriction of $\mu_{\mathsf s}$ to the bosonic scheme $T^*\mathrm{Rep}(\mathbf v,\mathbf w)$ agrees with the bosonic moment map \eqref{moment map}. The restriction of bosonic moment map $\mu$ to the stable locus is a smooth morphism between smooth varieties, and so it follows that $\mu_{\mathsf s}|_{T^*\mathrm{Rep}(\mathbf v,\mathbf w|\mathbf u)^{\mathrm{st}}}$ is smooth as a morphism between super-schemes. In particular $\mu^{-1}_{\mathsf s}(0)^{\mathrm{st}}$ and $\mu^{-1}_{\mathsf s}(\mathcal Z)^{\mathrm{st}}$ are smooth super-schemes. Moreover, the action of $\GL(\mathbf v)$ on the stable locus is free, thus the quotient maps 
\begin{align}\label{quotientmaps-super}
    q_{\mathsf s}: \mu^{-1}_{\mathsf s}(0)^{\mathrm{st}}\longrightarrow \mu^{-1}_{\mathsf s}(0)^{\mathrm{st}}\sslash \GL(\mathbf v),\quad \widetilde q_{\mathsf s}: \mu^{-1}_{\mathsf s}(\mathcal Z)^{\mathrm{st}}\longrightarrow \mu^{-1}_{\mathsf s}(\mathcal Z)^{\mathrm{st}}\sslash \GL(\mathbf v)~,
\end{align}
are principal $\GL(\mathbf v)$-bundles, and ${\mathcal M}(\mathbf v,\mathbf w|\mathbf u)$ and $\widetilde{\mathcal M}(\mathbf v,\mathbf w|\mathbf u)$ are smooth super-schemes.

\section{Jet-Flatness}\label{sec:JetFlatness}

Throughout this section, all the schemes are assumed to be finite type over complex numbers. For a scheme $X$, denote $J_nX$ the $n$-th jet scheme of $X$, i.e.\  
\begin{align*}
    J_nX=\Hom(\Spec\mathbb C[t]/(t^{n+1}),X)~.
\end{align*}
\begin{defn}
We say that a morphism of schemes $f:X\to Y$ is $J_n$-flat if $J_nf: J_nX\to J_nY$ is flat. We say $f$ is jet-flat if it is $J_n$-flat for all $n$.
\end{defn}

Clearly, smooth morphisms are jet-flat. In fact, a surjective morphism $f:X\to Y$ is smooth if and only if $J_nf$ is surjective for all $n$. The following lemma is elementary.

\medskip

\begin{lem}
Let $f:X\to Y$ and $g:Y\to Z$ be morphisms between schemes. Then the following statements are true.
\begin{itemize}
    \item[(1)] If both $f$ and $g$ are $J_n$-flat morphisms, then $g\circ f:X\to Z$ is a $J_n$-flat morphism. 
    \item[(2)] Let $W\to Y$ be a morphism, and denote by $f_W:X\times_Y W\to W$ the base change of $f$ to W, then $f_W$ is $J_n$-flat if $f$ is $J_n$-flat. 
    \item[(3)] Assume $f$ is smooth and surjective, then $g$ is $J_n$-flat if and only if $g\circ f$ is $J_n$-flat. 
\end{itemize}
\end{lem}

\medskip

\begin{lem}
Let $X$ and $Y$ be smooth schemes and let $f:X\to Y$ be a morphism. If $f$ is $J_n$-flat for some $n\in \mathbb Z_{\ge 0}$, then $f$ is $J_m$-flat for all $m\le n$.
\end{lem}

\begin{proof}
Consider the commutative diagram
\begin{equation*}
\begin{tikzcd}
    J_nX \ar[r,"J_nf"] \ar[d,"\phi_n"] & J_nY \ar[d,"\phi'_n"] \\
    J_{n-1}X  \ar[r,"J_{n-1}f"] & J_{n-1}Y
\end{tikzcd}
\end{equation*}
Here $\phi_n:J_nX\to J_{n-1}X$ and $\phi'_n:J_nY\to J_{n-1}Y$ are the natural projections which are induced from the canonical embedding $\Spec\mathbb C[t]/(t^{n})\hookrightarrow \Spec\mathbb C[t]/(t^{n+1})$. Since $X$ and $Y$ are smooth, $\phi_n$ and $\phi'_n$ are smooth and surjective. Thus the flatness of $J_nf$ implies the flatness of $J_{n-1}f$. And by induction, $f$ is $J_m$-flat for all $m\le n$.
\end{proof}

\medskip 

\begin{defn}
    Let $X$ be a normal algebraic variety over $\mathbb{C}$ and $\pi: \tilde{X} \to X$ a resolution of singularities, where $\tilde{X}$ is a smooth variety and $\pi$ is a proper and birational morphism, which is an isomorphism over the smooth locus of $X$. A point $x \in X$ is called a rational singularity if higher direct image sheaf $R^i \pi_* \mathcal{O}_{\tilde{X}} = 0,~ \forall  i>0$ in a neighborhood of $x$, or equivalently $H^i(\tilde{U}, \mathcal{O}_{\tilde{X}}) = 0,~\forall i>0$ for small neighborhoods $U \subset X$ of $x$, with preimage $\tilde{U}\subset \tilde{X}$ or $U$.
\end{defn}

Rational singularities are therefore cohomologically mild: the singularity behaves like a smooth point from the point of view of the structure sheaf, the pushforward of $\mathcal{O}_{\tilde{X}}$ is $\mathcal{O}_X$ and higher cohomology vanishes. The rationality of singularities is necessary in what follows towards the proof of jet-flatness.

\medskip

\begin{lem}\label{lem: jet flatness for homogeneous map}
Let $V$ and $U$ be vector spaces, and let $f:V\to U$ be a homogeneous morphism of degree $d$, i.e.\ $f$ is equivariant under the $\mathbb C^{\times}$-actions which scale $V$ and $U$ with weights $1$ and $d$ respectively. Then the following statements are true.
\begin{itemize}
    \item[(1)] If $f$ is flat, and $f^{-1}(0)$ is reduced, irreducible and has rational singularities, then $f$ is jet-flat.
    \item[(2)] If $f$ is $J_n$-flat, and $g:V\to U$ is a homogeneous morphism of degree $d'<d$, then $h=f+g:V\to U$ is also $J_n$-flat, and $\dim J_nh^{-1}(0)=\dim J_nf^{-1}(0)$. Assume moreover that $J_nf^{-1}(0)$ is reduced and irreducible, then $J_nh^{-1}(0)$ is reduced and irreducible.
\end{itemize}
\end{lem}

\medskip

\begin{proof}
To prove (1), we notice that $J_nf:J_nV\to J_nU$ is a homogeneous morphism between vector spaces of degree $d$, so it suffices to show that the dimension of $(J_nf)^{-1}(0)\cong J_n(f^{-1}(0))$ is less than or equal to $(n+1)(\dim V-\dim U)$ for all $n\ge 0$. The case when $n=0$ follows from the flatness of $f$. In particular, the flatness of $f$ implies $f^{-1}(0)$ is a locally complete intersection and, by the assumptions stated in point (1), $f^{-1}(0)$ is reduced and irreducible and has rational singularities. Then \cite[Theorem 0.1]{mustata2001jet}, together with \cite[Proposition 1.4]{mustata2001jet}, imply that $\dim J_n(f^{-1}(0))\le (n+1)(\dim V-\dim U)$. This proves (1).

\emph{(Deformation argument)} To prove (2), it suffices to prove the $n=0$ case, since we can replace $V$ by $J_nV$ and $U$ by $J_nU$ respectively. Consider the morphism $F:V\times \mathbb A^1\to U\times \mathbb A^1$ such that $$F(x,t)=(f(x)+t^{d-d'}g(x),t)~.$$ We claim that $F$ is flat. In fact, if we let $\mathbb C^{\times}$ scale $V$, $U$, and $\mathbb A^1$ with weights $1$, $d$, and $1$ respectively, then $F$ is $\mathbb C^{\times}$-equivariant. Notice $\mathbb C^{\times}$-action on $V\times \mathbb A^1$ contracts the whole space to the origin, so it is enough to show that $F$ is flat in an open neighborhood of $(0,0)\in V\times \mathbb A^1$. Equivalently, we need to show that $\dim_{(0,0)}F^{-1}(0,0)\le \dim V-\dim U$. Since $F^{-1}(0,0)=f^{-1}(0)$, and the flatness of $f$ implies that $f^{-1}(0)$ is of pure dimension $\dim V-\dim U$, this proves the claim. Then it follows that the specialization $F(\cdot,1):V\to U$ is flat, and $\dim F^{-1}(0,1)=\dim F^{-1}(0,0)$. Since $F(\cdot,1)=h$, this proves the first statement of (2).

Assume moreover that $f^{-1}(0)$ is reduced and irreducible, then we claim that $F^{-1}(s,t)$ is reduced and irreducible for all $(s,t)\in U\times \mathbb A^1$. Notice that the smooth locus of $F$ is $\mathbb C^{\times}$-invariant, thus $F^{-1}(s,t)$ contains a smooth point. By the flatness of $F$ and existence of a smooth point in $F^{-1}(s,t)$, $F^{-1}(s,t)$ is a locally complete intersection. Hence, to prove the second statement in (2), we only need to show that $F^{-1}(s,t)$ is irreducible;  the reducedness then follows from the existence of a smooth point. Suppose that $F^{-1}(s,t)$ is not irreducible, and we take two irreducible components $F^{-1}(s,t)_k$, with $k=0,1$. Let $\ell\subset U\times \mathbb A^1$ be the line $\{(\lambda s,\lambda t):\lambda\in \mathbb C\}$, and define $F^{-1}(\ell)_k$ to be the closure of $\mathbb C^\times\times  F^{-1}(s,t)_k$ in $F^{-1}(\ell)$, where $\mathbb C^\times\times  F^{-1}(s,t)_k$ is the orbit of $F^{-1}(s,t)_k$ under the $\mathbb C^\times$-action. $F^{-1}(\ell)_{k=0,1}$ are two irreducible components of $F^{-1}(\ell)$. By $\mathbb C^\times$-invariance, $F^{-1}(\ell)_k\cap F^{-1}(0,0)$ is non-empty for both $k=0,1$. Moreover, by the semi-continuity of the dimension of fibers, we have $$\dim F^{-1}(\ell)_k\cap F^{-1}(0,0)\ge \dim F^{-1}(s,t)_k=\dim F^{-1}(s,t)=\dim F^{-1}(0,0)~,$$ but $F^{-1}(0,0)$ is irreducible, this implies that $F^{-1}(\ell)_k\cap F^{-1}(0,0)=F^{-1}(0,0)$. By the flatness of $F$ and the reducedness of $F^{-1}(0,0)$, there exists $x\in F^{-1}(0,0)$ such that $F^{-1}(\ell)$ is smooth at $x$. However, the fact that $x$ belongs to two irreducible components $F^{-1}(\ell)_0$ and $F^{-1}(\ell)_1$ contradicts with the smoothness at $x$. Thus $F^{-1}(s,t)$ is irreducible. This finishes the proof of (2).
\end{proof}

\subsection{Jet-flatness of moment maps}
Let $G$ be a reductive group and let $R$ be a linear representation of $G$. Consider the symplectic representation $T^*R=R\oplus R^*$, then there is a moment map $\mu: T^*R\to \mathfrak{g}^*$, where $\mathfrak{g}$ is the Lie algebra of $G$. Let $\mathfrak Z:=[\mathfrak{g},\mathfrak{g}]^\perp = \mathfrak{g}/[\mathfrak{g},\mathfrak{g}]$. Then $\mathfrak Z$ is a linear subspace of $\mathfrak{g}^*$, i.e.\  there is a canonical embedding $i:\mathfrak Z\hookrightarrow\mathfrak{g}^*$. We define the extended moment map 
\begin{gather} 
\begin{split}
& ~~~\,\widetilde \mu :  T^*R\times \mathfrak Z\to \mathfrak{g}^* \\
& \widetilde \mu(x,z):=\mu(x)+i(z)~.
\end{split}
\end{gather} 

\medskip

\begin{proposition}\label{prop: various jet flatness}
Consider the following four statements.
\begin{itemize}
    \item[$(P_1)$] $\mu$ is flat, and $\mu^{-1}(0)$ is reduced and irreducible and has rational singularities \footnote{This is also called the CIFR condition in the literature, see \cite{herbig2021does}.}.
    \item[$(P_2)$] $\mu$ is jet-flat.
    \item[$(\widetilde P_1)$] $\widetilde\mu$ is flat, and $\widetilde\mu^{-1}(0)$ is reduced and irreducible and has rational singularities.
    \item[$(\widetilde P_2)$] $\widetilde\mu$ is jet-flat.
\end{itemize}
Then we have the implications:
\begin{equation*}
\begin{tikzcd}
    (P_1)\ar[r,Rightarrow] \ar[d,Rightarrow] & (P_2)\ar[d,Rightarrow]\\
    (\widetilde P_1)\ar[r,Rightarrow] & (\widetilde P_2)
\end{tikzcd}
\end{equation*}
Moreover, if $(R,G)$ comes from a quiver representation, i.e.\ $(R,G)=(\mathrm{Rep}(\mathbf v,\mathbf w),\GL(\mathbf v))$ for some quiver $Q$, then we also have the implication
\begin{equation*}
\begin{tikzcd}
    (P_2)\ar[r,Rightarrow] & (\widetilde P_1)~.
\end{tikzcd}
\end{equation*}
\end{proposition}

\medskip

\begin{proof}
The implication $(P_1)\Longrightarrow (P_2)$ follows from Lemma \ref{lem: jet flatness for homogeneous map}, by letting $V=T^*R$ and $U=\mathfrak g^*$ and $f=\mu$. Similarly the implication $(P_2)\Longrightarrow (\widetilde P_2)$ follows from Lemma \ref{lem: jet flatness for homogeneous map}, by letting $V=T^*R\times \mathfrak Z$ and $U=\mathfrak g^*$ and $f=\mu\circ \mathrm{pr}_{T^*R}$ and $g=i\circ \mathrm{pr}_{\mathfrak Z}$. 

The implication $(\widetilde P_1)\Longrightarrow (\widetilde P_2)$ follows from an argument similar to that of the first statement in Lemma \ref{lem: jet flatness for homogeneous map}. Namely, if we set the $\mathbb C^{\times}$-weight of $T^*R$ to be $1$, and set the $\mathbb C^{\times}$-weights of $\mathfrak g^*$ and $\mathfrak Z$ to be $2$, then $J_n\widetilde\mu:J_n(T^*R\times \mathfrak Z)\to J_n\mathfrak g^*$ is $\mathbb C^{\times}$-equivariant, such that $\mathbb C^{\times}$ contracts $J_n(T^*R\times \mathfrak Z)$ to the origin, so it suffices to show that the dimension of $(J_n\widetilde\mu)^{-1}(0)\cong J_n(\widetilde\mu^{-1}(0))$ is less than or equal to $(n+1)(\dim T^*R+\dim \mathfrak Z-\dim \mathfrak g^*)$ for all $n\ge 0$. Since $\widetilde\mu$ is flat, $\widetilde\mu^{-1}(0)$ is a locally complete intersection of dimension $\dim T^*R+\dim \mathfrak Z-\dim \mathfrak g^*$; moreover, since $\widetilde\mu^{-1}(0)$ is reduced and irreducible and has rational singularities, \cite[Theorem 0.1]{mustata2001jet} together with \cite[Proposition 1.4]{mustata2001jet} implies that $\dim J_n(\widetilde\mu^{-1}(0))\le (n+1)(\dim T^*R+\dim \mathfrak Z-\dim \mathfrak g^*)$. This proves the implication $(\widetilde P_1)\Longrightarrow (\widetilde P_2)$.

Let us prove $(P_1)\Longrightarrow (\widetilde P_1)$. Assume that $(P_1)$ holds. Then $\widetilde \mu$ is jet-flat by what we have shown above, thus it suffices to show that
\begin{itemize}
    \item $(J_n\widetilde \mu)^{-1}(0)$ is irreducible and contains a smooth point for all $n$.
\end{itemize}
We prove the above assertion using the deformation argument similar to that of Lemma \ref{lem: jet flatness for homogeneous map}. Define $F_n:J_nT^*R\times J_n\mathfrak Z\times\mathbb A^1 \to J_n\mathfrak g^*\times \mathbb A^1$ by
\begin{align*}
    F_n(x,z,t)=(J_n\mu(x)+t\cdot J_ni(z),t)~.
\end{align*}
Then $F_n(\cdot,0)=J_n\mu\circ\mathrm{pr}_{J_nT^*R}$ and $F_n(\cdot,1)=J_n\widetilde\mu$. We assign the $\mathbb C^{\times}$-weights to $T^*R,\mathfrak Z,\mathbb A^1,\mathfrak g^*$ as follows:

\medskip 

\begin{center}
\begin{tabular}{| c | c | c | c | c |} 
 \hline
& $T^*R$ & $\mathfrak Z$ & $\mathbb A^1$ & $\mathfrak g^*$  \\  
 \hline
$\mathbb C^{\times}$-weights & $1$ & {$2$} & $1$ & $2$  \\ 
 \hline
\end{tabular}
\end{center}

\medskip 

Then $F_n$ is $\mathbb C^{\times}$-equivariant, and $\mathbb C^{\times}$ contracts $J_nT^*R\times J_n\mathfrak Z\times\mathbb A^1$ to the origin. To show that $F_n$ is flat, it suffices to demonstrate that the dimension condition $\dim F_n^{-1}(0,0)\le (n+1)(\dim T^*R+\dim \mathfrak Z-\dim \mathfrak g^*)$ is satisfied. This bound follows from the isomorphism $F_n^{-1}(0,0)\cong J_n\mu^{-1}(0)\times J_n\mathfrak Z$ and \cite[Proposition 1.4]{mustata2001jet}, and proves the flatness of $F_n$. Next, we notice that $F_n^{-1}(0,0)$ is irreducible and contains a smooth point. This is so because $F_n(0,0)=J_n\mu(0)\circ\mathrm{pr}_{J_nT^*R}$, where $\mu^{-1}(0)$ is irreducible and contains a smooth point (by the assumption that (1) is true) and $\mu$ is jet-flat, together with Lemma \ref{lem: jet flatness for homogeneous map} (2). Since $\mathbb C^{\times}$ also contracts $J_n\mathfrak g^*\times \mathbb A^1$ to the origin, we conclude that $F_n^{-1}(s,t)$ is irreducible and contains a smooth point for all $(s,t)\in J_n\mathfrak g^*\times \mathbb A^1$. Specializing to $(s,t)=(0,1)$, we see that $(J_n\widetilde \mu)^{-1}(0)$ is irreducible and contains a smooth point. This proves the implication $(P_1)\Longrightarrow (\widetilde P_1)$.

Finally, if $(R,G)=(\mathrm{Rep}(\mathbf v,\mathbf w),\GL(\mathbf v))$ for some quiver $Q$, and assume that $\mu$ is jet-flat. Then $\widetilde\mu$ is jet-flat according the implication $(P_2)\Longrightarrow (\widetilde P_2)$ that we have previously proven. Therefore, to show that $\widetilde\mu^{-1}(0)$ is reduced and irreducible and has rational singularities, it suffices to show that $(J_n\widetilde \mu)^{-1}(0)$ is irreducible and contains a smooth point for all $n$. Define $\mu_n:J_nT^*R\to \mathfrak g^*$ by $$\mu_n:=\tau_n\circ J_n\mu~,$$ where $\tau_n:J_n\mathfrak g^*\to \mathfrak g^*$ is the natural projection from the jet scheme to the base. Then $(J_n\widetilde \mu)^{-1}(0)\cong \mu_n^{-1}(\mathfrak Z)$, where $\mathfrak Z$ is regarded as a subspace of $\mathfrak g^*$ via the natural embedding $i$. $\mu_n$ is flat because $\tau_n$ is smooth and $J_n\mu$ is flat. For a generic $\lambda\in \mathfrak Z$, $\mu^{-1}(\lambda)$ is smooth and connected \cite[Section 8]{crawley2003normality}, thus there exists open subset $U\subset \mathfrak Z$ such that $\mu|_{\mu^{-1}(U)}: \mu^{-1}(U)\to U $ is smooth with irreducible fibers. We conclude that $\mu_n|_{\mu_n^{-1}(U)}: \mu_n^{-1}(U)\to U $ is smooth with connected fibers, which implies that $ \mu_n^{-1}(\mathfrak Z)$ is irreducible and contains a smooth a point. This finishes the proof of $(P_2)\Longrightarrow (\widetilde P_1)$ in the quiver case.
\end{proof}

\medskip

\begin{example}
If $G$ is a torus, then $\widetilde\mu$ can be decomposed into $\widetilde\mu=\mathrm{pr}_{\mathfrak g^*}\circ a$, where $a: T^*R\times \mathfrak g^*\cong T^*R\times \mathfrak g^*$ is the automorphism such that
\begin{align*}
    a(x,z)=(x,z+\mu(x))~,
\end{align*}
and thus $\widetilde\mu$ is smooth. In particular $(\widetilde P_1)$ holds. However, it is possible that $\mu$ fails to be flat, for example when $G$ acts on $R$ trivially. It is also known that $(R,G)$ has property $(P_1)$ if $R$ is stable and faithful \cite[Proposition 5.4]{herbig2020symplectic}.
\end{example}

\medskip

\begin{example}
If $G$ is semisimple and $R=\mathfrak g^{\oplus p}$ where $p\ge 2$, then $(R,G)$ has property $(P_1)$ \cite[Theorem E]{herbig2021does}.
\end{example}

\medskip

\begin{defn}
    A quiver $Q$ is called totally negative if every node has at least two loops and every pair of nodes has at least one arrow between them.
\end{defn}

\medskip

\begin{example}\label{eq:totally negative quiver}
Let $(R,G)=(\mathrm{Rep}(\mathbf v,\mathbf w),\GL(\mathbf v))$ such that $Q$ is totally negative (see \cite{2022arXiv220914791V} for a detailed study). If $\mu$ is flat, then  $(R,G)$ has property $(P_2)$. If $\mathcal M^0(\mathbf v,\mathbf w)^{\mathrm{reg}}\neq \emptyset$, then $(R,G)$ has property $(P_1)$.
\end{example}

In the notation of Section \ref{sec:QSV}, we have the following proposition.

\medskip

\begin{proposition}\label{prop: J_n-flatness of super moment maps}
$\mu_{\mathsf s}$ is $J_n$-flat if and only if $\mu$ is $J_n$-flat. $\widetilde\mu_{\mathsf s}$ is $J_n$-flat if and only if $\widetilde\mu$ is $J_n$-flat.
\end{proposition}

\begin{proof}
Let us prove the equivalence: $\mu_{\mathsf s}$ is $J_n$-flat $\Longleftrightarrow$ $\mu$ is $J_n$-flat, and the extended version follows from similar argument. Since $J_n\gl(\mathbf v)^*$ is purely bosonic and smooth, the flatness of $J_n\mu_{\mathsf s}$ is equivalent to that $\{J_n\mu_{\mathsf s}^*(a_i)\}_{1\le i\le \dim J_n\gl(\mathbf v)}$ is a regular sequence, where $\{a_i\}_{1\le i\le \dim J_n\gl(\mathbf v)}$ is a basis of $J_n\gl(\mathbf v)$ and $J_n\mu_{\mathsf s}^*:\mathcal O_{J_n\gl(\mathbf v)^*}\to \mathcal O_{J_nT^*\mathrm{Rep}(\mathbf v,\mathbf w|\mathbf u)}$ is the pull-back map between function rings. Then the result is a corollary of the super-commutative algebra Lemma \ref{lem: regular sequence} below.
\end{proof}

\medskip

\begin{lem}\label{lem: regular sequence}
Let $A=A^0\oplus A^1$ be a Noetherian super-commutative algebra with even (resp. odd) component $A^0$ (resp. $A^1$), define $\mathfrak n_A$ to be the ideal generated by $A^1$, and define $A^{\mathrm{bos}}:=A/\mathfrak n_A$. Let $\{a_i\in A^0\}_{1\le i\le m}$ be a finite set of even elements, and suppose that $A$ satisfies the following condition:
\begin{itemize}
    \item[$(*)$] $\mathfrak n_A^i/\mathfrak n_A^{i+1}$ is a locally free $A^{\mathrm{bos}}$-module of finite rank~.
\end{itemize}
Then $\{a_i\}$ is a regular sequence if and only if $\{\overline{a}_i\}$ is a regular sequence, where $\overline{a}_i$ is the image of ${a}_i$ in $A^{\mathrm{bos}}$. Moreover, if $\{a_i\}$ is a regular sequence, then the quotient algebra $A/(a_i:1\le i\le m)$ also satisfies the condition $(*)$.
\end{lem}

\medskip

\begin{proof}
It is enough to prove the lemma for $m=1$, and the general situation follows from induction. In other words, we need to show that an even element $a\in A^0$ has no zero-divisor if and only if its image $\overline{a}\in A^{\mathrm{bos}}$ has no zero-divisor in $A^{\mathrm{bos}}$. Moreover, if $a\in A^0$ has no zero-divisor, then $\mathfrak n_A^i/(\mathfrak n_A^{i+1}+aA\cap \mathfrak n_A^i)$ is a locally free $A/(\mathfrak n_A+aA)$-module of finite rank. We observe that the multiplication map $A\overset{\times a}{\longrightarrow}A$ preserves the filtration $A\supset \mathfrak n_A\supset \mathfrak n_A^2\supset \cdots \supset \mathfrak n_A^k=0$ for some finite $k$ (since $A$ is Noetherian, such $k$ always exists).

The ``if'' part: suppose that $\overline{a}$ has no zero-divisor in $A^{\mathrm{bos}}$, then the associated graded map of $A\overset{\times a}{\longrightarrow}A$ with respect to the filtration induced by the powers of $\mathfrak n_A$ is the multiplication by $\overline{a}$:
\begin{align*}
    \bigoplus_{j=0}^{k-1}\mathfrak n_A^j/\mathfrak n_A^{j+1}\overset{\times\overline{a}}{\longrightarrow}\bigoplus_{j=0}^{k-1}\mathfrak n_A^j/\mathfrak n_A^{j+1}.
\end{align*}
Since $\overline{a}$ has no zero-divisor in $A^{\mathrm{bos}}=A/\mathfrak n_A$ and every direct summand $\mathfrak n_A^j/\mathfrak n_A^{j+1}$ is a locally free $A^{\mathrm{bos}}$-module of finite rank, it follows that the multiplication by $\overline{a}$ map is injective. Therefore the multiplication by $a$ map $A\overset{\times a}{\longrightarrow}A$ is also injective, i.e.\  $a$ has no zero-divisor in $A$.

The ``only if'' part. Suppose that ${a}$ has no zero-divisor in $A$. Let us take $k$ such that $\mathfrak n_A^{k-1}\neq 0$ and $\mathfrak n_A^k=0$. Note that the multiplication map $\mathfrak n_A^{k-1}\overset{\times a}{\longrightarrow}\mathfrak n_A^{k-1}$ agrees with the multiplication map $\mathfrak n_A^{k-1}\overset{\times \overline{a}}{\longrightarrow}\mathfrak n_A^{k-1}$. By our assumption, the map $\mathfrak n_A^{k-1}\overset{\times a}{\longrightarrow}\mathfrak n_A^{k-1}$ is injective. Then it follows that the map $(\mathfrak n_A^{k-1})_{\mathfrak{p}}\overset{\times \overline{a}}{\longrightarrow}(\mathfrak n_A^{k-1})_{\mathfrak{p}}$ is injective for every prime ideal ${\mathfrak{p}}\in \Spec A^{\mathrm{bos}}$. After localization to ${\mathfrak{p}}$, the locally free module $\mathfrak n_A^{k-1}$ becomes free, thus $A^{\mathrm{bos}}_{\mathfrak{p}}\overset{\times \overline{a}}{\longrightarrow}A^{\mathrm{bos}}_{\mathfrak{p}}$ is injective. Therefore $A^{\mathrm{bos}}\overset{\times \overline{a}}{\longrightarrow}A^{\mathrm{bos}}$ is injective, i.e.\  $\overline{a}$ has no zero-divisor in $A^{\mathrm{bos}}$.

Finally, if $a$ has no zero-divisor, then the natural inclusion $aA\cap \mathfrak n_A^i\supseteq a\cdot\mathfrak n_A^i$ is actually an equality $aA\cap \mathfrak n_A^i=a\cdot\mathfrak n_A^i$, because otherwise there exists $b\in \mathfrak n_A^j, b\notin \mathfrak n_A^{j+1},j<i$ such that $ab\in \mathfrak n_A^i$, but this will imply that the multiplication map $\mathfrak n_A^j/\mathfrak n_A^{j+1}\overset{\times \overline{a}}{\longrightarrow}\mathfrak n_A^j/\mathfrak n_A^{j+1}$ has nontrivial kernel, which contradicts with regularity of $\overline{a}$. Therefore we have
\begin{align*}
    \mathfrak n_A^i/(\mathfrak n_A^{i+1}+aA\cap \mathfrak n_A^i)=\mathfrak n_A^i/(\mathfrak n_A^{i+1}+a\cdot\mathfrak n_A^i)\cong (\mathfrak n_A^i/\mathfrak n_A^{i+1})\otimes _{A^{\mathrm{bos}}} A/(\mathfrak n_A+aA)~,
\end{align*}
which is locally free $A/(\mathfrak n_A+aA)$-module of finite rank. This completes the proof.
\end{proof}

\subsection{Conditions for moment map jet-flatness}\label{Conditions for moment map jet-flatness} 

We define the nilpotent cone $\mathcal N(\mathbf v,\mathbf w)$ of the doubled quiver representation $T^*\mathrm{Rep}(\mathbf v,\mathbf w)$ to be $q^{-1}(q(0))$, where $$q:\mu^{-1}(0)\to \mathcal M^0(\mathbf v,\mathbf w)$$ is the quotient map. Denote by $\mathsf g_i$ the number of loops at a node $i\in Q_0$, and let $\mathsf g:=\sum_{i\in Q_0}\mathsf g_i$.

\medskip

\begin{lem}[{\cite[Lemma 3.3]{budur2021rational}}]\label{lem: preimage of zero}
Let $\varphi_m:J_m\mu^{-1}(0)\to \mu^{-1}(0)$ be the natural projection, then
\begin{align*}
    \varphi_m^{-1}(0)\cong\begin{cases}
    J_{m-2}\mu^{-1}(0)\times T^*\mathrm{Rep}(\mathbf v,\mathbf w), & \text{if }m\ge 2 ~,\\
    T^*\mathrm{Rep}(\mathbf v,\mathbf w), & \text{if }m=1 ~.
    \end{cases}
\end{align*}
\end{lem}

\bigskip

\begin{lem}\label{lem: preimage of nilcone}
Let $\varphi_m:J_m\mu^{-1}(0)\to \mu^{-1}(0)$ be the natural projection. If $p\in \mathcal N(\mathbf v,\mathbf w)\setminus \{0\}$, then 
\begin{align*}
    \varphi_1^{-1}(p)\subsetneq T^*\mathrm{Rep}(\mathbf v,\mathbf w).
\end{align*}
Assume moreover that $m\ge 2$ and $\mu$ is $J_{m-2}$-flat, and $J_{m-2}\mu^{-1}(0)$ is reduced and irreducible, then 
\begin{align*}
    \dim \varphi_m^{-1}(p)<\dim \left(J_{m-2}\mu^{-1}(0)\times T^*\mathrm{Rep}(\mathbf v,\mathbf w)\right).
\end{align*}
\end{lem}

\bigskip 

\begin{proof}
Let us take an arbitrary point $(x(t),y(t),\alpha(t),\beta(t))\in \varphi_m^{-1}(p)$, where $$x(t)=\sum_{s=0}^m x^{(s)} t^s, \quad y(t)=\sum_{s=0}^m y^{(s)} t^s, \quad \alpha(t)=\sum_{s=0}^m \alpha^{(s)} t^s, \quad\beta(t)=\sum_{s=0}^m \beta^{(s)} t^s$$ are the mode expansions, then the equations that define $(x(t),y(t),\alpha(t),\beta(t))$ are
\begin{align*}
    (x^{(0)},y^{(0)},\alpha^{(0)},\beta^{(0)})=p,\quad \sum_{a\in Q_1}[x_a(t),y_a(t)]+\sum_{i\in Q_0} \alpha_i(t)\beta_i(t)\equiv 0 \pmod{t^{m+1}}.
\end{align*}
Define a linear map $\ell_p: T^*\mathrm{Rep}(\mathbf v,\mathbf w)\to \gl(\mathbf v)^*$ by 
\begin{align}\label{map ell_p}
    \ell_p(x,y,\alpha,\beta):= \sum_{a\in Q_1}([x_a^{(0)},y_a]+[x_a,y_a^{(0)}])+\sum_{i\in Q_0} (\alpha_i^{(0)}\beta_i+\alpha_i\beta_i^{(0)})~.
\end{align}
Note that $\ell_p$ is not trivial (the kernel $\ker \ell_p$ is not the whole $T^\ast \mathrm{Rep}(\mathbf v,\mathbf w)$), otherwise the point $(x^{(0)},y^{(0)},\alpha^{(0)},\beta^{(0)})$ is fixed by $\GL(\mathbf v)$-action, which contradicts with the fact that $p\in \mathcal N(\mathbf v,\mathbf w)\setminus \{0\}$. Thus $\varphi_1^{-1}(p)\cong \ker\ell_p\subsetneq T^*\mathrm{Rep}(\mathbf v,\mathbf w)$.

\medskip 

Now assume moreover that $m\ge 2$ and $\mu$ is $J_{m-2}$-flat, and that $J_{m-2}\mu^{-1}(0)$ is irreducible. With these assumptions, we are going to show that $\varphi_m^{-1}(p)$ is a proper subset of a reduced and irreducible variety $X_m$ such that $\dim X_m=\dim \left(J_{m-2}\mu^{-1}(0)\times T^*\mathrm{Rep}(\mathbf v,\mathbf w)\right)$. 

\medskip 

We first define a linear space $Q_m=T^*\mathrm{Rep}(\mathbf v,\mathbf w)^{\oplus m}$ which parametrizes $(x^{(s)},y^{(s)},\alpha^{(s)},\beta^{(s)})$ for $s=1,\cdots,m$, and a quadratic map $f: Q_m\to J_{m-2}\gl(\mathbf v)^*$ by the projection to the first $m-1$ summands, which is identified with $J_{m-2}T^*\mathrm{Rep}(\mathbf v,\mathbf w)$, followed by applying $J_{m-2}\mu$. By our assumption, $f$ is flat and $f^{-1}(0)$ is reduced and irreducible. Next we define a linear map $g: Q_m\to J_{m-2}\gl(\mathbf v)^*$ by projection to the last $m-1$ summands, followed by applying the map $\ell_p^{\oplus (m-1)}: T^*\mathrm{Rep}(\mathbf v,\mathbf w)^{\oplus (m-1)}\to (\gl(\mathbf v)^*)^{\oplus(m-1)}$, and then identify $(\gl(\mathbf v)^*)^{\oplus(m-1)}\cong J_{m-2}\gl(\mathbf v)^*$. We claim that $X_m:=h^{-1}(0)$ is the desired variety, where $h:=f+g:Q_m\to J_{m-2}\gl(\mathbf v)^*$. 

\medskip 

Let us first show that $\varphi_m^{-1}(p)$ is a proper subset of $X_m$. We identify $Q_m$ with the fiber of $p$ along the projection $J_mT^*\mathrm{Rep}(\mathbf v,\mathbf w)\to T^*\mathrm{Rep}(\mathbf v,\mathbf w)$, then the equation $$h(\{(x^{(s)},y^{(s)},\alpha^{(s)},\beta^{(s)})\}_{1\leq s \leq m})=0$$ is equivalent to $\sum_{a\in Q_1}[x_a(t),y_a(t)]+\sum_{i\in Q_0} \alpha_i(t)\beta_i(t)$ only containing terms linear in $t$, modulo $t^{m+1}$. Thus $\varphi_m^{-1}(p)\subset X_m$. Moreover, if we consider the embedding $$\mu^{-1}(0)^{(1)} := \mu^{-1}(0)\times\{0\}\times\cdots\times\{0\}\hookrightarrow T^*\mathrm{Rep}(\mathbf v,\mathbf w)^{\oplus m}=Q_m~,$$ then $\mu^{-1}(0)^{(1)}\subset X_m$. The intersection $\mu^{-1}(0)^{(1)}\cap \varphi_m^{-1}(p)$ is identified with $\mu^{-1}(0)\cap \ker\ell_p$, which must be a proper subset of $\mu^{-1}(0)$, otherwise $\ker\ell_p$ contains the tangent space of $\mu^{-1}(0)$ at $0$, which equals to the whole $T^*\mathrm{Rep}(\mathbf v,\mathbf w)$, and this contradicts with the non-triviality of $\ell_p$. Thus $\varphi_m^{-1}(p)$ is a proper subset of $X_m$.

\medskip 

Since $f$ is flat and $f^{-1}(0)\cong J_{m-2}\mu^{-1}(0)\times T^*\mathrm{Rep}(\mathbf v,\mathbf w)$ is reduced and irreducible by our assumptions, then according to Lemma \ref{lem: jet flatness for homogeneous map} $h$ is flat and $X_m=h^{-1}(0)$ is reduced and irreducible and $\dim X_m=\dim f^{-1}(0)=\dim \left(J_{m-2}\mu^{-1}(0)\times T^*\mathrm{Rep}(\mathbf v,\mathbf w)\right)$. Now $\varphi_m^{-1}(p)$ is a proper subset of irreducible variety $X_m$, thus 
\begin{align*}
    \dim \varphi_m^{-1}(p)<\dim X_m=\dim \left(J_{m-2}\mu^{-1}(0)\times T^*\mathrm{Rep}(\mathbf v,\mathbf w)\right)~.
\end{align*}
This finishes the proof.
\end{proof}

\medskip

Given the framed quiver data $(Q,\mathbf{v},\mathbf{w})$, the étale slice construction of \cite{crawley2003normality} defines an auxiliary quiver $(Q_\tau,\mathbf{v}_\tau,\mathbf{w}_\tau)$, where $\tau$ denotes the representation type of the auxiliary quiver. In Appendix \ref{sec:étaleSlice} we review and generalize the étale slice construction of \cite{crawley2003normality}. Referring to Appendix \ref{sec:étaleSlice} for the details, the theorem below is key for establishing criteria for the jet-flatness of moment maps.

\bigskip

\begin{theorembox}
\begin{thm}\label{thm: property P}
$(Q,\mathbf v,\mathbf w)$ has the property $(P_2)$ if the following three conditions are satisfied:
\begin{enumerate}[label=\arabic*)]
    \item $\mu$ is flat,
    \item $\dim \mathcal N(\mathbf v,\mathbf w)+\dim T^*\mathrm{Rep}(\mathbf v,\mathbf w)\le 2\dim \mu^{-1}(0)-2\mathsf g$, or equivalently
    \begin{align}\label{dimension bound}
        \dim \mathcal N(\mathbf v,\mathbf w)\le \mathbf v \cdot(2\mathbf w-\mathsf C\mathbf v)-2\mathsf g ~,
    \end{align}
    \item for every representation type $\tau$ which is not the trivial type, the auxiliary quiver $(Q_{\tau},\mathbf v_{\tau},\mathbf w_{\tau})$ has the property $(P_2)$.
\end{enumerate}
$(Q,\mathbf v,\mathbf w)$ has the property $(P_1)$ if the following three conditions are satisfied:
\begin{enumerate}[label=\alph*)]
    \item $\mu$ is flat and $\mu^{-1}(0)$ is reduced\footnote{The requirement for $\mu^{-1}(0)$ to be reduced is somewhat redundant. The properties that $\mu$ is flat and $\mu^{-1}(0)$ irreducible, imply that $\mu^{-1}(0)$ is reduced. Flatness implies that $\mu^{-1}(0)$ is a locally complete intersection, so one only needs to show that there is an open locus in $\mu^{-1}(0)$ which is reduced, and the stable locus $\mu^{-1}(0)_{(1,\mathbf v^+)}$ does the job, condition (c) implying that $\mu^{-1}(0)_{(1,\mathbf v^+)}$ is reduced).} and irreducible,
    \item inequality \eqref{dimension bound} holds,
    \item for every representation type $\tau$ which is not the trivial type, the auxiliary quiver $(Q_{\tau},\mathbf v_{\tau},\mathbf w_{\tau})$ has the property $(P_1)$.
\end{enumerate}
\end{thm}
\end{theorembox}

\bigskip

\begin{proof}
The proof of the theorem is based on \cite[Theorem 3.6]{budur2021rational}. In the \emph{loc. cit.}, a weaker version of the second statement is proven, namely $(Q,\mathbf v,\mathbf w)$ has the property $(P_1)$ if conditions a), b'), and c) are satisfied, where b') is the condition that inequality \eqref{dimension bound} strictly holds, i.e.
\begin{align}\label{dimension strict bound}
        \dim \mathcal N(\mathbf v,\mathbf w)< \mathbf v \cdot(2\mathbf w-\mathsf C\mathbf v)-2\mathsf g ~.
\end{align}
Let us explain why the strict inequality in \eqref{dimension strict bound} can be relaxed to \eqref{dimension bound}. The key step in the proof of \cite[Theorem 3.6]{budur2021rational} where the condition b') is used is to show that 
\begin{align*}
    \dim \varphi_m^{-1}(\mu^{-1}(0)_{\mathrm{sing}})< \dim \mu^{-1}(0)\cdot(m+1)
\end{align*}
for all $m\ge 1$, where $\mu^{-1}(0)_{\mathrm{sing}}$ is the singular locus of $\mu^{-1}(0)$ and $\varphi_m:J_m\mu^{-1}(0)\to \mu^{-1}(0)$ is the natural projection. It suffices to show that $\dim \varphi_m^{-1}(\mu_{\tau}^{-1}(0))< \dim \mu^{-1}(0)\cdot(m+1)$ for all representation types $\tau$ such that $\tau\neq (1,\mathbf v^+)$. If $\tau$ is not of trivial type, then we proceed by induction using the condition c). If $\tau$ is of trivial type, then we use the induction on $m$ to show that $\dim \varphi_m^{-1}(\mathcal N(\mathbf v,\mathbf w))<\dim \mu^{-1}(0)\cdot(m+1)-2\mathsf g$. When $m=1$, the condition b), Lemma \ref{lem: preimage of zero} and Lemma \ref{lem: preimage of nilcone} imply that
\begin{align*}
    \dim \varphi_1^{-1}(\mathcal N(\mathbf v,\mathbf w))&=\max(\dim \varphi_1^{-1}(0), \dim \varphi_1^{-1}(\mathcal N(\mathbf v,\mathbf w)\setminus 0))\\
    &\le \dim T^*\mathrm{Rep}(\mathbf v,\mathbf w)+\max(0, \dim \mathcal N(\mathbf v,\mathbf w)-1)\\
    &< \dim T^*\mathrm{Rep}(\mathbf v,\mathbf w)+\dim \mathcal N(\mathbf v,\mathbf w)\\
    &\le \dim T^*\mathrm{Rep}(\mathbf v,\mathbf w)+\mathbf v \cdot(2\mathbf w-\mathsf C\mathbf v)-2\mathsf g\\
    &=2\dim \mu^{-1}(0)-2\mathsf g ~.
\end{align*}
When $m\ge 2$, we use the induction on $m$, i.e.\  assume that $\mu$ is $J_{m-2}$-flat and $J_{m-2}\mu^{-1}(0)$ is reduced and irreducible, thus we can apply Lemma \ref{lem: preimage of nilcone} and obtain the following inequality
\begin{align*}
    \dim \varphi_m^{-1}(\mathcal N(\mathbf v,\mathbf w))&=\max(\dim \varphi_m^{-1}(0), \dim \varphi_m^{-1}(\mathcal N(\mathbf v,\mathbf w)\setminus 0))\\
    &\le \dim J_{m-2}\mu^{-1}(0)+\dim T^*\mathrm{Rep}(\mathbf v,\mathbf w)+\max(0, \dim \mathcal N(\mathbf v,\mathbf w)-1)\\
    &< \dim J_{m-2}\mu^{-1}(0)+\dim T^*\mathrm{Rep}(\mathbf v,\mathbf w)+\dim \mathcal N(\mathbf v,\mathbf w)\\
    &\le \dim J_{m-2}\mu^{-1}(0)+\dim T^*\mathrm{Rep}(\mathbf v,\mathbf w)+\mathbf v \cdot(2\mathbf w-\mathsf C\mathbf v)-2\mathsf g\\
    &=(m+1)\dim \mu^{-1}(0)-2\mathsf g ~.
\end{align*}
In either case, when $m=1$ or when $m\ge 2$, we have
\begin{align*}
    \dim \varphi_m^{-1}(\mu^{-1}_{\tau_{\mathrm{triv}}}(0))=\dim \varphi_m^{-1}(\mathcal N(\mathbf v,\mathbf w))+2\mathsf g< (m+1)\cdot\dim\mu^{-1}(0) ~.
\end{align*}
This finishes the proof of the second statement.

The proof of the first statement is easier. In fact, we only need to show that $$\dim \varphi_m^{-1}(\mathcal N(\mathbf v,\mathbf w))\le\dim \mu^{-1}(0)\cdot(m+1)-2\mathsf g ~.$$ We have inequality $\dim \varphi_m^{-1}(p)\le \dim \varphi_m^{-1}(0)$ for all $p\in \mathcal N(\mathbf v,\mathbf w)$ by semi-continuity of the dimension of fibers, thus we can apply Lemma \ref{lem: preimage of zero} to give a bound on the dimension of $\varphi_m^{-1}(\mathcal N(\mathbf v,\mathbf w))$ as follows
\begin{align*}
    \dim \varphi_m^{-1}(\mathcal N(\mathbf v,\mathbf w))&\le \dim \mathcal N(\mathbf v,\mathbf w)+\dim \varphi_m^{-1}(0)\\
    &\le \mathbf v \cdot(2\mathbf w-\mathsf C\mathbf v)-2\mathsf g+ (m-1)\dim \mu^{-1}(0)+\dim T^*\mathrm{Rep}(\mathbf v,\mathbf w)\\
    &=(m+1)\dim \mu^{-1}(0)-2\mathsf g ~.
\end{align*}
This finishes the proof of the first statement.
\end{proof}

We have the following analog of Lemma \ref{lem: preimage of zero} for $\widetilde\mu^{-1}(0)$.
\begin{lem}\label{lem: preimage of zero_tilde}
Let $\widetilde\varphi_m:J_m\widetilde\mu^{-1}(0)\to \widetilde\mu^{-1}(0)$ be the natural projection, then
\begin{align*}
    \widetilde\varphi_m^{-1}(0)\cong\begin{cases}
    J_{m-2}\widetilde\mu^{-1}(0)\times T^*\mathrm{Rep}(\mathbf v,\mathbf w), & \text{if }m\ge 2 ~,\\
    T^*\mathrm{Rep}(\mathbf v,\mathbf w), & \text{if }m=1 ~.
    \end{cases}
\end{align*}
\end{lem}

\begin{proof}
$\widetilde\varphi_m^{-1}(0)$ is the space of $(x(t) ,y(t), \alpha(t), \beta(t), z(t))$ where 
\footnote{Remark that the sums do not start at $s=0$; the zero modes are dropped here. Similarly to \cite[Lemma 3.3]{budur2021rational}, if the projection to the zero modes $(x^{(0)},y^{(0)},\alpha^{(0)},\beta^{(0)},z^{(0)})$ is a fixed point of the GL$(\mathbf{v})$ action, then the zero modes vanish and the commutators with the zero modes vanish. The sums therefore start with $s=1$.
}
$$x(t)=\sum_{s=1}^m x^{(s)} t^s, ~~ y(t)=\sum_{s=1}^m y^{(s)} t^s, ~~\alpha(t)=\sum_{s=1}^m \alpha^{(s)} t^s, ~~ \beta(t)=\sum_{s=1}^m \beta^{(s)} t^s, ~~ z(t)=\sum_{s=1}^m z^{(s)}t^s$$ are the mode expansions, then the equations that define $(x(t),y(t),\alpha(t),\beta(t),z(t))$ are
\begin{align*}
    \sum_{a\in Q_1}[x_a(t),y_a(t)]+\sum_i \alpha_i(t)\beta_i(t)+\sum_i z_i(t)\equiv 0 \pmod{t^{m+1}} ~.
\end{align*}
When $m=1$, then the above equation is equivalent to $z^{(1)}=0$, so we have $\widetilde\varphi_1^{-1}(0)\cong T^*\mathrm{Rep}(\mathbf v,\mathbf w)$. When $m\ge 2$, define $$\bar x(t)=\sum_{s=1}^{m-1}x^{(s)}t^{s-1},\quad  \bar y(t)=\sum_{s=1}^{m-1} y^{(s)} t^{s-1}, \quad \bar\alpha(t)=\sum_{s=1}^{m-1} \alpha^{(s)} t^{s-1}, \quad \bar\beta(t)=\sum_{s=1}^{m-1} \beta^{(s)} t^{s-1}$$ and $$\bar z(t)=\sum_{s=2}^m z^{(s)}t^{s-2}~,$$ then the above equation can be rewritten as follows
\begin{align*}
    z^{(1)}=0,\quad \sum_{a\in Q_1}[\bar x_a(t),\bar y_a(t)]+\sum_{i\in Q_0} \bar\alpha_i(t)\bar\beta_i(t)+\sum_{i\in Q_0} \bar z_i(t)\equiv 0 \pmod{t^{m-1}}~.
\end{align*}
Thus $\widetilde\varphi_m^{-1}(0)\cong J_{m-2}\widetilde\mu^{-1}(0)\times T^*\mathrm{Rep}(\mathbf v,\mathbf w)$.
\end{proof}

We also have the following analog of Lemma \ref{lem: preimage of nilcone} for $\widetilde\mu^{-1}(0)$.

\medskip

\begin{lem}\label{lem: preimage of nilcone_tilde}
Let $\widetilde\varphi_m:J_m\widetilde\mu^{-1}(0)\to \widetilde\mu^{-1}(0)$ be the natural projection. If $p\in \mathcal N(\mathbf v,\mathbf w)\setminus \{0\}$, then 
\begin{align*}
    \widetilde\varphi_1^{-1}(p)\subsetneq T^*\mathrm{Rep}(\mathbf v,\mathbf w).
~\end{align*}
Assume moreover that $m\ge 2$ and $\widetilde\mu$ is $J_{m-2}$-flat and $J_{m-2}\widetilde\mu^{-1}(0)$ is reduced and irreducible, then 
\begin{align*}
    \dim \widetilde\varphi_m^{-1}(p)<\dim \left(J_{m-2}\widetilde\mu^{-1}(0)\times T^*\mathrm{Rep}(\mathbf v,\mathbf w)\right)~.
\end{align*}
\end{lem}

\begin{proof}
Let us take an arbitrary point $(x(t),y(t),\alpha(t),\beta(t),z(t))\in \widetilde\varphi_m^{-1}(p)$, where $$x(t)=\sum_{s=0}^m x^{(s)} t^s, \quad y(t)=\sum_{s=0}^m y^{(s)} t^s, \quad \alpha(t)=\sum_{s=0}^m \alpha^{(s)} t^s, \quad \beta(t)=\sum_{s=0}^m \beta^{(s)} t^s~,$$  
$$z(t)=\sum_{s=0}^m z^{(s)} t^s$$ are the mode expansions, then the equations that define $(x(t),y(t),\alpha(t),\beta(t),z(t))$ are
\begin{align}\label{varphi^{-1}(p)}
    (x^{(0)},y^{(0)},\alpha^{(0)},\beta^{(0)})=p,\quad \sum_{a\in Q_1}[x_a(t),y_a(t)]+\sum_{i\in Q_0} \alpha_i(t)\beta_i(t)+\sum_{i\in Q_0} z_i(t)\equiv 0 \pmod{t^{m+1}}.
\end{align}
Then $\widetilde\varphi_1^{-1}(p)$ is isomorphic to $\ell_p^{-1}(\mathcal Z)$, where $\ell_p$ is the linear map defined in \eqref{map ell_p}. We claim that there exists $q\in T^*\mathrm{Rep}(\mathbf v,\mathbf w)$ such that $\widetilde\varphi_1(q)\notin \mathcal Z$. In fact $q$ can be taken such that it is nonzero only for one arrow and zero elsewhere. Let us assume that there exists $a\in Q_1$ such that $x_a^{(0)}$ is nonzero, and moreover $x_a^{(0)}$ is not a scalar matrix if $a$ is an edge loop. Then we can always find $y_a$ such that $x_a^{(0)}\cdot y_a$ is a nonzero non-scalar matrix if $a$ is not an edge loop, or $[x_a^{(0)}, y_a]\neq 0$ if $a$ is an edge loop. In this case we take the only nonzero arrow of $q$ to be $y_a$. Similarly, let us assume that there exists $i\in Q_0$ such that $\alpha_i^{(0)}$ is nonzero, then we can find $\beta_i$ such that $\alpha_i^{(0)}\cdot \beta_i$ is a nonzero non-scalar matrix. In this case we take the only nonzero arrow of $q$ to be $\beta_i$. Thus we have shown that $\widetilde\varphi_1^{-1}(p)$ is isomorphic to a proper subset of $T^*\mathrm{Rep}(\mathbf v,\mathbf w)$.

Now assume that $m\ge 2$ and $\widetilde\mu$ is $J_{m-2}$-flat and $J_{m-2}\widetilde\mu^{-1}(0)$ is reduced and irreducible, then we are going to show that $\varphi_m^{-1}(p)$ is a proper subset of a reduced and irreducible variety $Y_m$ such that $\dim Y_m=\dim \left(J_{m-2}\widetilde\mu^{-1}(0)\times T^*\mathrm{Rep}(\mathbf v,\mathbf w)\right)$.

Define a linear space $\widetilde Q_m=T^*\mathrm{Rep}(\mathbf v,\mathbf w)^{\oplus m}\oplus \mathcal Z^{\oplus (m-1)}$, and its direct summands are denoted by $T^*\mathrm{Rep}(\mathbf v,\mathbf w)^{(s)},s=1\cdots,m$ and $\mathcal Z^{(r)},r=2,\cdots m$ respectively. We define the coordinates of $T^*\mathrm{Rep}(\mathbf v,\mathbf w)^{(s)}$ to be $(x^{(s)},y^{(s)},\alpha^{(s)},\beta^{(s)})$ and the coordinate of $\mathcal Z^{(r)}$ to be $z^{(r)}$, and then define $$\bar x(t)=\sum_{s=1}^{m-1}x^{(s)}t^{s-1},\quad \bar y(t)=\sum_{s=1}^{m-1} y^{(s)} t^{s-1}, \quad \bar\alpha(t)=\sum_{s=1}^{m-1} \alpha^{(s)} t^{s-1}, \quad \bar\beta(t)=\sum_{s=1}^{m-1} \beta^{(s)} t^{s-1}~,$$  
$$\bar z(t)=\sum_{s=2}^{m} z^{(s)}t^{s-2}~.$$ Using the above coordinate system, we define a map $\widetilde f: \widetilde Q_m\to J_{m-2}\gl(\mathbf v)^*$ by 
\begin{align*}
    \widetilde f((x^{(s)},y^{(s)},\alpha^{(s)},\beta^{(s)})_{s=1}^m, (z^{(r)})_{r=2}^m):=\sum_{a\in Q_1}[\bar x_a(t),\bar y_a(t)]+\sum_{i\in Q_0} \bar\alpha_i(t)\bar\beta_i(t)+\sum_{i\in Q_0} \bar z_i(t) \pmod{t^{m-1}}.
\end{align*}
Note that we can identify $\widetilde Q_m\cong J_{m-2}(T^*\mathrm{Rep}(\mathbf v,\mathbf w)\times \mathcal Z)\oplus T^*\mathrm{Rep}(\mathbf v,\mathbf w)^{(m)}$, then $\widetilde f$ agrees with the composition of first projection to $J_{m-2}(T^*\mathrm{Rep}(\mathbf v,\mathbf w)\times \mathcal Z)$ then followed by $J_{m-2}\widetilde\mu$, thus $\widetilde f$ is flat and $\widetilde f^{-1}(0)$ is reduced and irreducible by our assumption. Next we define a linear map $\widetilde g: \widetilde Q_m\to J_{m-2}\gl(\mathbf v)^*$ by 
\begin{align*}
    \widetilde g((x^{(s)},y^{(s)},\alpha^{(s)},\beta^{(s)})_{s=1}^m, (z^{(r)})_{r=2}^m):=\sum_{r=2}^m \ell_p(x^{(r)},y^{(r)},\alpha^{(r)},\beta^{(r)})t^{r-2}~.
\end{align*}
We set $\widetilde h:=\widetilde f+\widetilde g$, and we claim that $Y_m:=\widetilde h^{-1}(0)$ is the desired variety. 

First of all, $Y_m$ contains the set of solutions to the equations \eqref{varphi^{-1}(p)}. In fact, $\varphi_m^{-1}(p)$ can be identified with the subset of $Y_m$ cut out by the linear relation
\begin{align}\label{defining eqn of varphi^{-1}(p) in Y_m}
    \ell_p(x^{(1)},y^{(1)},\alpha^{(1)},\beta^{(1)})\in \mathcal Z ~.
\end{align}
We claim that $\varphi_m^{-1}(p)$ is a proper subset of $Y_m$, in other words there exists $q\in Y_m$ such that \eqref{defining eqn of varphi^{-1}(p) in Y_m} fails to hold for $q$. According to our previous discussion, $\ell_p^{-1}(\mathcal{Z})\neq T^*\mathrm{Rep}(\mathbf v,\mathbf w)$. Since the tangent space of $\mu^{-1}(0)$ is $T^*\mathrm{Rep}(\mathbf v,\mathbf w)$, there must exist $(x,y,\alpha,\beta)\in \mu^{-1}(0)$ such that $\ell_p(x,y,\alpha,\beta)\notin \mathcal Z$, otherwise $\ell_p^{-1}(\mathcal Z)$ contains the tangent space of $\mu^{-1}(0)$ by the linearity of $\ell_p$ which is a contradiction. So we take $(x^{(1)},y^{(1)},\alpha^{(1)},\beta^{(1)})$ coordinate of $q$ to be $(x,y,\alpha,\beta)$ and zero elsewhere. Then it is easy to verify that $q\in Y_m$ but \eqref{defining eqn of varphi^{-1}(p) in Y_m} fails to hold for $q$. Thus we have shown that $\varphi_m^{-1}(p)$ is a proper subset of $Y_m$. Finally, we notice that $\widetilde f$ is flat, so the map $$F:=(\widetilde f+t\widetilde g,\mathrm{id}): \widetilde Q_m\times\mathbb A^1\to J_{m-2}\gl(\mathbf v)^*\times \mathbb A^1$$ is also flat, where $t$ is the coordinate on $\mathbb A^1$. Let us assign $\mathbb C^{\times}$-weights as follows
\begin{center}
\begin{tabular}{| c | c | c | c | c |} 
 \hline
  & $T^*\mathrm{Rep}(\mathbf v,\mathbf w)^{\oplus m}$ & $\mathcal Z^{\oplus(m-1)}$ & $J_{m-2}\gl(\mathbf v)^*$ & $\mathbb A^1$ \\  [1ex] 
 \hline
$\mathbb C^{\times}$-weights & $1$ & $2$ & $2$ & $1$ \\ [0.5ex] 
 \hline
\end{tabular}
\end{center}
Then $F$ is a homogenous map such that the $\mathbb C^{\times}$-actions are attracting both on the source and the target, thus the same deformation argument in the proof of Lemma \ref{lem: jet flatness for homogeneous map}(2) can be applied to $F$. Since $\widetilde f^{-1}(0)=F^{-1}(0,0)$ is reduced and irreducible, we conclude that $F^{-1}(u,t)$ is reduced and irreducible and $\dim F^{-1}(u,t)=\dim F^{-1}(0,0)$ for all $(u,t)\in J_{m-2}\gl(\mathbf v)^*\times \mathbb A^1$. In particular $Y_m\cong F^{-1}(0,1)$ is reduced and irreducible and $$\dim Y_m=\dim F^{-1}(0,0)=\dim \left(J_{m-2}\widetilde\mu^{-1}(0)\times T^*\mathrm{Rep}(\mathbf v,\mathbf w)\right)~.$$ 
This finishes the proof of the lemma.
\end{proof}

\bigskip

\begin{theorembox}
\begin{thm}\label{thm: property P tilde}
$(Q,\mathbf v,\mathbf w)$ has the property $(\widetilde P_2)$ if the following three conditions are satisfied:
\begin{enumerate}[label=\arabic*)]
    \item $\widetilde \mu$ is flat; 
    \item $\dim \mathcal N(\mathbf v,\mathbf w)+\dim T^*\mathrm{Rep}(\mathbf v,\mathbf w)\le 2\dim \widetilde \mu^{-1}(0)-2\mathsf g$, or equivalently
    \begin{align}\label{dimension bound tilde}
        \dim \mathcal N(\mathbf v,\mathbf w)\le \mathbf v \cdot(2\mathbf w-\mathsf C\mathbf v)-2\mathsf g+2|Q_0|~,
    \end{align}
    \item for every representation type $\tau$ which is not the trivial type, the auxiliary quiver $(Q_{\tau},\mathbf v_{\tau},\mathbf w_{\tau})$ satisfies that either (1) it has the property $(\widetilde P_2)$ and the transition map $\mathcal Z\to \mathcal Z_{\tau}$ is surjective, or (2) it has the property $(P_2)$.
\end{enumerate}
$(Q,\mathbf v,\mathbf w)$ has the property $(\widetilde P_1)$ if the following three conditions are satisfied:
\begin{enumerate}[label=\alph*)]
    \item $\widetilde \mu$ is flat and $\widetilde \mu^{-1}(0)$ is irreducible\footnote{Reducedness follows from the rest of conditions, similar to Theorem \ref{thm: property P}};
    \item inequality \eqref{dimension bound tilde} holds,
    \item for every representation type $\tau$ which is not the trivial type, the auxiliary quiver $(Q_{\tau},\mathbf v_{\tau},\mathbf w_{\tau})$ satisfies that either (1) it has the property $(\widetilde P_1)$ and the transition map $\mathcal Z\to \mathcal Z_{\tau}$ is surjective, or (2) it has the property $(P_1)$.
\end{enumerate}
\end{thm}
\end{theorembox}

\medskip

\begin{proof}
For the first statement, we need to show that
\begin{align*}
    \dim\widetilde\varphi_m^{-1}(\widetilde\mu^{-1}(0)_{\mathrm{sing}})\le \dim \widetilde\mu^{-1}(0)\cdot (m+1)
\end{align*}
for all $m\ge 1$, where $\widetilde\mu^{-1}(0)_{\mathrm{sing}}$ is the singular locus of $\widetilde\mu^{-1}(0)$ and $\widetilde\varphi_m:J_m\widetilde\mu^{-1}(0)\to \widetilde\mu^{-1}(0)$ is the natural projection. It suffices to show that $\dim \widetilde\varphi_m^{-1}(\widetilde\mu^{-1}_{\tau}(0))\le \dim \widetilde\mu^{-1}(0)\cdot(m+1)$ for all representation types $\tau$ such that $\tau\neq (1,\mathbf v^+)$. 

If $\tau$ is not of trivial type, then we proceed by induction using the condition 3) as follows. Assume that $(Q_\tau,\mathbf v_\tau,\mathbf w_\tau)$ has the property $(\widetilde P_2)$ and the transition map $\mathcal Z\to \mathcal Z_\tau$ is surjective, then for every $x\in \widetilde\mu^{-1}(0)_{\tau}$ there is an étale neighborhood $U$ of $x$ which is smooth over $\widetilde\mu_{\tau}^{-1}(0)$, whence it follows that $\dim J_mU=(m+1)\cdot\dim U$. Assume otherwise that $(Q_\tau,\mathbf v_\tau,\mathbf w_\tau)$ has the property $(P_2)$, then for every $x\in \widetilde\mu^{-1}(0)_{\tau}$ there is an étale neighborhood $U$ of $x$ which is jet-flat over $\mathcal Z$, whence it follows that $\dim J_mU=(m+1)\cdot\dim U$. In any case we have $\dim J_mU\le \dim \widetilde\mu^{-1}(0)\cdot(m+1)$.

If $\tau$ is of trivial type, then the condition 2) and Lemma \ref{lem: preimage of zero_tilde} imply that
\begin{align*}
    \dim \widetilde\varphi_m^{-1}(\mathcal N(\mathbf v,\mathbf w))&\le \dim \mathcal N(\mathbf v,\mathbf w)+\dim \widetilde\varphi_m^{-1}(0)\\
    &\le \mathbf v \cdot(2\mathbf w-\mathsf C\mathbf v)-2\mathsf g+2|Q_0|+(m-1)\dim \widetilde\mu^{-1}(0)+\dim T^*\mathrm{Rep}(\mathbf v,\mathbf w)\\
    &=(m+1)\cdot\dim \widetilde\mu^{-1}(0)-2\mathsf g ~.
\end{align*}
Thus for trivial type $\tau_{\mathrm{triv}}$, $$\dim\widetilde\varphi_m^{-1}(\widetilde\mu^{-1}_{\tau_{\mathrm{triv}}}(0))= \dim \widetilde\varphi_m^{-1}(\mathcal N(\mathbf v,\mathbf w))+2\mathsf g\le (m+1)\cdot\dim \widetilde\mu^{-1}(0)~.$$

The proof of the second statement is similar. It suffices to show that $\dim \widetilde\varphi_m^{-1}(\widetilde\mu^{-1}_{\tau}(0))< \dim \widetilde\mu^{-1}(0)\cdot(m+1)$ for all representation types $\tau$ such that $\tau\neq (1,\mathbf v^+)$. 

If $\tau$ is not of trivial type, then we proceed by induction using the condition c) as follows. Assume that $(Q_\tau,\mathbf v_\tau,\mathbf w_\tau)$ has the property $(\widetilde P_1)$ and the transition map $\mathcal Z\to \mathcal Z_\tau$ is surjective, then for every $x\in \widetilde\mu^{-1}_{\tau}(0)$ there is an étale neighborhood $U$ of $x$ such that there is a smooth map $U\to \widetilde\mu_{\tau}^{-1}(0)$ with connected fibers, whence it follows that $\dim J_m U=(m+1)\cdot\dim U$ and $J_m U$ is irreducible. Assume otherwise that $(Q_\tau,\mathbf v_\tau,\mathbf w_\tau)$ has the property $(P_1)$, then for every $x\in \widetilde\mu^{-1}_{\tau}(0)$ there is an étale neighborhood $U$ of $x$ such that there is a smooth map $U\to \widetilde\mu_{\tau}^{-1}(0)\times_{\mathcal Z_\tau}\mathcal Z'$ with connected fibers, where $\mathcal Z'$ is the image of the transition map $\mathcal Z\to\mathcal Z_\tau$. Since $J_m\widetilde\mu_{\tau}^{-1}(0)$ is irreducible and $J_m\widetilde\mu_\tau$ is flat,  $J_m(\widetilde\mu_{\tau}^{-1}(0)\times_{\mathcal Z_\tau}\mathcal Z')$ is also irreducible, whence it follows that $J_mU$ is irreducible. Moreover $U$ is jet-flat over $\mathcal Z$, thus $\dim J_mU=(m+1)\cdot\dim U$. In any case we have $\dim \widetilde\varphi_m^{-1}(\widetilde\mu^{-1}_{\tau}(0))< \dim \widetilde\mu^{-1}(0)\cdot(m+1)$.

If $\tau$ is of trivial type, then it suffices to show that $\dim \widetilde\varphi_m^{-1}(\mathcal N(\mathbf v,\mathbf w))<\dim \widetilde\mu^{-1}(0)\cdot(m+1)-2\mathsf g$. When $m=1$, the condition b),  Lemma \ref{lem: preimage of zero_tilde} and Lemma \ref{lem: preimage of nilcone_tilde} imply that
\begin{align*}
    \dim \widetilde\varphi_1^{-1}(\mathcal N(\mathbf v,\mathbf w))&=\max(\dim \widetilde\varphi_1^{-1}(0), \dim \widetilde\varphi_1^{-1}(\mathcal N(\mathbf v,\mathbf w)\setminus 0))\\
    &\le \dim T^*\mathrm{Rep}(\mathbf v,\mathbf w)+\max(0, \dim \mathcal N(\mathbf v,\mathbf w)-1)\\
    &< \dim T^*\mathrm{Rep}(\mathbf v,\mathbf w)+\dim \mathcal N(\mathbf v,\mathbf w)\\
    &\le \dim T^*\mathrm{Rep}(\mathbf v,\mathbf w)+\mathbf v \cdot(2\mathbf w-\mathsf C\mathbf v)-2\mathsf g+2|Q_0|\\
    &=2\dim \widetilde\mu^{-1}(0)-2\mathsf g ~.
\end{align*}
When $m\ge 2$, we use the induction on $m$, i.e.\  assume that $\widetilde\mu$ is $J_{m-2}$-flat and $J_{m-2}\widetilde\mu^{-1}(0)$ is reduced and irreducible, thus we can apply Lemma \ref{lem: preimage of nilcone_tilde} and obtain the following inequality
\begin{align*}
    \dim \widetilde\varphi_m^{-1}(\mathcal N(\mathbf v,\mathbf w))&=\max(\dim \widetilde\varphi_m^{-1}(0), \dim \widetilde\varphi_m^{-1}(\mathcal N(\mathbf v,\mathbf w)\setminus 0))\\
    &\le \dim J_{m-2}\widetilde\mu^{-1}(0)+\dim T^*\mathrm{Rep}(\mathbf v,\mathbf w)+\max(0, \dim \mathcal N(\mathbf v,\mathbf w)-1)\\
    &< \dim J_{m-2}\widetilde\mu^{-1}(0)+\dim T^*\mathrm{Rep}(\mathbf v,\mathbf w)+\dim \mathcal N(\mathbf v,\mathbf w)\\
    &\le \dim J_{m-2}\widetilde\mu^{-1}(0)+\dim T^*\mathrm{Rep}(\mathbf v,\mathbf w)+\mathbf v \cdot(2\mathbf w-\mathsf C\mathbf v)-2\mathsf g+2|Q_0|\\
    &=(m+1)\dim \widetilde\mu^{-1}(0)-2\mathsf g ~.
\end{align*}
In either case when $m=1$ or when $m\ge 2$, we have $\dim \widetilde\varphi_m^{-1}(\mathcal N(\mathbf v,\mathbf w))<\dim \widetilde\mu^{-1}(0)\cdot(m+1)-2\mathsf g$, thus $$\dim\widetilde\varphi_m^{-1}(\widetilde\mu^{-1}_{\tau_{\mathrm{triv}}}(0))= \dim \widetilde\varphi_m^{-1}(\mathcal N(\mathbf v,\mathbf w))+2\mathsf g< (m+1)\cdot\dim \widetilde\mu^{-1}(0) ~.$$
This finishes the proof of the second statement.
\end{proof}

\medskip 

\begin{remark}\label{rmk: bound of dim of nilpotent cone}
It is proven in \cite[Theorem 6.3]{crawley2003normality} that
\begin{align}\label{bound of dim of nilpotent cone}
    \dim \mathcal N(\mathbf v,\mathbf w)\le \mathbf v\cdot(\mathbf w+\mathsf Q\mathbf v)-\mathsf g ~,
\end{align}
where $\mathsf Q$ is the adjacency matrix of $Q$. Thus
\begin{align*}
    \text{\eqref{dimension bound}}&\Longleftarrow \mathbf v\cdot(\mathbf w-\mathsf C\mathbf v-\mathsf Q\mathbf v)\ge \mathsf g ~;\\
\text{\eqref{dimension bound tilde}}&\Longleftarrow \mathbf v\cdot(\mathbf w-\mathsf C\mathbf v-\mathsf Q\mathbf v)+2|Q_0|\ge \mathsf g ~.
\end{align*}
\end{remark}

\medskip 

\begin{remark}\label{rmk: N is Lagrangian when Q has no edge loop}
If $Q$ has no edge loop, i.e.\  $\mathsf g=0$, then according to \cite[Theorem 12.9]{lusztig1991quivers} $\mathcal N(\mathbf v,\mathbf w)$ is a Lagrangian subvariety of $T^*\mathrm{Rep}(\mathbf v,\mathbf w)$, in particular the inequality \eqref{bound of dim of nilpotent cone} is saturated, i.e.\  $\dim \mathcal N(\mathbf v,\mathbf w)=\mathbf v\cdot(\mathbf w+\mathsf Q\mathbf v)$ where $\mathsf Q$ is the adjacency matrix of $Q$. In this case, 
\begin{align*}
\text{\eqref{dimension bound}}&\Longleftrightarrow \mathbf v\cdot(\mathbf w-\mathsf C\mathbf v-\mathsf Q\mathbf v)\ge 0 ~;\\
\text{\eqref{dimension bound tilde}}&\Longleftrightarrow \mathbf v\cdot(\mathbf w-\mathsf C\mathbf v-\mathsf Q\mathbf v)+2|Q_0|\ge 0 .
\end{align*}
\end{remark}

\subsection{Examples}\label{sec:examples}

\begin{thm}\label{thm: Dynkin quiver tilde P_1, I}
Let $Q$ be a Dynkin quiver, and assume that $(\mathbf v,\mathbf w)$ is chosen such that $\mu$ is flat, then $(Q,\mathbf v,\mathbf w)$ has the property $(\widetilde P_1)$ if the following condition is satisfied:
\begin{itemize}
    \item[$\star)$] for every decomposition $\mathbf v=\mathbf v^{(0)}+\sum_{s=1}^r k_s\mathbf v^{(s)}$ such that $\mathbf v^{(0)}\in \mathbb Z^{Q_0}_{\ge 0}$ and $k_s\in \mathbb Z_{>0}$ and $\mathbf v^{(s)},s=1,\cdots,r$ are distinct positive roots, the following inequality holds:
    \begin{align}
     \mathbf v\cdot(\mathbf w-\frac{1}{2}\mathsf C\mathbf v)-\mathbf v^{(0)}\cdot(\mathbf w-\frac{1}{2}\mathsf C\mathbf v^{(0)})+2r\ge \sum_{s=1}^r k_s^2 ~.
    \end{align}
\end{itemize}
\end{thm}

\begin{proof}
We claim that if $\mu$ is flat and the condition $\star)$ is satisfied, then for every auxiliary quiver $(Q_{\tau},\mathbf v_{\tau},\mathbf w_{\tau})$, the moment map $\mu_\tau$ is flat and $(Q_{\tau},\mathbf v_{\tau},\mathbf w_{\tau})$ satisfies the the condition $\star)$. $\mu_\tau$ is flat by Lemma \ref{lem: inherit flatness}, it remains to show that $\star)$ is satisfied. Let $\mathbf v_\tau=\mathbf v^{(0)}_\tau+\sum_{s=1}^r k_s\mathbf v^{(s)}_\tau$ be a decomposition such that $\mathbf v^{(0)}_\tau\in \mathbb Z^{Q_{\tau,0}}_{\ge 0}$ and $k_s\in \mathbb Z_{>0}$ and $\mathbf v^{(s)}_\tau,s=1,\cdots,r$ are distinct positive roots of $Q_\tau$. Write the representation type $\tau=(1,\mathbf v^{(0)+};m_1,\beta^{(1)};\cdots;m_t,\beta^{(t)})$, and let $\eta$ be the linear map $\mathbb C^{Q_{\tau,0}}\to \mathbb C^{Q_0}$ such that $\eta(\mathbf e_i)=\beta^{(i)},i=1,\cdots,t$ where $\mathbf e_i$ is $i$-th unit vector. In the proof of Lemma \ref{lem: aux quiver of Dynkin quiver}, we show that $Q_\tau$ is a sub-quiver of $Q'$, and the latter is the Dynkin quiver of the semisimple Lie algebra $[\mathfrak l,\mathfrak l]$, where $\mathfrak l$ is a Levi subalgebra of the Lie algebra $\mathfrak q$ associated to the quiver $Q$. In particular, every root of $Q_\tau$ is a root of $\mathfrak l$, and therefore is a root of $\mathfrak q$. In other words, $\eta:\mathbb C^{Q_{\tau,0}}\to \mathbb C^{Q_0}$ maps roots of $Q_\tau$ to roots of $Q$. Now let $\bar{\mathbf v}^{(0)}:=\mathbf v^{(0)}+\eta(\mathbf v^{(0)}_\tau)\in \mathbb Z^{Q_0}_{\ge 0}$ and $\bar{\mathbf v}^{(s)}:=\eta(\mathbf v^{(s)}_\tau)\in\mathbb Z^{Q_0}_{\ge 0}, s=1,\cdots,r$, then we have decomposition
\begin{align*}
    \mathbf v=\bar{\mathbf v}^{(0)}+\sum_{s=1}^r k_s\bar{\mathbf v}^{(s)} = \mathbf v^{(0)}+\eta(\mathbf v_\tau) ~,
\end{align*}
such that $\bar{\mathbf v}^{(s)}, s=1,\cdots,r$ are distinct positive roots of $Q$. Since $(Q,\mathbf v,\mathbf w)$ satisfies the condition $\star)$, we have
\begin{align}\label{temp inequality}
    \mathbf v\cdot(\mathbf w-\frac{1}{2}\mathsf C\mathbf v)-\bar{\mathbf v}^{(0)}\cdot(\mathbf w-\frac{1}{2}\mathsf C\bar{\mathbf v}^{(0)})+2r\ge \sum_{s=1}^r k_s^2 ~.
\end{align}
We can rewrite the LHS of \eqref{temp inequality}, using Definition \ref{defn: auxiliary quiver}, as
\begin{align*}
    &\mathbf v\cdot(\mathbf w-\frac{1}{2}\mathsf C\mathbf v)-{\mathbf v}^{(0)}\cdot(\mathbf w-\frac{1}{2}\mathsf C{\mathbf v}^{(0)})-\eta(\mathbf v^{(0)}_\tau)\cdot(\mathbf w-\frac{1}{2}\mathsf C{\mathbf v}^{(0)})+\frac{1}{2}\eta(\mathbf v^{(0)}_\tau)\cdot \mathsf C\eta(\mathbf v^{(0)}_\tau)+2r\\
    &=\mathbf v_\tau\cdot(\mathbf w_\tau-\frac{1}{2}\mathsf C_\tau\mathbf v_\tau)-\mathbf v^{(0)}_\tau\cdot\mathbf w_\tau+\frac{1}{2}\mathbf v^{(0)}_\tau\cdot \mathsf C_\tau\mathbf v^{(0)}_\tau+2r ~,
\end{align*}
and then we arrive at the desired inequality
\begin{align*}
    \mathbf v_\tau\cdot(\mathbf w_\tau-\frac{1}{2}\mathsf C_\tau\mathbf v_\tau)-\mathbf v^{(0)}_\tau\cdot(\mathbf w_\tau-\frac{1}{2}\mathsf C_\tau\mathbf v^{(0)}_\tau)+2r\ge \sum_{s=1}^r k_s^2 ~.
\end{align*}
Thus $(Q_{\tau},\mathbf v_{\tau},\mathbf w_{\tau})$ satisfies the the condition $\star)$.

Now let us examine that $(Q,\mathbf v,\mathbf w)$ satisfies the conditions a), b), and c) in the Theorem \ref{thm: property P tilde}. For the condition a), the moment map $\mu$ is flat by assumption, then it follows from Propositions \ref{prop: flat moment map} and \ref{prop: flat moment map 2} that $\widetilde\mu^{-1}(0)$ is reduced and irreducible. For the condition b), we take $\mathbf v^{(s)},s=1,\cdots,r$ be distinct coordinate vectors and $\mathbf v^{(0)}=0$ then the inequality in the $\star)$ can be rewritten as
\begin{align*}
    \mathbf v\cdot(\mathbf w-\mathsf C\mathbf v-\mathsf Q\mathbf v)+2|Q_0|\ge 0 ~.
\end{align*}
According to Remark \ref{rmk: N is Lagrangian when Q has no edge loop}, the above inequality is equivalent to \eqref{dimension bound tilde}. This verifies the condition b). For the condition c), every auxiliary quiver $Q_\tau$ is a Dynkin quiver by the Lemma \ref{lem: aux quiver of Dynkin quiver}, and $\mu_\tau$ is flat by Lemma \ref{lem: inherit flatness}. According to our previous discussion, $(Q_{\tau},\mathbf v_{\tau},\mathbf w_{\tau})$ satisfies the the condition $\star)$. We notice that if a representation type $\tau$ is not trivial, then at least one of $\beta^{(i)}$ in $\tau=(1,\mathbf v^{(0)+};m_1,\beta^{(1)};\cdots;m_t,\beta^{(t)})$ is not a coordinate vector, thus
\begin{align*}
    \sum_{i\in Q_{\tau,0}} v_{\tau,i}=\sum_{i=1}^tm_i<\sum_{i=1}^tm_i\sum_{j\in Q_0}\beta^{(i)}_j=\sum_{j\in Q_0}( v_j-  v^{(0)}_j)\le \sum_{j\in Q_0}  v_j ~.
\end{align*}
Then we can use the induction on the quantity $|\mathbf v|:=\sum_{j\in Q_0} v_j$. Namely, let us assume that if $(Q',\mathbf v',\mathbf w')$ is a Dynkin quiver such that the moment map is flat and it satisfies the condition $\star)$ and $|\mathbf v'|<|\mathbf v|$, then $(Q',\mathbf v',\mathbf w')$ has the property $(\widetilde P_1)$. Then $(Q_{\tau},\mathbf v_{\tau},\mathbf w_{\tau})$ has the property $(\widetilde P_1)$ by induction, and the transition map $\mathcal Z\to \mathcal Z_\tau$ is surjective by the Lemma \ref{lem: aux quiver of Dynkin quiver}, thus $(Q,\mathbf v,\mathbf w)$ satisfies the condition c). Therefore $(Q,\mathbf v,\mathbf w)$ has the property $(\widetilde P_1)$ by Theorem \ref{thm: property P tilde}.
\end{proof}

\medskip 

\begin{example}\label{ex: A_1 quiver tilde P_1}
Let $Q$ be $A_1$ quiver, and take $\mathbf v,\mathbf w\in \mathbb N$ such that $\mathbf w\ge 2\mathbf v$, then the moment map is flat by Remark \ref{rmk: KM quiver simple locus}, and it is easy to see that the condition $\star)$ is satisfied, thus $(Q,\mathbf v,\mathbf w)$ has the property $(\widetilde P_1)$.
\end{example}

\begin{thm}\label{thm: Dynkin quiver tilde P_1, II}
Let $Q$ be a Dynkin quiver, and assume that $\mathbf w-\mathsf C\mathbf v\in \mathbb Z^{Q_0}_{\ge 0}$, and that $ v_i\le 2$ for all $i\in Q_0$, and that $\#\{i\in Q_0: v_i=2\}\le 2$, then $(Q,\mathbf v,\mathbf w)$ has the property $(\widetilde P_1)$. 
\end{thm}

\medskip 

\begin{proof}
Let us examine that $(Q,\mathbf v,\mathbf w)$ satisfies the conditions a), b), and c) in the Theorem \ref{thm: property P tilde}. For the condition a), the moment map $\mu$ is flat by Remark \ref{rmk: KM quiver simple locus} and Proposition \ref{prop: simple locus}; then it follows from Proposition \ref{prop: flat moment map 2} that $\widetilde\mu^{-1}(0)$ is reduced and irreducible. 

According to Remark \ref{rmk: N is Lagrangian when Q has no edge loop}, the condition b) is equivalent to $\mathbf v\cdot(\mathbf w-\mathsf C\mathbf v-\mathsf Q\mathbf v)+2|Q_0|\ge 0$. Since $\mathbf w-\mathsf C\mathbf v\in \mathbb Z^{Q_0}_{\ge 0}$ by the assumption, it suffices to show that 
\begin{align}\label{combinatorial inequality}
    \mathbf v\cdot\mathsf Q\mathbf v\le 2|Q_0|
\end{align}
Since $Q$ has no loops and there is at most one trivalent vertex in $Q$, thus 
\begin{align*}
    \text{LHS of \eqref{combinatorial inequality}}\le (|Q_0|-1)+6=|Q_0|+5 ~.
\end{align*}
If $|Q_0|\ge 5$ then \eqref{combinatorial inequality} holds. When $Q$ is of type A, there is no trivalent vertex, so $\text{LHS of \eqref{combinatorial inequality}}\le |Q_0|+4$, thus \eqref{combinatorial inequality} holds for $A_4$ quiver. Note that \eqref{combinatorial inequality} automatically holds for $A_1$ quiver, so it remains to check \eqref{combinatorial inequality} for $A_2,A_3,D_4$ quivers, which is straightforward and we omit the details. 

For the condition c), every auxiliary quiver $Q_\tau$ is a Dynkin quiver by the Lemma \ref{lem: aux quiver of Dynkin quiver}, and $\mathbf w_\tau-\mathsf C_\tau\mathbf v_\tau\in \mathbb Z^{Q_{\tau,0}}_{\ge 0}$ by Lemma \ref{lem: inherit positivity}. It is also easy to see that $ v_{\tau,i}\le 2$ for all $i\in Q_{\tau,0}$ and $\#\{i\in Q_{\tau,0}: v_{\tau,i}=2\}\le 2$. Then we can use the induction on the quantity $|\mathbf v|:=\sum_{j\in Q_0} v_j$. Namely, let us assume that if $(Q',\mathbf v',\mathbf w')$ is a Dynkin quiver such that $\mathbf w'-\mathsf C'\mathbf v'\in \mathbb Z^{Q'_0}_{\ge 0}$ and $v'_i\le 2$ for all $i\in Q'_0$, and that $\#\{i\in Q'_0: v'_i=2\}\le 2$, then $(Q',\mathbf v',\mathbf w')$ has the property $(\widetilde P_1)$. The same argument in the proof of Theorem \ref{thm: Dynkin quiver tilde P_1, I} shows that $|\mathbf v_\tau|<|\mathbf v|$, so $(Q_{\tau},\mathbf v_{\tau},\mathbf w_{\tau})$ has the property $(\widetilde P_1)$ by induction. The transition map $\mathcal Z\to \mathcal Z_\tau$ is surjective by the Lemma \ref{lem: aux quiver of Dynkin quiver}. This verifies the condition c) for $(Q,\mathbf v,\mathbf w)$. Therefore $(Q,\mathbf v,\mathbf w)$ has the property $(\widetilde P_1)$ by Theorem \ref{thm: property P tilde}.
\end{proof}

\medskip 

\begin{thm}\label{thm: Jordan quiver tilde P_1, III}
Let $Q$ be the Jordan quiver (one node with one loop), then $(Q,\mathbf v,\mathbf w)$ has the property $(P_1)$ if $\mathbf w>\mathbf v$, and it has property $(\widetilde P_1)$ if $\mathbf w\ge \mathbf v$.
\end{thm}

\medskip

\begin{proof}
The Cartan matrix $\mathsf C$ is zero in this case, then it is easy to see that \eqref{key inequality} holds for $(Q,\mathbf v,\mathbf w)$ with arbitrary nonzero $\mathbf v$ and nonzero $\mathbf w$, so $\mu$ is flat according to Proposition \ref{prop: flat moment map}. Then $\widetilde\mu$ is flat by Remark \ref{rmk: mu flat implies tilde mu flat}, and $\widetilde\mu^{-1}(0)$ is reduced and irreducible by Proposition \ref{prop: flat moment map 2}. 

Now let us assume that $\mathbf w>\mathbf v$, we are going to show that $(Q,\mathbf v,\mathbf w)$ has the property $(P_1)$, by verifying the conditions a), b), and c) in Theorem \ref{thm: property P}. It is easy to see that \eqref{key strict inequality} holds for $(Q,\mathbf v,\mathbf w)$ when $\mathbf w>\mathbf v$, so $\mu^{-1}(0)^{\mathrm{reg}}$ is nonempty. Then according to Proposition \ref{prop: simple locus}, $\mu^{-1}(0)$ is reduced and irreducible, whence the condition a) is verified. According to Remark \ref{rmk: bound of dim of nilpotent cone}, the condition b) follows from the inequality $\mathbf v\cdot(\mathbf w-\mathbf v)\ge 1$, which is obviously true since $\mathbf w-\mathbf v> 0$. It is easy to see that every auxiliary quiver $Q_\tau$ is disjoint union of Jordan quivers. Moreover $\mathbf v_{\tau,i}<\mathbf v$ and $\mathbf w_{\tau,i}=\mathbf w$ for all $i\in Q_{\tau,0}$, in particular $\mathbf w_{\tau,i}>\mathbf v_{\tau,i}$, so we can use the induction on $\mathbf v$, and the condition c) is verified. Therefore $(Q,\mathbf v,\mathbf w)$ has the property $(P_1)$ by Theorem \ref{thm: property P}.

Next let us assume that $\mathbf w\ge \mathbf v$, we are going to show that $(Q,\mathbf v,\mathbf w)$ has the property $(\widetilde P_1)$, by verifying the conditions a), b), and c) in Theorem \ref{thm: property P tilde}. The condition a) has been verified above for arbitrary nonzero $\mathbf v$ and nonzero $\mathbf w$. According to Remark \ref{rmk: bound of dim of nilpotent cone}, the condition b) follows from the inequality $\mathbf v\cdot(\mathbf w-\mathbf v)+2\ge 1$, which is obviously true since $\mathbf w-\mathbf v\ge 0$. Since every auxiliary quiver $Q_\tau$ is disjoint union of Jordan quivers, and $\mathbf v_{\tau,i}<\mathbf v$ and $\mathbf w_{\tau,i}=\mathbf w$ for all $i\in Q_{\tau,0}$, the previous step implies that $(Q_\tau,\mathbf v_\tau,\mathbf w_\tau)$ has the property $(P_1)$, this verifies the condition c). Therefore $(Q,\mathbf v,\mathbf w)$ has the property $(\widetilde P_1)$ by Theorem \ref{thm: property P tilde}.
\end{proof}

\subsection{Natural maps between function rings on jets}\label{sec:semiclassicalAnalysis}

In this section we state formally two of the main theorems in this paper, Theorems \ref{thm:semiclassical-injectivity} and \ref{thm:semiclassical-injectivity2}, which were announced in Section \ref{sec:intro}, as well as Theorem \ref{thm:semiclassical-injectivity3}. Theorem \ref{thm:semiclassical-injectivity} states the conditions for injectivity between the ring of $J_n\mathrm{GL}(\mathbf{v})$-invariant functions on $n$-jets of $\mu^{-1}_{\mathsf{s}}(0)$, and global sections of the structure sheaf on the $n$-jet of the Nakajima quiver super-variety $J_n\mathcal{M}(\mathbf{v},\mathbf{w}|\mathbf{u})$ (and similarly for the extended moment map and extended Nakajima quiver super-variety). Then, Theorem \ref{thm:semiclassical-injectivity2} states the conditions for an isomorphism between the ring of functions on the quiver super-variety $\widetilde{\mathcal{M}}^0(\mathbf{v},\mathbf{w}|\mathbf{u})$ and global sections of the structure sheaf on $\widetilde{\mathcal{M}}(\mathbf{v},\mathbf{w}|\mathbf{u})$, while Theorem \ref{thm:semiclassical-injectivity3} is an analogue for \emph{good} ADE quivers.

\bigskip

\noindent In order to state and prove Theorems \ref{thm:semiclassical-injectivity}-\ref{thm:semiclassical-injectivity3}, we first need to establish some preliminary facts.

\subsubsection{Preliminary results}

From \cite{Grothendieck1964tudeLD}, a super-scheme $X$ is said to have property $(S_i)$ if 
\begin{align*}
    \depth(\mathcal O_{X,Z})\ge \min(i,\codim_X (Z))
\end{align*}
for every irreducible closed subset $Z$ of $X$. Here $\depth(\mathcal O_{X,Z})$ is defined as the minimal $j$ such that the local cohomology sheaf $\mathcal H_Z^j(\mathcal O_X)$ is nonzero. The local cohomology is defined in \cite[Section 1]{hartshorne2006local}. Using \cite[Theorem 2.3]{hartshorne2006local}, on an open affine subscheme $U=\Spec A$ of $X$, the local cohomology sheaf $\mathcal H_Z^j(\mathcal O_X)$ can be presented as follows. Take $f_1,\cdots,f_m\in A^0$ such that $Z=V(\mathrm{rad}(f_1,\cdots,f_m))$ is the variety of the corresponding radical ideal, and let $K^{\bullet}(\mathbf f;A^0)$ be the Koszul complex on $A^0$ associated to $\mathbf f=(f_1,\cdots,f_m)$, where $A^0$ is the even part of $A$, and similarly let $K^{\bullet}(\mathbf f^t;A^0)$ be the Koszul complex on $A^0$ associated to $\mathbf f^t=(f_1^t,\cdots,f_m^t)$. Then there are obvious maps between complexes $K^{\bullet}(\mathbf f^t;A^0)\to K^{\bullet}(\mathbf f^{t+1};A^0)$, and so 
\begin{align}\label{local cohomology complex}
    \mathcal H_Z^j(\mathcal O_X)|_{U}\cong H^j\left(\underset{\substack{\longrightarrow\\ t}}{\lim} \:K^{\bullet}(\mathbf f^t;A^0)\otimes_{A^0} A\right).
\end{align}
We shall denote the complex in the right hand side as $K^{\bullet}(\mathbf f^{\infty};A)$.

\medskip

\begin{lem}[Super-scheme analog of {\cite[Theorem 7.1]{crawley2003normality}}]\label{lem: quotient has (Si)}
Let $A$ be a finitely-generated super-commutative algebra, and assume that $X=\Spec A$ has a reductive group $G$ action. Suppose that $X$ and an open subset $U$ of $X\sslash G$ have property $(S_i)$. If $q^{-1}(X\sslash G\setminus U)$ has codimension at least $i$ in $X$, where $q: X\to X\sslash G$ is the quotient map, then $X\sslash G$ has property $(S_i)$.
\end{lem}

\medskip

\begin{proof}
The proof is essentially the same as \cite[Theorem 7.1]{crawley2003normality}, we only need to make sure that (1) the natural map $H^j_{Y}(X\sslash G)\to H^j_{q^{-1}(Y)}(X)$ is injective for every closed subset $Y$ in $X\sslash G$, where $q:X\to X\sslash G$ is the quotient map, (2) the excision isomorphism \cite[Proposition 1.3]{hartshorne2006local} $H^j_{Z\cap U}(X\sslash G)\cong H^j_{Z\cap U}(U)$ can be applied to super-scheme. The point (2) is obvious, since the \cite[Proposition 1.3]{hartshorne2006local} is true for any topological space with a sheaf of abelian groups, in particular holds for the topological space $X\sslash G$ and sheaf $\mathcal O_X$. For the point (1), \cite[Theorem 7.1]{crawley2003normality} uses \cite[Corollary 6.8]{hochster1974rings}; here we modify the argument in \cite{hochster1974rings}. Notice that both $X$ and $X\sslash G$ are super-schemes over the base bosonic scheme $\Spec (A^{0})^G$, and since $G$ is reductive, the embedding map $A^G\hookrightarrow A$ splits. Let $Y=V(\mathrm{rad}(f_1,\cdots,f_m))$ where $f_i\in (A^0)^G$, then $q^{-1}(Y)=V(\mathrm{rad}(q^*(f_1),\cdots,q^*(f_m)))$, therefore the natural map $H^j_{Y}(X\sslash G)\to H^j_{q^{-1}(Y)}(X)$ is induced from the injective map between complexes \eqref{local cohomology complex}:
\begin{align*}
    \underset{\substack{\longrightarrow\\ t}}{\lim} \:K^{\bullet}(\mathbf f^t;(A^0)^G)\otimes_{(A^0)^G} A^G\longrightarrow \underset{\substack{\longrightarrow\\ t}}{\lim} \:K^{\bullet}(\mathbf f^t;(A^0)^G)\otimes_{(A^0)^G} A.
\end{align*}
The above map between complexes induces an injective map between cohomologies because $A^G\hookrightarrow A$ splits. This finishes the proof.
\end{proof}

\bigskip

\begin{lem}\label{lem: property S_i}
Let $A=A^0\oplus A^1$ be a Noetherian super-commutative algebra with even (resp. odd) component $A^0$ (resp. $A^1$), define $\mathfrak n_A$ to be the ideal generated by $A^1$, and define $A^{\mathrm{bos}}:=A/\mathfrak n_A$. Suppose that $A$ satisfies the following condition:
\begin{itemize}
    \item[$(*)$] $\mathfrak n_A^i/\mathfrak n_A^{i+1}$ is a locally free $A^{\mathrm{bos}}$-module of finite rank,
\end{itemize}
then $\Spec A$ has property $(S_i)$ if and only if $\Spec A^{\mathrm{bos}}$ has property $(S_i)$.
\end{lem}

\medskip 

\begin{proof}
Let $X=\Spec A$ and $X^{\mathrm{bos}}=\Spec A^{\mathrm{bos}}$, then it is enough to show that $$\depth(\mathcal O_{X,Z})=\depth(\mathcal O_{X^{\mathrm{bos}},Z{^{\mathrm{bos}}}})$$ for every irreducible closed subset $Z$ of $X$ {and $Z^{\mathrm{bos}}$ irreducible closed subset of $X^{\mathrm{bos}}$}. Equivalently, it suffices to show that for every fixed $j$, $\mathcal H_Z^k(\mathcal O_X)=0$ for all $k<j$ if and only if $\mathcal H_Z^k(\mathcal O_{X^{\mathrm{bos}}})=0$ for all $k<j$. We prove this assertion by induction on $r:=$maximal integer such that $\mathfrak n_A^r\neq 0$. The initial case $r=0$ is trivial since $A=A^{\mathrm{bos}}$ in this case, so let us assume $r>0$.

Since $\mathfrak n_A^i/\mathfrak n_A^{i+1}$ is a locally free $A^{\mathrm{bos}}$-module of finite rank, then $\mathcal H_Z^k(\mathcal O_{X^{\mathrm{bos}}})=0$ if and only if $\mathcal H_Z^k(\mathfrak n_A^i/\mathfrak n_A^{i+1})=0$. Thus $\mathcal H_Z^k(\mathcal O_{X^{\mathrm{bos}}})=0$ if and only if $\mathcal H_Z^k(\mathrm{gr}\: \mathcal O_X)=0$, where $$\mathrm{gr}\: \mathcal O_X=\bigoplus_{i\ge 0}\mathfrak n_A^i/\mathfrak n_A^{i+1}~.$$ Now if $\mathcal H_Z^k(\mathrm{gr}\: \mathcal O_X)=0$ for all $k<j$, then $\mathcal H_Z^k(\mathcal O_X)=0$ for all $k<j$, because the former is an associated graded of the latter. 

Conversely, if $\mathcal H_Z^k(\mathcal O_X)=0$ for all $k<j$, then $\mathcal H_Z^k(A/\mathfrak n_A^r)=\mathcal H_Z^{k-1}(\mathfrak n_A^r)$ for all $k<j$ from long exact sequences in cohomology. Since $\mathcal H_Z^{k}(\mathfrak n_A^r)=0$ for $k<0$, we have $\mathcal H_Z^0(A/\mathfrak n_A^r)=0$. Then it follows from induction on $r$ that $\mathcal H_Z^0(\mathcal O_{X^{\mathrm{bos}}})=0$. Thus we obtain that $\mathcal H_Z^{0}(\mathfrak n_A^r)=0$, whence it implies that $\mathcal H_Z^1(A/\mathfrak n_A^r)=0$. It follows from induction on $r$ that $\mathcal H_Z^1(\mathcal O_{X^{\mathrm{bos}}})=0$. Ultimately, we conclude that $\mathcal H_Z^k(\mathcal O_{X^{\mathrm{bos}}})=0$ for all $k<j$. This completes the proof.
\end{proof}

\bigskip

\begin{lem}\label{lem: restrict to open is injective}
Let $A$ be a Noetherian super-commutative algebra with the same assumption as the Lemma \ref{lem: property S_i}, and let $U$ be an open and dense subset in $X=\Spec A$. If $X$ has property $(S_1)$, then the natural map $A\to \mathcal O_X(U)$ is injective. If $X$ has property $(S_2)$, and $\codim_X(Z)\ge 2$, then the natural map $A\to \mathcal O_X(U)$ is an isomorphism.
\end{lem}

\begin{proof}
Let $Z=X\setminus U$, then there is a long exact sequence
\begin{align*}
    0\longrightarrow H^0_Z(\mathcal O_X) \longrightarrow H^0(\mathcal O_X)\longrightarrow H^0_U(\mathcal O_X)\longrightarrow H^1_Z(\mathcal O_X)\cdots
\end{align*}
Note that $H^0(\mathcal O_X)=A$ and $H^0_U(\mathcal O_X)=\mathcal O_X(U)$. If $X$ has property $(S_1)$, then $H^0_Z(\mathcal O_X)$ vanishes and $\codim_X(Z)\ge 1$. In this case $A\to \mathcal O_X(U)$ is injective. If $X$ has property $(S_2)$, and $\codim_X(Z)\ge 2$, then both $H^0_Z(\mathcal O_X)$ and $H^1_Z(\mathcal O_X)$ vanish. In this case $A\to \mathcal O_X(U)$ is an isomorphism.
\end{proof}

\subsubsection{Main results}

\medskip

\begin{theorembox}
\begin{thm}\label{thm:semiclassical-injectivity}
If $\mu$ is flat and $\mu^{-1}(0)$ is reduced and irreducible and has rational singularities (i.e.\  the property $(P_1)$ in Proposition \ref{prop: various jet flatness}), then the natural maps 
\begin{align*}
    \mathbb C[J_n\mu_{\mathsf s}^{-1}(0)]^{J_n\GL(\mathbf v)}\longrightarrow \Gamma(J_n\mathcal M(\mathbf v,\mathbf w|\mathbf u),\mathcal O_{J_n\mathcal M(\mathbf v,\mathbf w|\mathbf u)})~,
\end{align*}
are injective for all $n\in \mathbb Z_{\ge 0}$. If $\widetilde\mu$ is flat, and $\widetilde\mu^{-1}(0)$ is reduced and irreducible and has rational singularities (i.e.\  the property $(\widetilde P_1)$ in Proposition \ref{prop: various jet flatness}), then the natural maps 
\begin{align*}
    \mathbb C[J_n\widetilde\mu_{\mathsf s}^{-1}(0)]^{J_n\GL(\mathbf v)}\longrightarrow \Gamma(J_n\widetilde{\mathcal M}(\mathbf v,\mathbf w|\mathbf u),\mathcal O_{J_n\widetilde{\mathcal M}(\mathbf v,\mathbf w|\mathbf u)})~,
\end{align*}
are injective for all $n\in \mathbb Z_{\ge 0}$.
\end{thm}
\end{theorembox}

\medskip

\begin{proof}
If the quiver $(Q,\mathbf v,\mathbf w)$ has the property $(P_1)$ in Proposition \ref{prop: various jet flatness}, then $J_n\mu^{-1}(0)$ is reduced and irreducible, in particular it has property $(S_1)$. According to Proposition \ref{prop: J_n-flatness of super moment maps}, $J_n\mu_{\mathsf s}$ is also flat. 

Applying Lemma \ref{lem: regular sequence} to $A=\mathbb C[T^*\mathrm{Rep}(\mathbf v,\mathbf w|\mathbf u)]$ and a basis $\{a_i\}_{1\le i\le \dim J_n\gl(\mathbf v)}$ of $J_n\gl(\mathbf v)$, we see that the super-commutative algebra $ \mathbb C[J_n\mu_{\mathsf s}^{-1}(0)]$ satisfies the assumption in Lemma \ref{lem: property S_i}, and thus $J_n\mu_{\mathsf s}^{-1}(0)$ also has property $(S_1)$. 

Applying Lemma \ref{lem: restrict to open is injective} to $X=J_n\mu_{\mathsf s}^{-1}(0)$ and $U=J_n\mu_{\mathsf s}^{-1}(0)^{\mathrm{st}}$, we see that the natural map $ \mathbb C[J_n\mu_{\mathsf s}^{-1}(0)]\to \mathcal O_{J_n\mu_{\mathsf s}^{-1}(0)}(J_n\mu_{\mathsf s}^{-1}(0)^{\mathrm{st}})$ is injective. Taking $J_n\GL(\mathbf v)$-invariant, we see that the natural map 
\begin{align*}
    \mathbb C[J_n\mu_{\mathsf s}^{-1}(0)]^{J_n\GL(\mathbf v)}\longrightarrow \Gamma(J_n\mathcal M(\mathbf v,\mathbf w|\mathbf u),\mathcal O_{J_n\mathcal M(\mathbf v,\mathbf w|\mathbf u)}),
\end{align*}
is injective. This proves the first statement. The second statement follows from similar argument, for which we omit the details.
\end{proof}

\begin{theorembox}
\begin{thm}\label{thm:semiclassical-injectivity2}
If $\mu^{-1}(0)^{\mathrm{reg}}$ is nonempty, then $\widetilde{\mathcal M}^0(\mathbf v,\mathbf w|\mathbf u)$ has property $(S_2)$. Moreover, the natural map
\begin{align*}
    \mathbb C[\widetilde{\mathcal M}^0(\mathbf v,\mathbf w|\mathbf u)]\longrightarrow \Gamma(\widetilde{\mathcal M}(\mathbf v,\mathbf w|\mathbf u),\mathcal O_{\widetilde{\mathcal M}(\mathbf v,\mathbf w|\mathbf u)})
\end{align*}
is an isomorphism.
\end{thm}
\end{theorembox}

\begin{proof}
By Proposition \ref{prop: simple locus}, the non-emptiness of $\mu^{-1}(0)^{\mathrm{reg}}$ implies that $\mu$ is flat and thus $\mu^{-1}(\mathcal Z)$ is a locally complete intersection. In particular $\mu^{-1}(\mathcal Z)$ has property $(S_2)$, and consequently so does $\mu^{-1}_{\mathsf s}(\mathcal Z)$ by Lemma \ref{lem: property S_i}. 

Note that $\mu^{-1}_{\mathsf s}(\mathcal Z)^{\mathrm{reg}}$ is smooth and $\mu^{-1}_{\mathsf s}(\mathcal Z)^{\mathrm{reg}}\to \widetilde{\mathcal M}^0(\mathbf v,\mathbf w|\mathbf u)^{\mathrm{reg}}$ is a principal $\GL(\mathbf v)$-bundle, thus $\widetilde{\mathcal M}^0(\mathbf v,\mathbf w|\mathbf u)^{\mathrm{reg}}$ is smooth. In particular $\widetilde{\mathcal M}^0(\mathbf v,\mathbf w|\mathbf u)^{\mathrm{reg}}$ has property $(S_2)$. We claim that $\mu^{-1}(\mathcal Z)\setminus \mu^{-1}(\mathcal Z)^{\mathrm{reg}}$ has codimension at least $2$ in $\mu^{-1}(\mathcal Z)$. In fact, $\mu^{-1}(\mathcal Z)^{\mathrm{reg}}$ is invariant under the $\mathbb C^{\times}$-action that scales $T^*\mathrm{Rep}(\mathbf v,\mathbf w)$ with weight $1$, then the non-emptiness of $\mu^{-1}(0)^{\mathrm{reg}}$ implies that $\mu^{-1}(\lambda)^{\mathrm{reg}}$ is nonempty for all $\lambda\in\mathcal Z$. In addition, $\mu^{-1}(\lambda)^{\mathrm{reg}}=\mu^{-1}(\lambda)$ for $\lambda$ in an open dense subset of $\mathcal Z$. Thus $\mu^{-1}(\mathcal Z)\setminus \mu^{-1}(\mathcal Z)^{\mathrm{reg}}$ has codimension at least $2$ in $\mu^{-1}(\mathcal Z)$. 

We conclude that $\widetilde{\mathcal M}^0(\mathbf v,\mathbf w|\mathbf u)$ has property $(S_2)$ by applying the Lemma \ref{lem: quotient has (Si)} to $X=\mu_{\mathsf s}^{-1}(\mathcal Z)$ and $U=\widetilde{\mathcal M}^0(\mathbf v,\mathbf w|\mathbf u)^{\mathrm{reg}}$. 

Similarly we can show that $\widetilde{\mathcal M}^0(\mathbf v,\mathbf w|\mathbf u)\setminus \widetilde{\mathcal M}^0(\mathbf v,\mathbf w|\mathbf u)^{\mathrm{reg}}$ has codimension at least $2$ in $\widetilde{\mathcal M}^0(\mathbf v,\mathbf w|\mathbf u)$, then Lemma \ref{lem: restrict to open is injective} implies that the restriction map
\begin{align*}
    \mathbb C[\widetilde{\mathcal M}^0(\mathbf v,\mathbf w|\mathbf u)]\longrightarrow \Gamma(\widetilde{\mathcal M}^0(\mathbf v,\mathbf w|\mathbf u)^{\mathrm{reg}},\mathcal O_{\widetilde{\mathcal M}^0(\mathbf v,\mathbf w|\mathbf u)})
\end{align*}
is an isomorphism. The projection $p:\widetilde{\mathcal M}(\mathbf v,\mathbf w|\mathbf u)\to \widetilde{\mathcal M}^0(\mathbf v,\mathbf w|\mathbf u)$ induces an isomorphism $\widetilde{\mathcal M}^0(\mathbf v,\mathbf w|\mathbf u)^{\mathrm{reg}}\cong \widetilde{\mathcal M}(\mathbf v,\mathbf w|\mathbf u)^{\mathrm{reg}}$, thus the composition $$\mathbb C[\widetilde{\mathcal M}^0(\mathbf v,\mathbf w|\mathbf u)]\longrightarrow \Gamma(\widetilde{\mathcal M}(\mathbf v,\mathbf w|\mathbf u),\mathcal O_{\widetilde{\mathcal M}(\mathbf v,\mathbf w|\mathbf u)})\hookrightarrow \Gamma(\widetilde{\mathcal M}(\mathbf v,\mathbf w|\mathbf u)^{\mathrm{reg}},\mathcal O_{\widetilde{\mathcal M}(\mathbf v,\mathbf w|\mathbf u)})$$ is an isomorphism. This implies that $\mathbb C[\widetilde{\mathcal M}^0(\mathbf v,\mathbf w|\mathbf u)]\to\Gamma(\widetilde{\mathcal M}(\mathbf v,\mathbf w|\mathbf u),\mathcal O_{\widetilde{\mathcal M}(\mathbf v,\mathbf w|\mathbf u)})$ is an isomorphism.
\end{proof}

\subsubsection{Canonical quiver super-variety associated to a good quiver}

Throughout this subsection, we assume the following goodness condition on the quiver $Q$:
\begin{align*}
    \mathbf w^{\mathsf f}:=\mathbf w-\mathsf C\mathbf v\in \mathbb Z^{Q_0}_{\ge 0}.
\end{align*}
We say that $(Q,\mathbf v,\mathbf w|\mathbf w^{\mathsf f})$ is the canonical super quiver data associated to the quiver data $(Q,\mathbf v,\mathbf w)$.

\medskip 

\begin{theorembox}
\begin{thm}\label{thm:semiclassical-injectivity3}
Let $Q$ be a finite or affine ADE quiver (including the case $\widehat{A}_0$), and assume that $\mu^{-1}(0)^{\mathrm{reg}}$ is nonempty, then $\mathcal M^0(\mathbf v,\mathbf w|\mathbf w^{\mathsf f})$ has property $(S_2)$. Moreover, the natural map
\begin{align*}
    \mathbb C[\mathcal M^0(\mathbf v,\mathbf w|\mathbf w^{\mathsf f})]\longrightarrow \Gamma(\mathcal M(\mathbf v,\mathbf w|\mathbf w^{\mathsf f}),\mathcal O_{\mathcal M(\mathbf v,\mathbf w|\mathbf w^{\mathsf f})})
\end{align*}
is an isomorphism.
\end{thm}
\end{theorembox}

\medskip 

\begin{proof}
We say that $(Q,\mathbf v,\mathbf w)$ is Kleinian if the associated unframed quiver $Q^+$ is an extended Dynkin quiver, and $\mathbf v^+=\delta$, the minimal positive imaginary root for the extended Dynkin quiver. It is easy to see that  $(Q,\mathbf v,\mathbf w)$ is Kleinian if and only if $(Q^+,\mathbf v^+)$ is nearly Kleinian in the sense of \cite[Section 8]{crawley2003normality}. In the case that $(Q,\mathbf v,\mathbf w)$ is Kleinian, we compute that $\mathbf w^{\mathsf f}=0$, then the result is well-known, see for example \cite[Section 8]{crawley2003normality}.

For the general cases, we proceed using \cite[Lemma 8.2]{crawley2003normality}, which says that if $$\tau=(1,\mathbf v^{(0)+}; k_1,\beta^{(1)}; \cdots; k_r,\beta^{(r)}) $$ is a representation type, then at least one of the following holds:
\begin{itemize}
    \item[(1)] $\sum_{t=1}^r k_t^2<\sum_{i\in Q_0}  v_i^2$.
    \item[(2)] $\dim \mu^{-1}(0)-\dim \mu^{-1}(0)_\tau\ge 2$, where $\mu^{-1}(0)_\tau$ is the locally closed subset which consists of representations of type $\tau$.
\end{itemize}
Assume that we are not in the Kleinian case. Using the étale local structure (Proposition \ref{prop: étale slice for super moment map}), for a point $x\in \mathcal M^0(\mathbf v,\mathbf w|\mathbf w^{\mathsf f})$ of representation type $\tau$, if (1) holds, we can use the induction. If (2) holds, then we apply Lemma \ref{lem: quotient has (Si)} to $X=\mu^{-1}_{\mathsf s}(0)$ and $U=\mathcal M^0(\mathbf v,\mathbf w|\mathbf w^{\mathsf f})\setminus \overline{\mathcal M^0(\mathbf v,\mathbf w|\mathbf w^{\mathsf f})_\tau}$. This completes the proof.
\end{proof}

\section{Chiralization of Quiver Representations} \label{sec:ChiralizationQuiverRep}

In this section, we construct a vertex superalgebra $\CV[Q]$ associated with the quiver data $(Q,\mathbf{v},\mathbf{w})$ and the corresponding extended, affine Nakajima quiver super-variety $\widetilde{\mathcal{M}}(\mathbf{v},\mathbf{w}|\mathbf{u})$ from Section \ref{sec:GeomQuiverVarieties}. Let us first outline the general idea, noting that some relevant vertex algebra background is collected in Appendix \ref{appendix:VA-PVA}. 

\medskip

Given a quiver super-variety $\mathcal{M}(\mathbf v,\mathbf w|\mathbf u)$, we  define an associated vertex algebra from the quadruple $(G,M_{\mathsf s},\omega_{\mathsf s},\mu_{\mathsf s})$, with $(M_{\mathsf s},\omega_{\mathsf s})$ the standard super symplectic vector space and symplectic structure defined by $T^*\mathrm{Rep}(\mathbf v,\mathbf w|\mathbf u)$, that define a super $\beta\gamma$ system $\mathcal{V}_{\beta\gamma}$, with a chiral moment map  
\begin{equation}\label{eq:momentmap-CDO}
    \mu_{\mathsf s}^{\mathrm{ch}}:V^{k}(\mathfrak{g})\to\mathcal{V}_{\beta\gamma}(M_{\mathsf s},\omega_{\mathsf s})~.
\end{equation}
$(G,\mu_{\mathsf s})$ are the reductive group and moment map discussed in Section \ref{sec:GeomQuiverVarieties}, and $V^{k}(\mathfrak{g})$ is the level $k$ affine vertex algebra associated with the corresponding Lie group $\mathfrak{g}$. Then a vertex algebra $\mathcal{V}[G,M_{\mathsf s},\omega_{\mathsf s},\mu^{\mathrm{ch}}_{\mathsf s}]$ associated to the quiver super-variety $\mathcal{M}(\mathbf v,\mathbf w|\mathbf u)$ can be naively defined by BRST cohomology with respect to a BRST charge $\mathcal Q$, from the associated vertex super-algebra $\mathcal{V}_{\beta\gamma} (M_{\mathsf s},\omega_{\mathsf s})$, modified consistently by introducing a $\mathsf b\mathsf c$-system and a Heisenberg vertex operator algebra associated to the quiver. This modification is necessary in order for the BRST charge $\mathcal Q$ to be nilpotent $\mathcal Q^2= 0$ and construct the BRST cohomology. We review this construction below.

\subsection{General construction}\label{sec:generalconstructionSVA}

We start by collecting the building blocks employed in the construction of a vertex superalgebra $\mathcal{V}^\#$ in this section. Let $\mathfrak{g}$ be a (complex) reductive Lie algebra, $G$ is the corresponding Lie group, and $(M,R)$ symplectic representations of $G$. Additionally, let $Z$ be a vector space, with $\dim Z=m$ and symmetric inner product $(\cdot,\cdot)$. With this data, one can define a vertex superalgebra $\mathcal{V}^\#[M,R;Z]$ to be 
\begin{equation}\label{eqdef:Vhash-generic}
    \mathcal{V}^\#[M,R;Z]:= \CV_{\beta\gamma}[M]\otimes\CV_{\Phi\Psi}[\Pi R]\otimes \mathcal{H}^{(\cdot,\cdot)}[Z] ~,
\end{equation}
where
\begin{itemize}
    \item the symplectic boson vertex operator algebra formed $\CV_{\beta\gamma}[M]$ is strongly generated by a system of $\beta\gamma$ fields $$(\beta_i (z),\gamma_i(z))_{i=1,\ldots,\dim M}~$$ with mode expansions $\beta_i(z)=\sum_{n\in\BZ}\beta_{i,(n)}z^{-n-1}$ and $\gamma_i(z)=\sum_{n\in\BZ}\gamma_{i,(n)}z^{-n}$, and where the vacuum module vector space of $\CV_{\beta\gamma}[M]$ 
    \begin{equation}
        [\beta_{i,(-n_\beta)},\gamma_{i,(-n_\gamma)} \, | \, n_\beta,n_\gamma\in \mathbb{Z}_{\geq 1}, i=1,\ldots ,\dim M]|0\rangle ~
    \end{equation}
    is generated by the action of the modes $\beta_{i,(n)}$, $\gamma_{i,(n)}$ with $n\in\mathbb{Z}_{\leq -1}$ on the Fock vacuum $|0\rangle$. The vertex algebra structure of $\CV_{\beta\gamma}[M]$ is given by the operator product expansion (OPE) of the fields 
    \begin{equation}
        \beta_i(z)\gamma_j(w)\sim \frac{\delta_{ij}}{z-w} ~,
    \end{equation}
    where we use the symbol "$\sim$" inside OPEs to denote an equivalence up to non-singular terms. $\CV_{\beta\gamma}[M]$ is also referred to as the vertex algebra of chiral differential operators (CDO) on $T^\ast V$, denoted $D^{\mathrm{ch}}(T^*V)$, with $V\cong \mathbb{C}^{\dim M}$.
    \item the free fermion vertex operator algebra $\CV_{\Phi\Psi}[\Pi R]$ is similarly strongly generated by a system of odd fields $(\Phi_i(z),\Psi_i(z))_{i=1,\ldots,\dim R}$ on the odd symplectic vector space $\Pi R$, with the OPE
    \begin{equation}
        \Phi_i(z)\Psi_j(w)\sim \frac{\delta_{ij}}{z-w} ~.
    \end{equation}
    \item the Heisenberg vertex algebra $\mathcal{H}^{(\cdot,\cdot)}[Z]$ on $Z=\oplus_{i=1}^m \mathbb{C}h^i$, with vacuum module vector space
    \begin{equation}
      [h^1_{(-n_1)},\ldots ,h^m_{(-n_m)} \, | \, n_i\in \mathbb{Z}_{\geq 1}, i=1,\ldots ,m]|0\rangle ~
    \end{equation}
    a direct sum of vector spaces of polynomials in the mode coefficients of the Heisenberg fields $h^i(z)=\sum_{n\in\mathbb{Z}}h^i_{(n)}z^{-n-1}$, with $1\leq i\leq m$ and modes $h^i_{(-n)},\, n\in\BZ_{\geq 1},$ acting on the vacuum $|0\rangle$ of the Heiseneberg Fock space, and with vertex algebra structure given by the OPE 
    \begin{equation}
        h^i(z)h^j(w)\sim \frac{(h^i,h^j)}{(z-w)^2} ~.
    \end{equation}
\end{itemize}
Additionally, to the reductive Lie algebra $\mathfrak g$, we associate a fermionic $\mathsf b \mathsf c$-ghost system $\CV_{\mathsf b \mathsf c}[\mathfrak g]$ with fields $(b_i(z),c_i(z))_{i=1,\ldots , \dim \mathfrak g}$ in the adjoint representation of $G$, with the OPE
\begin{equation}
    b_i(z)c_j(w)\sim \frac{\delta_{ij}}{z-w}~.
\end{equation}
$\CV_{\mathsf b \mathsf c}[\mathfrak g]$ is a system of chiral differential operators $D^{\mathrm{ch}}(\Pi T^*\mathfrak g)$ on the odd symplectic vector space $\Pi T^*\mathfrak g$.

\subsection{BRST reduction}\label{sec:BRST reduction}

Given a vertex superalgebra $\CV^\#$ of the type described above, there is a corresponding vertex superalgebra $\CV:=\CV^\#\Kquotient G$ which is constructed mathematically by BRST reduction \cite{Feigin_1984,Frenkel_1986,Voronov_1993}. We recall below the presentation in \cite{Coman:2023xcq}. For the reductive Lie algebra $\mathfrak{g}$, let $\kappa$ be a $\mathfrak{g}$-invariant symmetric inner form on $\mathfrak{g}$, and let $\kappa_{\mathfrak{g}}$ be the Killing form on $\mathfrak{g}$. Then the vector space of the current algebra $V^\kappa[\mathfrak g]$ is  
\begin{equation}
    \mathbb C[J_\infty\mathfrak g]=\mathbb C[E^\alpha_{(-n)}\:|\:n\in \mathbb Z_{\ge 1}, E^\alpha\in\text{a basis of }\mathfrak g]|0\rangle
\end{equation} 
endowed with the OPE
\begin{align}
    E^{\alpha}(z)E^\beta(w)\sim \frac{\kappa^{\alpha,\beta}}{(z-w)^2}+\frac{f^{\alpha\beta}_\gamma E^\gamma(w)}{z-w}\;,
\end{align}
where $\kappa^{\alpha,\beta}=\kappa(E^\alpha,E^\beta)$ is the matrix element of the inner form $\kappa$, and $f^{\alpha\beta}_\gamma$ defined by $f^{\alpha\beta}_\gamma E^\gamma=[E^\alpha,E^\beta]$ are the structure constants of $\mathfrak g$. Suppose now that the vertex superalgebra $V=\CV^\#$ is such that there exists a vertex superalgebra map
\begin{align}\label{map:VSA-JV}
    J_V: V^{-\kappa_{\mathfrak g}}[\mathfrak g]\to V \;.
\end{align}
Then we define the vertex superalgebra 
\begin{align}\label{BRST complex}
    \widetilde C(\mathfrak g, V):=V\otimes \CV_{\mathsf b\mathsf c}[\mathfrak g] \;,
\end{align}
where $\CV_{\mathsf b\mathsf c}[\mathfrak g]$ is the $\mathsf b\mathsf c$-system of $\mathfrak g$, which admits a vertex superalgebra map
\begin{align}
    J_{\mathsf b\mathsf c}: V^{\kappa_{\mathfrak g}}[\mathfrak g]\to \CV_{\mathsf b\mathsf c}[\mathfrak g]\;, \quad J_{\mathsf b\mathsf c}(E^\alpha)=f^{\alpha\beta}{}_{\gamma}\mathsf c_{\beta}\mathsf b^\gamma.
\end{align}
$\CV_{\mathsf b\mathsf c}[\mathfrak g]$ is graded by ghost numbers, i.e.\  $\deg \mathsf c=+1,\deg\mathsf b=-1$, so the image of $J$ is contained in degree zero. Consider the BRST current $\JBRST$ (see also \cite{Karabali:1989dk, Beem:2013sza}), which has ghost number $1$  
\begin{align}
    \JBRST:=\mathsf c_\alpha\left(J_V(E^\alpha)+\frac{1}{2}J_{\mathsf b\mathsf c}(E^\alpha)\right)\in  \widetilde C^1(\mathfrak g, V)\;. \label{def:BRSTcurrent}
\end{align}
The operator 
\begin{equation}
    \mathcal{Q}:=\JBRST_{(0)} := \oint \frac{dz}{2\pi i} \, \JBRST(z) \label{Q_BRST}
\end{equation} 
acts on $\widetilde C(\mathfrak g, V)$, and squares to zero $\mathcal{Q}^2=0$. Therefore $(\widetilde C(\mathfrak g, V), \mathcal{Q})$ is a chain complex of vertex superalgebra, called the BRST complex, which is graded by ghost numbers. Now the linear map 
\begin{align}\label{level zero currents in BRST complex}
    E^\alpha\mapsto J(E^\alpha):=\mathcal{Q}(\mathsf b^\alpha)=J_V(E^\alpha)+J_{\mathsf b\mathsf c}(E^\alpha) \;,
\end{align}
is a vertex superalgebra chain complex map 
\begin{equation} 
J: (V^0(\mathfrak g),0)\to (\widetilde C(\mathfrak g, V), \mathcal{Q})
\end{equation}
where the LHS is endowed with trivial differential. In particular the Lie algebra $\mathfrak g$ acts on the complex $(\widetilde C(\mathfrak g, V), \mathcal{Q})$ by $E^\alpha\mapsto J(E^\alpha)_{(0)}$. The relative BRST complex is defined to be the subspace 
\begin{align}
    C(\mathfrak g, V):=\{v\in \widetilde C(\mathfrak g, V)\:|\: J(E^\alpha)_{(0)}v=\mathsf b^\alpha_{(0)}v=0,\forall \alpha\}\;,
\end{align}
which is a sub-complex because $[J(E^\alpha)_{(0)},\mathcal{Q}]=0$ and $[\mathsf b^\alpha_{(0)},\mathcal{Q}]=J(E^\alpha)_{(0)}$. In terms of this complex, the relative BRST cohomology is defined to be
\begin{align}
    H^{\infty/2+\bullet}_{\mathrm{BRST}}(\mathfrak g,V):=H^{\bullet}(C(\mathfrak g, V), \mathcal{Q})\;.
\end{align}
This is a $\mathbb Z$-graded vertex superalgebra, with
\begin{equation}\label{eq:quiverVA-0}
    \CV:=H^{\infty/2+0}_{\mathrm{BRST}}(\mathfrak g,V) ~.
\end{equation}

\subsection{Quiver setting}\label{sec:VSA-quiverSetting}

Let us now recall the quiver setting from Section \ref{sec:GeomQuiverVarieties}, where we introduced a super-quiver $Q=(Q_0,Q_1)$ with its associated vector spaces $(V,W)$, dimension vector data $(\mathbf v, \mathbf w)$, and representation varieties. Throughout this subsection, we assume that
\begin{align*}
    \mathbf u:= (u_i)_{i\in Q_0}:=\mathbf w-\mathsf C\mathbf v\in \mathbb Z^{Q_0}_{\ge 0}~,
\end{align*}
and we define a $Q_0$-graded vector space $U=\bigoplus_{i\in Q_0}U_i$ such that $\dim U_i= u_i$. Then the role of the reductive Lie algebra $\mathfrak g$ from the previous subsections is played by $$\mathfrak g := \mathfrak{gl}(\mathbf v) := \prod_{i\in Q_0} \mathfrak{gl}(v_i)~,$$ and $G:=\GL(\mathbf v)$. 

\subsubsection{Chiral differential operators}

The role of the representations $(M,R)$ from Section \ref{sec:generalconstructionSVA} is taken here by $\mathrm{Rep}(\mathbf v,\mathbf w|\mathbf u)$, where
\begin{align}\label{def:BRST-cohomology}
    \mathrm{Rep}(\mathbf v,\mathbf w|\mathbf u):=\mathrm{Rep}(\mathbf v,\mathbf w)\oplus\bigoplus_{i\in Q_0}\Pi\Hom(U_i,V_i)~,
\end{align}
and we recall from equation \eqref{def:prequotient} the shorthand notation 
\begin{equation} 
\mathfrak{R}:=T^*\mathrm{Rep}(\mathbf v,\mathbf w |\mathbf u)\times\mathcal Z ~.
\end{equation}
The system of chiral differential operators $\mathcal D^{\mathrm{ch}}(T^*\mathrm{Rep}(\mathbf v,\mathbf w|\mathbf u))$ is now formed by matrix-valued $\beta\gamma$ fields and $\Phi\Psi$ free fermion fields
\begin{itemize}
    \item $\CV_{\beta\gamma}[\mathrm{Rep}(\mathbf v,\mathbf w)]$: bosonic fields $((I_i)^n{}_p,(J_i)^q{}_r)_{i\in Q_0;~\substack{n,r=1,\ldots,v_i \\ p,q=1,\ldots,w_i}}$ with the OPE 
    \begin{equation}
        (I_i)^n{}_p(z)(J_j)^q{}_r(w) \sim \frac{\delta^n_r \delta^q_p \delta_{ij}}{z-w} ~,
    \end{equation}
    and bosonic fields $((X_a)^n{}_p,(Y_a)^q{}_r)_{a\in Q_1;~\substack{n,r=1,\ldots,v_i\\ p,q=1,\ldots,v_{i+1}}}$ with the OPE
    \begin{equation}
        (X_a)^n{}_p(z)(Y_b)^q{}_r(w) \sim \frac{\delta^n_r \delta^q_p \delta_{ab}}{z-w} ~.
    \end{equation}
    \item $\CV_{\Phi\Psi}[\bigoplus_{i\in Q_0}\Pi\Hom(U_i,V_i)]$: fermionic fields $((\Phi_i)^n{}_p,(\Psi_i)^q{}_r)_{i\in Q_0;~\substack{n,r=1,\ldots,v_i \\ p,q=1,\ldots,u_i}}$ with the OPE 
    \begin{equation}
        (\Phi_i)^n{}_p(z)(\Psi_j)^q{}_r(w) \sim \frac{\delta^n_r \delta^q_p \delta_{ij}}{z-w} ~.
    \end{equation}
\end{itemize}
The index convention follows the assignment in Figure \ref{fig:convention}.

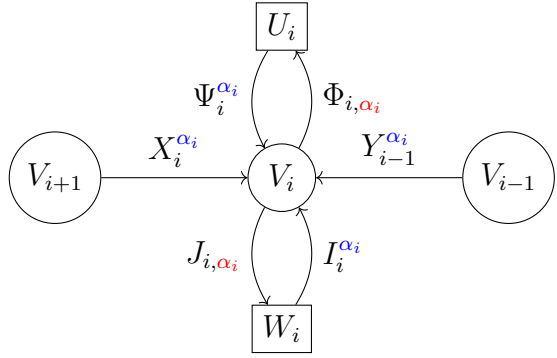
\begin{figure}[ht!]
    \centering
    \begin{tikzpicture}[vertex/.style={draw,circle,minimum size=4mm},fvertex/.style={draw,rectangle,minimum size=4mm}]
          \node[fvertex] (di) at (0,0) {$U_i$};
          \node[vertex] (giplus1) at (-3,-2) {$V_{i+1}$};
          \node[vertex] (gi) at (0,-2) {$V_i$};
          \node[vertex] (giminus1) at (3,-2) {$V_{i-1}$};
          \node[fvertex] (fi) at (0,-4) {$W_i$};
        %
          \draw[->] (giplus1) -- node[above]{$X_{i}^{\textcolor{blue}{\alpha_{i}}}$} (gi);
          \draw[->] (giminus1) -- node[above]{$Y_{i-1}^{\textcolor{blue}{\alpha_{i}}}$} (gi);
          \draw[->, bend right] (fi) to node[right]{$I_{i}^{\textcolor{blue}{\alpha_{i}}}$} (gi);
          \draw[->, bend right] (gi) to node[left]{$J_{i,\textcolor{red}{\alpha_{i}}}$} (fi);
          \draw[->, bend right] (di) to node[left]{$\Psi^{\textcolor{blue}{\alpha_{i}}}_{i}$} (gi);
          \draw[->, bend right] (gi) to node[right]{$\Phi_{i,\textcolor{red}{\alpha_{i}}}$} (di);
    \end{tikzpicture}
    \caption{Index convention at the node $i\in Q_0$, where $\alpha_{i}$ is the index for the subgroup $G_{i}$ of $G$. Other indices are suppressed.}
    \label{fig:convention}
\end{figure}

\subsubsection{A Heisenberg vertex algebra associated to a quiver \texorpdfstring{Q}{Q}}

We define the $U(1)^{Q_0}$ Heisenberg vertex algebra $\CH(\mathcal Z)$ associated to $Q$ by a direct sum of vector spaces of polynomials in mode coefficients of Heisenberg field $h^{i\in Q_0}$, ($h^i_{(-n)}$, $n\in\BZ_{\geq 1}$), acting on the Fock vacuum with the OPE,
\beq
h_i(z) h_j(w)\sim \frac{\mathsf{C}_{ij}}{(z-w)^2}~.
\nonumber
\eeq

\subsubsection{Parent vertex superalgebra \texorpdfstring{$\CV^\#[Q]$}{CVQ}}

We use the vertex algebra $\CH(\mathcal Z)$ as a building block of a \emph{parent} vertex superalgebra of chiral differential operators $\CV^\#[Q]$ associated to the prequotient $\mathfrak R=T^*\mathrm{Rep}(\mathbf v,\mathbf w | \mathbf u)\times  \mathcal Z$  \eqref{def:prequotient} of the extended moment map \eqref{eqdef:extendedsupermomentmap} which defined the Nakajima quiver super-variety \eqref{eqdef:ExtNakajimaQuiverSuperVariety} in Section \ref{sec:GeomQuiverVarieties}
\begin{equation} \label{eqdef:prequotientquiverVOA}
\CV^\#[Q] := \mathcal D^{\mathrm{ch}}(\mathfrak R) := \CV_{\beta\gamma}[\mathrm{Rep}(\mathbf v,\mathbf w)] \otimes \CV_{\Phi\Psi}[\bigoplus_{i\in Q_0}\Pi\Hom(U_i,V_i)] \otimes \CH(\mathcal Z) ~.
\end{equation}

\subsubsection{A \texorpdfstring{$\mathsf b \mathsf c$}{}-system associated to a super quiver}

We define the $\mathsf b \mathsf c$-system $\CV_{\bc}[\mathfrak g]$ associated to a super quiver $Q$ to be the vertex algebra generated by the matrix-valued fermionic fields $((b_i)^n{}_p,(c_i)^q{}_r)_{i\in Q_0;~\substack{n,r=1,\ldots,v_i \\ p,q=1,\ldots,v_i}}$ with the OPE 
\begin{equation}
        (b_i)^n{}_p(z)(c_j)^q{}_r(w) \sim \frac{\delta^n_r \delta^q_p \delta_{ij}}{z-w} ~.
\end{equation}

\subsubsection{BRST reduction in the quiver setting}

In Section \ref{sec:BRST reduction}, we described the construction of a vertex superalgebra $\CV=H^{\infty/2+0}_{\mathrm{BRST}}(\mathfrak g,\CV^\#)$ \eqref{eq:quiverVA-0} via the procedure of BRST reduction from a parent vertex superalgebra $\CV^\#$. Here we specialize and define the quiver vertex superalgebra $\CV[Q]:=\CV(Q,\mathbf v, \mathbf w)$ by BRST reduction from the vertex superalgebra $\CV^\#[Q]$ \eqref{eqdef:prequotientquiverVOA} associated to the extended representation space $\mathfrak R=T^*\mathrm{Rep}(\mathbf v,\mathbf w | \mathbf u)\times  \mathcal Z$  \eqref{def:prequotient} of the quiver. The role of the vertex superalgebra map $J_V$ \eqref{map:VSA-JV} is played by the extended chiral moment map
\begin{align}
\label{eqdef:ChiralExtNakajimaQuiverSuperVariety}
    \widetilde\mu_{\mathsf s}^{\mathrm{ch}}: V^{-\kappa_{\mathfrak g}}[\mathfrak g]\to \CV^\#[Q] \;,
\end{align}
which is a chiralization of the classical counterpart \eqref{eqdef:extendedsupermomentmap}, where the component at node $i\in Q_0$
\begin{align}\label{eq:chiral extended moment map}
    \widetilde\mu^{\mathrm{ch}}_{i}:=J_{\mathfrak{u}(v_i)}-h^i\cdot\mathrm{id}
\end{align}
is defined in \cite{Coman:2023xcq}, analogously to \eqref{extended super moment map}, in terms of the Heisenberg field $h^i$ and an affine $\mathfrak{u}(v_i)$ current of the form (cf.\ \cite{Costello:2018fnz}) 
\begin{align}
\label{def:unBRSTcurrentViShifted_general}
    (J_{\mathfrak{u}({v}_i)})^\alpha{}_\beta &= 
    -\sum_{\mathsf{h}(a)=i} (X_{a})^\alpha{}_\gamma (Y_{a})^{\gamma}{}_\beta
    +\sum_{\mathsf{t}(a)=i} 
    (Y_{a})^\alpha{}_\gamma(X_{a})^{\gamma}{}_\beta-(I_{i})^\alpha{}_\gamma(J_{i})^{\gamma}{}_\beta + (\Psi_{i})^\alpha{}_\gamma(\Phi_{i})^{\gamma}{}_\beta  \;.
\end{align}

\bigskip

\begin{defn}
The quiver vertex superalgebra $\CV[Q]$ associated to $\CV^\#[Q]$ \eqref{eqdef:prequotientquiverVOA} is defined by zeroth relative BRST cohomology 
\begin{equation}\label{eqdef:quiverVSA}
    \CV[Q]:=H^{\infty/2+0}_{\mathrm{BRST}}(\mathfrak g,\CV^\#[Q]) ~.
\end{equation}
\end{defn}
\noindent This is the vertex superalgebra analog of the deformed Nakajima quiver super-variety $\widetilde{\mathcal M}(\mathbf v,\mathbf w|\mathbf u)$ by Hamiltonian reduction in \eqref{eqdef:ExtNakajimaQuiverSuperVariety},
\begin{align}
    \widetilde{\mathcal M}(\mathbf v,\mathbf w|\mathbf u)=\mu^{-1}_{\mathsf s}(\mathcal Z)^{\mathrm{st}}\sslash \GL(\mathbf v)~,
\end{align}
which chiralizes the ring of global sections $\Gamma(\widetilde{\mathcal M}(\mathbf v,\mathbf w|\mathbf u), \mathcal{O}_{\widetilde{\mathcal M}(\mathbf v,\mathbf w|\mathbf u)})$. 

\subsubsection{Modules of the quiver vertex superalgebra}

Having defined the quiver vertex superalgebra $\CV[Q]$ \eqref{eqdef:quiverVSA}, we will also consider its module  
\begin{equation}\label{eqdef:quiverVSAmodule}
    \CV_\lambda[Q]:=H^{\infty/2+0}_{\mathrm{BRST}}(\mathfrak g,\CV_\lambda^\#[Q]) ~,
\end{equation}
which was defined in \cite{Coman:2023xcq} with $\lambda :=(\lambda_i)_{i\in Q_0}\in\mathbb{Z}^{Q_0}$, where 
\begin{equation} \label{eqdef:prequotientquiverVOA-module}
\CV_\lambda^\#[Q] := \mathcal D_\lambda^{\mathrm{ch}}(\mathfrak R) := \CV_{\beta\gamma}[\mathrm{Rep}(\mathbf v,\mathbf w)] \otimes \CV_{\Phi\Psi}[\bigoplus_{i\in Q_0}\Pi\Hom(U_i,V_i)] \otimes \CH_\lambda(\mathcal Z) ~
\end{equation}
is a module of the vertex superalgebra $\mathcal D^{\mathrm{ch}}(\mathfrak R)$, and $\CH_\lambda(\mathcal Z)$ is the Verma module of the Heisenberg algebra $\CH(\mathcal Z)=\{h^i, i\in Q_0\}$ which is generated from a vector $|\lambda\rangle$, on which the zero modes $h^i_{(0)}$ act with weights $\lambda_i\in\mathbb{Z}$ as $h^i_{(0)} |\lambda\rangle= \lambda_i |\lambda\rangle$.

\subsection{Associated variety}\label{sec:MapPoissonAlgebra}

Let us recall now some standard concepts from the literature on vertex algebras (see e.g.\ \cite{arakawa2017introduction}). A canonical filtration (also known as Li's filtration) $F^\bullet$ of a finitely strongly generated vertex algebra $\CV$, with strong generators $\mathcal{O}^1, \dots , \mathcal{O}^k\in\CV$, is defined by 
\beq\label{eq:canonical filtration}
\begin{split}
    \CV &=F^{0}\CV\supset F^{1}\CV\supset\cdots\supset F^{p}\CV\supset F^{p+1}\CV\supset\cdots, 
    \\
    F^{p}\CV :&=\mathrm{Span}\left\{\mathcal{O}^{1}_{(-n_1-1)}\cdots \mathcal{O}^{k}_{(-n_k-1)}| 0\rangle\Biggr|\, 
    \mathcal{O}^1, \dots , \mathcal{O}^k \in \CV,  \sum_{i=1}^{k}n_i\geq p\in\BZ_{\geq0}\right\}\;.  
\end{split}
\eeq
The associated graded vector space
\beq\label{def:associated graded algebra}
    \mathrm{gr}_{F}\CV:=\bigoplus_{p=0}\mathrm{gr}_{F}^{p}\CV \;,
    \quad 
    \mathrm{gr}_{F}^{p}\CV:= F^{p}\CV/F^{p+1}\CV \;
\eeq  
admits a natural vertex Poisson algebra structure \cite[Proposition 3.14]{arakawa2017introduction}, with $\mathrm{gr}_{F}\CV$ isomorphic to $\cV$ as vector spaces, while the vertex Poisson algebra structure of the former replaces the vertex operator algebra structure of the latter. There is a canonical filtration \eqref{eq:canonical filtration} $F^\bullet \CV$ for every vertex algebra $\CV$. Zhu's $C_2$-algebra of $\CV$ 
\begin{equation}
    \mathcal{R}(\CV):=F^0\CV/F^1\CV
\end{equation}
is a Poisson vertex algebra \cite[Proposition 3.15]{arakawa2017introduction}. Denoting by $\overline{\mathcal{O}}$ the image of $\mathcal{O}\in\CV$ under the surjection $\CV\twoheadrightarrow \CR(\CV)$, Zhu's $C_2$-algebra has a commutative product $\overline{\mathcal{O}^i} \cdot \overline{\mathcal{O}^j} = \overline{\mathcal{O}^i_{(-1)}\mathcal{O}^j}$. The vertex algebra $\CV$ is finitely strongly generated by a set of elements $\mathcal{O}^1, \dots , \mathcal{O}^n\in\CV$ if their images $\overline{\mathcal{O}^1}, \dots , \overline{\mathcal{O}^n}\in\CR(\CV)$ generate $\CR(\CV)$ as a $\mathbb{C}$-algebra. Then $\CV$ is spanned by elements of the form $\mathcal{O}^{i_1}_{-n_1}\cdots \mathcal{O}^{i_k}_{-n_k}|0\rangle$, with $k\geq 0$ and $n_1\geq 0$.

\medskip 

Having defined $\CR(\CV)$ from the associated graded algebra $\mathrm{gr}_{F}\CV$, in the opposite direction this is recovered from the $\infty$-jet $J_\infty\CR(\CV)$, which admits a natural vertex Poisson algebra structure \cite[Theorem 3.9]{arakawa2017introduction}, via a natural surjective map \cite[Theorem 3.17]{arakawa2017introduction} $J_\infty\CR(\CV)\twoheadrightarrow \mathrm{gr}_{F}\CV$. 

\medskip 

The associated variety $X_{\CV}$ of $\CV$ is the geometry corresponding to the non-nilpotent part of Zhu's $C_2$-algebra
\begin{equation}
    X_{\CV}:=\mathrm{Specm}\,\CR(\CV) ~,
\end{equation}
with a sequence of surjective maps $\CV \twoheadrightarrow \CR(\CV)\twoheadrightarrow (\CR(\CV))_{\mathrm{red}}=\mathbb{C}[X_{\CV}]$.  

\medskip

\paragraph{Note:} in the context of quiver vertex algebras for type A linear quivers $(Q,\mathbf v,\mathbf w)$, it is expected \cite[Conjecture 7.1]{Coman:2023xcq} (see also \cite{Beem:2023dub}) that the associated variety can be identified with the affine quiver variety $\widetilde{\CM}^0(\mathbf v,\mathbf w)$ \eqref{def:affinequivervariety}.

\section{Quiver sheaves of \texorpdfstring{$\hbar$}{}-Adic Vertex Superalgebras}\label{sec:QuiverSheavesVSA}

In Section \ref{sec:ChiralizationQuiverRep}, we constructed a vertex superalgebra $\CV[Q]$ \eqref{eqdef:quiverVSA} associated with the quiver data $(Q,\mathbf{v},\mathbf{w})$ by chiralizing the corresponding Nakajima quiver super-variety $\widetilde{\mathcal{M}}(\mathbf{v},\mathbf{w}|\mathbf{u})$. In contrast, in this section, we will construct the sheaf of ($\hbar$-adic) vertex superalgebras $\mathcal{D}^{\mathrm{ch}}_{\widetilde{\mathcal M},\hbar}$ by micro-localization\footnote{That is, localizing on open sets in the (extended) phase space $\mathfrak{R}$ \eqref{def:prequotient}, where elements are invertible.} of the parent vertex superalgebra $\CV^\#[Q]$ \eqref{eqdef:prequotientquiverVOA}, in order to define a sheaf of $\hbar$-adic vertex superalgebras $\mathcal{D}^{\mathrm{ch}}_{\mathfrak R,\hbar}$ associated to the extended representation space $\mathfrak R$ of the quiver \eqref{def:prequotient}, and then perform the BRST reduction procedure locally to reduce to the vertex superalgebra $\mathcal{D}^{\mathrm{ch}}_{\widetilde{\mathcal M},\hbar}$.

\subsection{Sheaves of \texorpdfstring{$\hbar$}{}-adic vertex algebras}

$\hbar$-adic vertex algebras were introduced in \cite{li2004vertex} with $\hbar$ a formal variable. These allow to consider inverses of rational functions and infinite sums with singular terms, and to localize vertex algebras. 

\medskip

\begin{defn}\label{def:h-adic vertex superalgebras}
Formally, an $\hbar$-adic vertex algebra is a tuple $(V, |0\rangle, \partial, Y(-,z))$ such that $V$ is a flat $\mathbb C[\![\hbar]\!]$-module complete in $\hbar$-adic topology, and $Y(-,z):V\to \mathrm{End}_{\mathbb C[\![\hbar]\!]}(V)[\![z^{\pm}]\!]$ is a $\mathbb C[\![\hbar]\!]$-linear map such that
\begin{enumerate}
    \item modes $a_{(n)}$ are continuous with respect to the $\hbar$-adic topology, for all $a\in V, n\in \mathbb Z$;
    \item $(V/\hbar^NV,|0\rangle,\partial,Y(-,z))$ is a vertex algebra for all $N\in \mathbb Z_{\ge 1}$.
\end{enumerate}
\end{defn}

\subsubsection{Example (sheaf) of \texorpdfstring{$\hbar$}{}-adic vertex algebra}\label{sec:hbaradicbetagammapair} 

The first building block introduced in Section \ref{sec:generalconstructionSVA} and employed to construct a vertex superalgebra has been a pair of $\beta\gamma$ bosonic fields or chiral differential operators $D^{\mathrm{ch}}( T^*\mathbb{C})$ (specializing $M$ in Section \ref{sec:generalconstructionSVA} to $\dim M=1$). The $\hbar$-adic version of this $\beta\gamma$-pair has vector space (or vacuum module)  
\begin{align}
    D^{\mathrm{ch}}(T^\ast \mathbb C)_{\hbar}:=\mathbb C[\![\hbar]\!][\beta_{(-n)},\gamma_{(-n)}\:|\:n\in \mathbb Z_{\ge 1}]|0\rangle \;,
\end{align}
with vertex algebra structure given by the OPE
\begin{align}
    \beta(z)\gamma(w)\sim \frac{\hbar}{z-w}\;.
\end{align}
$D^{\mathrm{ch}}(T^\ast\mathbb C)_\hbar/(\hbar)$ is a vertex Poisson algebra isomorphic to the coordinate ring on the jet bundle $$D^{\mathrm{ch}}(T^\ast\mathbb C)_\hbar/(\hbar)\cong\mathbb C[{\jet} T^\ast \mathbb C]$$ endowed with the natural vertex Poisson algebra structure induced by the Poisson structure on $T^*\mathbb C$. From the $\hbar$-adic CDO $D^{\mathrm{ch}}( T^*\mathbb{C})$, we can construct a sheaf of $\hbar$-adic vertex algebra $\mathscr D^{\mathrm{ch}}_{T^\ast\mathbb C,\hbar}$ on $T^\ast\mathbb C$ by the sheafification of the following functor
\begin{align}
    U_f\mapsto\mathscr D^{\mathrm{ch}}_{T^\ast\mathbb C,\hbar}(U_f):=\mathbb C[{\jet}U_f][\![\hbar]\!] \;,
\end{align}
where $f$ is a polynomial in $\beta,\gamma$, and $U_f\subset  T^\ast \mathbb C$ is the open subset of nonzero locus of $f$ allowing to perform micro-localization. The $\hbar$-adic vertex algebra structure on the vector space $\mathbb C[{\jet}U_f][\![\hbar]\!]$ is defined using Wick's contraction formula \cite{MR1651389} (see also \cite{Kapranov_1998,Bezrukavnikov_2004} for localization)
\begin{align}
    A(z)B(w)\sim :e^{\frac{\hbar}{z-w}\partial_\beta\wedge\partial_\gamma}(A(z)\otimes B(w)): \;
\end{align}
where $A,B\in \mathbb C[\partial^n\beta,\partial^n\gamma\:|\:n\in \mathbb Z_{\ge 0}][f^{-1}]$, and $::$ is the normal-ordered product. For example in \cite{Coman:2023xcq} we let $f=\beta\gamma$ and took $A=\beta^{-1}$ and $B=\gamma^{-1}$, for which the OPE is
\begin{align}
    \beta^{-1}(z)\gamma^{-1}(w)\sim \sum_{n=1}^{\infty}n!\hbar^n\frac{:\beta^{-n-1}(z)\gamma^{-n-1}(w):}{(z-w)^n} \;.
\end{align}
In the opposite direction to the sheaf construction, we recover
\begin{equation}
     D^{\mathrm{ch}}(T^\ast \mathbb C)_{\hbar}= \Gamma(T^\ast\mathbb C, \mathscr D^{\mathrm{ch}}_{T^\ast \mathbb C,\hbar}) ~.
\end{equation}

\subsubsection{General building blocks}

The $\hbar$-adic $\beta\gamma$-pair from Section \ref{sec:hbaradicbetagammapair} and the corresponding sheaf construction of $\mathscr D^{\mathrm{ch}}_{T^\ast \mathbb C,\hbar}$ by localization generalize to the building blocks introduced in Section \ref{sec:generalconstructionSVA}, which entered in the construction of a vertex superalgebra $\CV^\#$ \eqref{eqdef:prequotientquiverVOA}. These building blocks become
\begin{itemize}
    \item $D^{\mathrm{ch}}(T^*V)_{\hbar}$ is a an $\hbar$-adic $\beta\gamma$-system $$D^{\mathrm{ch}}(T^*V)_{\hbar}=D^{\mathrm{ch}}(T^*\mathbb C)_{\hbar}^{\otimes m}~, \qquad V\cong \mathbb C^m~, \qquad \forall m\in \mathbb Z_{\ge 1} ~$$ and $\mathscr D^{\mathrm{ch}}_{T^*V,\hbar}$ is a sheaf of $\hbar$-adic asymptotic chiral differential operators (ACDO) on $T^*V$ \cite{arakawa2015localization} defined by the localization of $D^{\mathrm{ch}}(T^*V)_{\hbar}$ using Wick's contraction formula;
    \item $D^{\mathrm{ch}}(T^*\Pi U)_{\hbar}$ is similarly a system of $\hbar$-adic $\Phi\Psi$ free fermions where $U$ is a vector space and $\Pi U$ is the odd vector space associated to $U$;
    \item $D^{\mathrm{ch}}(T^*\Pi S)_{\hbar}$ is an $\hbar$-adic $\mathsf b\mathsf c$-system where $S$ is a vector space; 
    \item More generally, the vertex superalgebra $D^{\mathrm{ch}}(T^*(V\oplus\Pi S))_{\hbar}$ is an $\hbar$-adic $\beta\gamma\mathsf b\mathsf c$-system. We only localize the bosonic variables, i.e.\  only the $\beta\gamma$ part is allowed to take the inverse, thus defining a sheaf of $\hbar$-adic vertex superalgebras $\mathscr D^{\mathrm{ch}}_{T^*(V\oplus\Pi S),\hbar}$ on the variety $T^*V$;
    \item $\mathcal{H}(W,(\cdot,\cdot))_\hbar$ is the $\hbar$-adic Heisenberg algebra associated with the vector space $W=\bigoplus_{i=1}^m\mathbb C\cdot a^i$ with a symmetric inner product $(\cdot,\cdot)$. The $\hbar$-adic Heisenberg vertex algebra has vector space  
    \begin{align}
    \mathcal{H}(W,(\cdot,\cdot))_\hbar:=\mathbb C[\![\hbar]\!][a^1_{(-n)},\cdots,a^m_{(-n)}\:|\:n\in \mathbb Z_{\ge 1}]|0\rangle \;,
    \end{align}
    with $\hbar$-adic vertex algebra structure given by the OPE
    \begin{align}
    a^i(z)a^j(w)\sim \frac{\hbar^2(a^i,a^j)}{(z-w)^2} \;.
    \end{align}
    $\mathscr H^{(\cdot,\cdot)}_{W,\hbar}$ is a sheaf of $\hbar$-adic Heisenberg vertex algebras defined by localization 
    \begin{align}
    \mathscr H^{(\cdot,\cdot)}_{W,\hbar}(U_f):=\mathbb C[J_\infty U_f][\![\hbar]\!] \;,
    \end{align}
    endowed with the aforementioned $\hbar$-adic vertex algebra structure (see also \cite[4.3]{kuwabara2017vertex}), where $U_f$ is the locus where a polynomial function $f\in \mathbb C[a^1,\cdots,a^m]$ is non-vanishing. The localization is defined similarly to that in the case of $\hbar$-adic $\beta\gamma$ systems, using the Wick contraction formula for the OPE 
    \begin{align}
    A(z)B(w)\sim :e^{\frac{\hbar^2(a^i,a^j)}{(z-w)^2} \partial_{a^i} \otimes \partial_{a^j}}(A(z)\otimes B(w)): \;
    \end{align}
    for $A,B\in \mathbb C[\partial^n a^1,\cdots,\partial^na^m\:|\:n\in \mathbb Z_{\ge 0}][f^{-1}]$, where $::$ is the normal-ordered product\footnote{Recall the example from \cite{Coman:2023xcq}: let $f=a^1a^2$, then take $A=a^1$ and $B=a^2$, the OPE reads:
    \begin{align}
    (a^1)^{-1}(z)(a^2)^{-1}(w)\sim \sum_{n=1}^{\infty}n!\hbar^{2n}(a^1,a^2)^n\frac{:(a^1)^{-n-1}(z)(a^2)^{-n-1}(w):}{(z-w)^{2n}} \;.
    \end{align}}.
\end{itemize}

\subsubsection{Quiver setting \texorpdfstring{$\hbar$}{}-adic vertex superalgebras}

In the quiver setting, for the quiver data $(Q,\mathbf v, \mathbf w)$ from Sections \ref{sec:GeomQuiverVarieties} and \ref{sec:VSA-quiverSetting}, we define a sheaf of $\hbar$-adic vertex superalgebras $\mathscr D^{\mathrm{ch}}_{\mathfrak R,\hbar}$ on the prequotient $\mathfrak R$ \eqref{def:prequotient}, by 
\begin{align}\label{eqdef:sheafhbaradicQuiverSVA-Prequotient}
    \mathscr D^{\mathrm{ch}}_{\mathfrak R,\hbar}:=\mathscr D^{\mathrm{ch}}_{T^*\mathrm{Rep}(\mathbf v,\mathbf w|\mathbf u),\hbar}\widehat{\boxtimes} \mathscr H^{(\cdot,\cdot)_Q}_{\mathbb C^{Q_0},\hbar} \;,
\end{align}
where $\widehat\boxtimes$ is the $\hbar$-completed external tensor product of sheaves, and $(\cdot,\cdot)_{Q}$ is the Cartan inner product on $\mathcal Z$
\begin{align}
    (h_i,h_j)_Q:=2\delta_{ij}-\#(a:i\to j\in Q_1)-\#(a:j\to i\in Q_1)\;.
\end{align}
$\mathscr D^{\mathrm{ch}}_{\mathfrak R,\hbar}$ is the sheafified version of $\CVs[Q]$, namely the $\hbar$-adic completion $\CVs[Q]_{\hbar}$ is isomorphic to the ring of global section $\Gamma(\mathfrak R,\mathscr D^{\mathrm{ch}}_{\mathfrak R,\hbar})$.

\subsection{\texorpdfstring{$\hbar$}{}-adic BRST reduction}

In Section \ref{sec:ChiralizationQuiverRep}, we constructed a quiver vertex superalgebra $\CV[Q]=H^{\infty/2+0}_{\mathrm{BRST}}(\mathfrak g,\CV^\#[Q])$ \eqref{eqdef:quiverVSA} via the BRST reduction procedure outlined in Section \ref{sec:BRST reduction}, starting from the parent vertex superalgebra $\CV^\#[Q]$ \eqref{eqdef:prequotientquiverVOA}. With the aim to proceed analogously with the sheaf of $\hbar$-adic vertex superalgebras $\mathscr D^{\mathrm{ch}}_{\mathfrak R,\hbar}$, we will review here the procedure of $\hbar$-adic BRST reduction, following the presentation in \cite[Section 10.1]{Coman:2023xcq}. As before, we let $\mathfrak g$ be a reductive Lie algebra, with $\kappa$ a $\mathfrak g$-invariant symmetric inner form on $\mathfrak g$ and $\kappa_{\mathfrak g}$ the Killing form on $\mathfrak g$. 

\bigskip

\begin{defn}\label{def:hbaradicCurrentAlgebraVg}
    $V^\kappa(\mathfrak g)_\hbar$ is an $\hbar$-adic current algebra  with vector space the $\mathbb C[\![\hbar]\!]$-module 
    \begin{equation}
    \mathbb{C}[J_\infty\mathfrak g][\![\hbar]\!]:=\mathbb C[\![\hbar]\!][E^\alpha_{(-n)}\:|\:n\in \mathbb{Z}_{\ge 1}, E^\alpha\in\mathrm{a~basis~ of~ }\mathfrak g]|0\rangle    
    \end{equation} 
    endowed with the OPE
    \begin{align}
    E^{\alpha}(z)E^\beta(w)\sim \frac{\hbar^2\kappa^{\alpha,\beta}}{(z-w)^2}+\frac{\hbar f^{\alpha\beta}_\gamma E^\gamma(w)}{z-w} \;,
    \end{align}
where $\kappa^{\alpha,\beta}=\kappa(E^\alpha,E^\beta)$ and $f^{\alpha\beta}{}_\gamma$ are the structure constants of $\mathfrak g$, with $f^{\alpha\beta}{}_\gamma E^\gamma=[E^\alpha,E^\beta]$.  
\end{defn}

If  $V$ is an $\hbar$-adic vertex superalgebra such that there exists an $\hbar$-adic vertex superalgebra map
\begin{align}
    J_V: V^{-\kappa_{\mathfrak g}}(\mathfrak g)_\hbar\to V \;,
\end{align}
then we define $\hbar$-adic vertex superalgebra 
\begin{align}
    \widetilde C(\mathfrak g, V)_\hbar:=V\otimes \CV_{\mathsf b\mathsf c}(\mathfrak g)_\hbar \;,
\end{align}
where $\CV_{\mathsf b\mathsf c}(\mathfrak g)_\hbar$ is the $\hbar$-adic $\mathsf b\mathsf c$ system of $\mathfrak g$, i.e.\  chiral differential operators $D^{\mathrm{ch}}(\Pi T^*\mathfrak g)_\hbar$ of the odd symplectic vector space $\Pi T^*\mathfrak g$. $\CV_{\mathsf b\mathsf c}(\mathfrak g)_\hbar$ admits an $\hbar$-adic vertex superalgebra map
\begin{align}
    {}^{\hbar}J_{\mathsf b\mathsf c}: V^{\kappa_{\mathfrak g}}(\mathfrak g)_\hbar\to \CV_{\mathsf b\mathsf c}(\mathfrak g)_\hbar\;, \quad 
     {}^{\hbar}J_{\mathsf b\mathsf c}(E^\alpha)=f^{\alpha\beta}{}_{\gamma}\mathsf b^\gamma\mathsf c_{\beta}\;,
\end{align}
and is graded by ghost numbers with $\deg \mathsf c=+1$ and $\deg\mathsf b=-1$, so the image of $J$ is contained in the degree zero piece $\CV^0_{\mathsf b\mathsf c}(\mathfrak g)_\hbar$. The $\hbar$-adic BRST current is the ghost number $1$ element
\begin{align}
    \JBRST_{\hbar}:=\mathsf c_\alpha\left(J_V(E^\alpha)+\frac{1}{2}J_{\mathsf b\mathsf c}(E^\alpha)\right)\in  \widetilde C^1(\mathfrak g, V)_\hbar \;.
\end{align}
The operator 
\begin{equation} \label{eqdef:hbarBRSTcurrent}
\mathcal{Q}_{\hbar}:=\frac{1}{\hbar}\JBRST_{(0)}:=\frac{1}{\hbar}\oint \frac{dz}{2\pi i} \, \JBRST(z)
\end{equation}
acts on $\widetilde C(\mathfrak g, V)_\hbar$ and squares to zero $\mathcal{Q}_{\hbar}^2=0$. Therefore the $\hbar$-adic vertex superalgebra $\widetilde C(\mathfrak g, V)_\hbar$ defines a chain complex $(\widetilde C(\mathfrak g, V)_\hbar, \mathcal{Q}_{\hbar})$ which is graded by ghost numbers, called the $\hbar$-adic BRST complex. Moreover, the linear map 
\begin{align}\label{eq:J(E)}
    E^\alpha\mapsto J(E^\alpha):=\mathcal{Q}_{\hbar}(\mathsf b^\alpha)=J_V(E^\alpha)+J_{\mathsf b\mathsf c}(E^\alpha)
\end{align}
is an $\hbar$-adic vertex superalgebra chain complex map 
$$ J: (V^0(\mathfrak g)_\hbar,0)\to (\widetilde C(\mathfrak g, V)_\hbar, \mathcal{Q}_{\hbar}) \;, $$ 
where the LHS is endowed with trivial differential. In particular the Lie algebra $\mathfrak g$ acts on the complex $(\widetilde C(\mathfrak g, V)_\hbar, \mathcal{Q}_{\hbar})$ by $E^\alpha\mapsto \frac{1}{\hbar}J(E^\alpha)_{(0)}$. The relative $\hbar$-adic BRST complex is defined to be the subspace 
\begin{align}
    C(\mathfrak g, V)_\hbar:=\{v\in \widetilde C(\mathfrak g, V)_\hbar\:|\: J(E^\alpha)_{(0)}v=\mathsf b^\alpha_{(0)}v=0,\forall \alpha\} \;,
\end{align}
which is a sub-complex because $[J(E^\alpha)_{(0)},\mathcal{Q}_{\hbar}]=0$ and $[\mathsf b^\alpha_{(0)},\mathcal{Q}_{\hbar}]=J(E^\alpha)_{(0)}$, and which is $\BZ$-graded by ghost numbers. We define the relative $\hbar$-adic BRST cohomology to be the $\mathbb Z$-graded $\hbar$-adic vertex superalgebra
\begin{align}
    H^{\infty/2+\bullet}_{\hbar\mathrm{BRST}}(\mathfrak g,V):=H^{\bullet}(C(\mathfrak g, V)_\hbar, \mathcal{Q}_{\hbar})\;.
\end{align}

\subsection{Quiver sheaf of \texorpdfstring{$\hbar$}{}-adic vertex superalgebras}\label{sec:VSA-quiverSetting-Local}

Let us now return to the quiver setting of Section \ref{sec:GeomQuiverVarieties} with quiver and dimension data $(Q,\mathbf v, \mathbf w)$, and where $\mathfrak g=\mathfrak{gl}(\mathbf v)$. Recall that we have defined a sheaf of $\hbar$-adic vertex superalgebras $\mathscr D^{\mathrm{ch}}_{\mathfrak R,\hbar}$ \eqref{eqdef:sheafhbaradicQuiverSVA-Prequotient} on the extended representation space $\mathfrak R$ by\eqref{def:prequotient} \footnote{The notation $\widehat\boxtimes$ refers to the $\hbar$-adic completion tensor product, adapted to $\hbar$-adic topology, where for $\hbar$-adic vertex algebras $V,\,W$, $V\widehat\boxtimes W\cong V\widehat\otimes_{\mathbb{C}[[\hbar]]}W =\varprojlim_{n} \left( (V/\hbar^n V)\otimes (W/\hbar^n W) \right)$.}
\begin{align}
    \mathscr D^{\mathrm{ch}}_{\mathfrak R,\hbar}=\mathscr D^{\mathrm{ch}}_{T^*\mathrm{Rep}(\mathbf v,\mathbf w|\mathbf u),\hbar}\widehat{\boxtimes} \mathscr H^{(\cdot,\cdot)_Q}_{\mathbb C^{Q_0},\hbar} \;,
\end{align}
which is the sheaf analog of the vertex superalgebra $\CVs[Q]$. Moreover, in Section \ref{sec:VSA-quiverSetting} we constructed the quiver vertex superalgebra $\CV[Q]=H^{\infty/2+0}_{\mathrm{BRST}}(\mathfrak g,\CV^\#[Q])$ \eqref{eqdef:quiverVSA} by BRST reduction. Let us proceed here analogously for the sheaf $\mathscr D^{\mathrm{ch}}_{\mathfrak R,\hbar}$. The chiral extended moment map $\widetilde\mu_{\mathrm{ch}}$ is an $\hbar$-adic vertex superalgebra map from the current algebra 
\begin{equation}\label{def:hbar-BRST-cohomology}
    \widetilde\mu_{\mathrm{ch}}:V^{-\kappa_{\mathfrak{g}}}(\mathfrak{g})_\hbar \to \CVs[Q]_\hbar ~.
\end{equation}
In order to perform the BRST reduction procedure locally, we need to specify the open sets in the Nakajima quiver variety $\widetilde{\mathcal{M}}$ and prequotient $\mathfrak{R}$ which are relevant in the localization procedure.

\bigskip

\begin{definition}{\cite[Definition 10.1]{Coman:2023xcq}}\label{def:good open subsets}
We say an open subset $U\subset \widetilde\CM$ \textit{good} if there exists a $G$-closed open affine $\widetilde U\subset \mathfrak R^{\mathrm{st}}$ such that $\widetilde U\cap \widetilde\mu^{-1}(0)=q^{-1}(U)$ where $q^{-1}(U)$ is the preimage of $U$ under the quotient map $q:\widetilde\mu^{-1}(0)^{\mathrm{st}}\to \widetilde\CM$. We call such $\widetilde U$ a \textit{good lift} of $U$.
\end{definition}
\noindent It follows from the definition that a good open subset is affine. Remark now the following. 
\begin{lem}{\cite{Coman:2023xcq}}\label{lem:good open form a base}
The set of good open subsets forms a base of Zariski topology on $\widetilde\CM$.
\end{lem}

In \cite[Section 10]{Coman:2023xcq}, we defined a presheaf of $\mathbb Z$-graded $\hbar$-adic vertex superalgebras on the set of good open subsets of $\widetilde\CM$ by assigning every good open subset $U\subset \widetilde\CM$ to the relative $\hbar$-adic BRST cohomology
\begin{align}\label{presheaf}
    U\mapsto H^{\infty/2+\bullet}_{\hbar\mathrm{BRST}}(\mathfrak{g},\mathscr D^{\mathrm{ch}}_{\mathfrak R,\hbar}(\widetilde U))\;,
\end{align}
where $\widetilde U$ is a good lift of $U$, and   $H^{\infty/2+\bullet}_{\hbar \mathrm{BRST}} (\mathfrak{g} ,\mathscr D^{\mathrm{ch}}_{\mathfrak R,\hbar}(\widetilde U))$ does not depend on the choice of good lift $\widetilde U$. 

\medskip

\begin{prop}
$H^{\infty/2+\bullet}_{\hbar\mathrm{BRST}}(\mathfrak{g},\mathscr D^{\mathrm{ch}}_{\mathfrak R,\hbar}(\widetilde U))$ does not depend on the choice of good lift $\widetilde U$.
\end{prop}

\medskip

\begin{proof}
Let $\widetilde U_1,\widetilde U_2$ be two good lifts of $U$. Then $\widetilde U_{12}:=\widetilde U_1\cap\widetilde U_2$ is also a good lift of $U$. Then it is enough to show that the induced maps 
\begin{align*}
    H^{\infty/2+\bullet}_{\hbar\mathrm{BRST}}(\mathfrak{g},\mathscr D^{\mathrm{ch}}_{\mathfrak R,\hbar}(\widetilde U_{i}))\longrightarrow H^{\infty/2+\bullet}_{\hbar\mathrm{BRST}}(\mathfrak{g},\mathscr D^{\mathrm{ch}}_{\mathfrak R,\hbar}(\widetilde U_{12}))\:, \text{ for }\:i=1,2
\end{align*}
are isomorphisms. By symmetry, we only need to show this for $i=1$. Let us introduce the notations $W=\widetilde U_1\setminus \left(\widetilde U_{12}\cap \widetilde{\mu}^{-1}(0)\right)$ and $V=W\cap \widetilde U_{12}$. Note that both $W$ and $V$ are $G$-closed affine open subsets in the stable locus $\mathfrak{R}^{\mathrm{st}}$, and $W\cap \widetilde{\mu}^{-1}(0)=V\cap \widetilde{\mu}^{-1}(0)=\emptyset$. We have the short exact sequence 
\begin{align*}
    0\longrightarrow\mathscr D^{\mathrm{ch}}_{\mathfrak R,\hbar}(\widetilde U_1)\longrightarrow \mathscr D^{\mathrm{ch}}_{\mathfrak R,\hbar}(\widetilde U_{12})\oplus \mathscr D^{\mathrm{ch}}_{\mathfrak R,\hbar}(W)\longrightarrow \mathscr D^{\mathrm{ch}}_{\mathfrak R,\hbar}(V)\longrightarrow 0
\end{align*}
since $\mathscr D^{\mathrm{ch}}_{\mathfrak R,\hbar}$ is a sheaf. Then the proposition follows from the relative $\hbar$-adic BRST cohomology long exact sequence associated to the above short exact sequence, and the Lemma \ref{lem: vanishing BRST away from zero} below.
\end{proof}

\medskip

\begin{lem}\label{lem: vanishing BRST away from zero}
Let $S\subset \mathfrak{R}$ be a $G$-closed open affine subset such that $S\cap \widetilde{\mu}^{-1}(0)=\emptyset$, then $$H^{\infty/2+\bullet}_{\hbar \mathrm{BRST}} (\mathfrak{gl}(\mathbf v),\mathscr D^{\mathrm{ch}}_{\mathfrak R,\hbar}(S))=0\:.$$
\end{lem}

\medskip

\begin{proof}
Consider the $\hbar$-adic filtration $\prescript{\hbar}{}{F}_{\bullet}$ on the relative $\hbar$-adic BRST complex $C(\mathfrak{gl}(\mathbf v),\mathscr D^{\mathrm{ch}}_{\mathfrak R,\hbar}(S))_\hbar$ (see below Section \ref{sec:MapPoissonAlgebra}). Let $\prescript{\hbar}{}{E}^{p,q}_r$ be the spectral sequence associated with this filtration. Since $\prescript{\hbar}{}{F}_{\bullet}$ is bounded above and complete, then we have a convergent spectral sequence $$\prescript{\hbar}{}{E}^{p,q}_r\Longrightarrow_p H^{\infty/2+p+q}_{\hbar\mathrm{BRST}}(\mathfrak{gl}(\mathbf v),\mathscr D^{\mathrm{ch}}_{\mathfrak R,\hbar}(S))$$ by the complete convergence theorem. Note that $\prescript{\hbar}{}{E}^{p,q}_0\cong \hbar^p C^{p+q}(\mathfrak{gl}(\mathbf v),\mathscr O_{J_\infty\mathfrak{R}}(S))$, where the coboundary map is the relative vertex Poisson BRST differential $\mathcal Q_P$. Then it is enough to show that the relative vertex Poisson BRST reduction $$H^{\infty/2+\bullet}_{\mathrm{PBRST}}(\mathfrak{gl}(\mathbf v),\mathscr O_{J_\infty\mathfrak{R}}(S))$$ is trivial. Write $\mathcal Q_P=\mathcal Q_P^++\mathcal Q_P^-$ as in \cite[Equation (E.7)]{Coman:2023xcq}, and let $$E^{p,q}_0=H^q(C^{p,q}(\mathfrak{gl}(\mathbf v),\mathscr O_{J_\infty\mathfrak{R}}(S)),\mathcal Q_P^-)$$ be the first page of the associated spectral sequence $E^{p,q}_r$. By \cite[Lemma E.2]{Coman:2023xcq}, $E^{p,q}_r$ converges to $H^{\infty/2+p+q}_{\mathrm{PBRST}}(\mathfrak{gl}(\mathbf v),\mathscr O_{J_\infty\mathfrak{R}}(S))$  and its first page is computed in \cite[(E.15)]{Coman:2023xcq}, which is zero in our case because $S\cap \widetilde{\mu}^{-1}(0)$ is an empty set. This implies $H^{\infty/2+\bullet}_{\mathrm{PBRST}}(\mathfrak{gl}(\mathbf v),\mathscr O_{J_\infty\mathfrak{R}}(S))=0$, therefore the lemma follows. 
\end{proof}

\subsubsection{Sheafification of the quiver presheaf of \texorpdfstring{$\hbar$}{}-adic vertex superalgebras}

Now $\mathscr H^{\infty/2+\bullet}_{\hbar\mathrm{BRST}}(\mathfrak{gl}(\mathbf v),\mathscr D^{\mathrm{ch}}_{\mathfrak R,\hbar})$ denotes the sheafification of the presheaf \eqref{presheaf} in the Zariski topology. This is a sheaf of $\mathbb Z$-graded $\hbar$-adic vertex superalgebras, with cohomological degree zero piece  
\begin{align}\label{def:sheaf of BRST reductions}
    \mathscr D^{\mathrm{ch}}_{\widetilde\CM,\hbar}:= \mathscr H^{\infty/2+0}_{\hbar\mathrm{BRST}}(\mathfrak{gl}(\mathbf v),\mathscr D^{\mathrm{ch}}_{\mathfrak R,\hbar})\;.
\end{align}
We note that if $U\subset \widetilde\CM$ is a good open affine subset, then 
\begin{align}\label{eq:sheaf on affine}
    \mathscr D^{\mathrm{ch}}_{\widetilde\CM,\hbar}(U)=H^{\infty/2+0}_{\hbar\mathrm{BRST}}(\mathfrak{gl}(\mathbf v),\mathscr D^{\mathrm{ch}}_{\mathfrak R,\hbar}(\widetilde U))\;.
\end{align}
In \cite{Coman:2023xcq}, we have proven the following theorem.

\medskip

\begin{thm}{\cite[Theorem 10.1]{Coman:2023xcq}}\label{thm: vanishing for sheaf}
The BRST cohomology sheaf $\mathscr H^{\infty/2+n}_{\hbar\mathrm{BRST}}(\mathfrak{gl}(\mathbf v),\mathscr D^{\mathrm{ch}}_{\mathfrak R,\hbar})$ vanishes for $n\neq 0$. Moreover, there is an isomorphism between sheaves of vertex Poisson superalgebras
\begin{align}
    \mathscr D^{\mathrm{ch}}_{\widetilde\CM,\hbar}/\hbar\mathscr D^{\mathrm{ch}}_{\widetilde\CM,\hbar}\cong \mathscr O_{J_\infty\widetilde\CM(Q,\mathbf v,\mathbf w|\mathbf u)}\;,
\end{align}
where $\mathbf u:=\mathbf w-\mathsf C\mathbf v$ and $\widetilde\CM(Q,\mathbf v,\mathbf w|\mathbf u)$ is the extended quiver super variety.
\end{thm}

There is a one-dimensional torus $\mathbb S$ that acts on $\mathfrak R$ with a weight assignment to the coordinates (see \cite[Section 10.2]{Coman:2023xcq}; remark that the arguments in loc.cit. carry though when $\mathbf u\neq\mathbf 0$, with $\mathrm{wt}(\Phi_{(-n)})=\mathrm{wt}(\Psi_{(-n)})=n$). $\mathscr{D}^{\mathrm{ch}}_{\widetilde{\mathcal{M}},\hbar}(\widetilde{\mathcal{M}})$ is a complete $\mathbb S$-module, with weights bounded from below
\begin{align*}
    \mathscr D^{\mathrm{ch}}_{\widetilde\CM,\hbar}(\widetilde\CM)={\prod_{m\ge -M}}'\mathscr D^{\mathrm{ch}}_{\widetilde\CM,\hbar}(\widetilde\CM)_m \;,
\end{align*}
where $\mathscr D^{\mathrm{ch}}_{\widetilde\CM,\hbar}(\widetilde\CM)_m$ is the weight $m$ subspace and ${\prod}'$ means the subspace of the product space consisting of elements $(x_i)_{i=-M}^{\infty}$ such that there exists $N$ and $x_n\in \hbar^{n-N}\mathscr D^{\mathrm{ch}}_{\widetilde\CM,\hbar}(\widetilde\CM)_N$ for all $n>N$. The subspace of $\mathbb S$-finite elements
\begin{align}
    \mathscr D^{\mathrm{ch}}_{\widetilde\CM,\hbar}(\widetilde\CM)_{\mathbb S\text{-fin}}:=\bigoplus_{m\ge 0}\mathscr D^{\mathrm{ch}}_{\widetilde\CM,\hbar}(\widetilde\CM)_m \;
\end{align}
is a vertex superalgebra over the base ring $\mathbb C[\hbar]$. We thus construct a vertex superalgebra associated to the quiver data $(Q,\mathbf v, \mathbf w)$, which we denote $\mathsf D^{\mathrm{ch}}(\widetilde\CM)$, as the $\hbar=1$ specialization of the $\mathbb S$-finite part
\begin{align}\label{eqdef:quiverVSA-Local}
\mathsf D^{\mathrm{ch}}(\widetilde\CM):=\mathscr D^{\mathrm{ch}}_{\widetilde\CM,\hbar}(\widetilde\CM)_{\mathbb S\text{-fin}}/(\hbar-1)\;.
\end{align}
We claim, in particular, that there exists a natural vertex superalgebra map from $\CV[Q]$ \eqref{eqdef:quiverVSA}
\begin{equation}\label{naturalVSAmap-1}
    \CV[Q]\to\mathsf D^{\mathrm{ch}}(\widetilde\CM)~.
\end{equation}
We will analyze this in the next section.

\subsubsection{Map of quiver Poisson algebras}

The canonical filtration from Section \ref{sec:MapPoissonAlgebra} naturally extends to a sheaf $\mathscr D$ of $\hbar$-adic vertex algebras. Namely, we define the subsheaf $^{\hbar}F^p\mathscr D$ to be the sheafification of the sub-presheaf
\begin{align}
    U\mapsto {}^{\hbar}F^p(\mathscr D(U)) \;,
\end{align}
where ${}^{\hbar}F^p(\mathscr D(U))$ is defined the same way as \eqref{eq:canonical filtration}. This allows us to define the Zhu's $C_2$ sheaf
\begin{align}\label{eq:C2sheaf}
    \CR({\mathscr D}):={}^{\hbar}F^0\mathscr D/{}^{\hbar}F^1\mathscr D\;.
\end{align}
In the quiver vertex superalgebra case where $\mathscr D=\mathscr D^{\mathrm{ch} }_{\widetilde\CM,\hbar}$, 
\begin{align}
    \CR(\mathscr D^{\mathrm{ch}}_{\widetilde\CM,\hbar})/\hbar \CR(\mathscr D^{\mathrm{ch}}_{\widetilde\CM,\hbar})\cong \mathscr D^{\mathrm{ch}}_{\widetilde\CM,\hbar}/(\hbar \mathscr D^{\mathrm{ch}}_{\widetilde\CM,\hbar}+{}^{\hbar}F^2\mathscr D^{\mathrm{ch}}_{\widetilde\CM,\hbar}) \cong  \CR(\mathscr O_{J_\infty\widetilde\CM(Q,\mathbf v,\mathbf w|\mathbf u)})\cong  \mathscr O_{\widetilde\CM(Q,\mathbf v,\mathbf w|\mathbf u)}\;
\end{align} 
induces (see \cite[Section 10.3]{Coman:2023xcq}) an isomorphism 
\begin{align}
    \CR(\mathscr D^{\mathrm{ch}}_{\widetilde\CM,\hbar})\cong \mathscr O_{\widetilde\CM(Q,\mathbf v,\mathbf w|\mathbf u)}[\![\hbar]\!]\;.
\end{align}
We have an obvious inclusion 
\begin{align}
    {}^{\hbar}F^1(\mathscr D^{\mathrm{ch}}_{\widetilde\CM,\hbar}(\widetilde\CM))\subset 
    {}^{\hbar}F^1(\mathscr D^{\mathrm{ch}}_{\widetilde\CM,\hbar})(\widetilde\CM)\;,
\end{align}
and thus get the following natural maps between $\hbar$-adic Poisson algebras
\begin{align}
    \CR(\mathscr D^{\mathrm{ch}}_{\widetilde\CM,\hbar}(\widetilde\CM))\twoheadrightarrow \mathscr D^{\mathrm{ch}}_{\widetilde\CM,\hbar}(\widetilde\CM)/{}^{\hbar}F^1(\mathscr D^{\mathrm{ch}}_{\widetilde\CM,\hbar})(\widetilde\CM)\hookrightarrow \CR(\mathscr D^{\mathrm{ch}}_{\widetilde\CM,\hbar})(\widetilde\CM)\cong \mathscr O_{\widetilde\CM(Q,\mathbf v,\mathbf w|\mathbf u)}(\widetilde\CM)[\![\hbar]\!] \;,
\end{align}
where the first map is surjective and the second map is injective. Restricting to $\mathbb S$-finite part and specializing to $\hbar=1$,  we get a map between Poisson algebras
\begin{align}\label{eq:map from C_2 alg to function ring of quiver variety}
    \CR(\mathsf D^{\mathrm{ch}}(\widetilde\CM))\longrightarrow \mathscr O_{\widetilde\CM(Q,\mathbf v,\mathbf w|\mathbf u)}(\widetilde\CM) \;.
\end{align}

\subsubsection{Quiver sheaf of modules of \texorpdfstring{$\hbar$}{}-adic vertex superalgebras}

Having defined the module $\CV_\lambda[Q]$ \eqref{eqdef:quiverVSAmodule} of the quiver vertex superalgebra $\CV[Q]$ \eqref{eqdef:quiverVSA}, we define here its local analog.  
\begin{equation}
    \mathscr D^{\mathrm{ch}}_{\mathfrak R,\hbar,\lambda}=\mathscr D^{\mathrm{ch}}_{T^*\mathrm{Rep}(\mathbf v,\mathbf w|\mathbf u),\hbar}\widehat{\boxtimes} \mathscr H^{(\cdot,\cdot)_Q}_{\mathbb C^{Q_0},\hbar,\lambda} 
\end{equation}
is the localization of the module $\mathcal D_\lambda^{\mathrm{ch}}(\mathfrak R)$ \eqref{eqdef:prequotientquiverVOA-module}, where 
\begin{equation}
    \mathscr D^{\mathrm{ch}}_{\mathfrak R,\hbar,\lambda}(\widetilde{U})=\mathscr D^{\mathrm{ch}}_{T^*\mathrm{Rep}(\mathbf v,\mathbf w|\mathbf u),\hbar}(U')\widehat{\boxtimes} \mathcal H^{(\cdot,\cdot)_Q}_{\mathbb C^{Q_0},\hbar,\lambda} 
\end{equation}
for $\widetilde{U}=U'\times\mathcal{Z}$, with $U'\subset T^*\mathrm{Rep}(\mathbf v,\mathbf w|\mathbf u)$ and $\widetilde{U}\cap\widetilde\mu^{-1}_{\mathsf{s}}(0)$ identified with the preimage of $U\subset\widetilde{\mathcal{M}}$ under the quotient map \eqref{quotientmaps-super}. $\mathcal H_{\hbar,\lambda}:= \mathcal H^{(\cdot,\cdot)_Q}_{\mathbb C^{Q_0},\hbar,\lambda}$ is the Verma module of the $\hbar$-adic Heisenberg algebra $\mathcal H_{\hbar}:=\mathcal H^{(\cdot,\cdot)_Q}_{\mathbb C^{Q_0},\hbar}$. Similarly to the sheaf of BRST reductions $\mathscr D^{\mathrm{ch}}_{\widetilde\CM,\hbar}$ \eqref{def:sheaf of BRST reductions}, in \cite{Coman:2023xcq} we defined the module
\begin{align}\label{def:sheaf of BRST reductions module}
    \mathscr D^{\mathrm{ch}}_{\widetilde\CM,\hbar,\lambda}:= \mathscr H^{\infty/2+0}_{\hbar\mathrm{BRST}}(\mathfrak{gl}(\mathbf v),\mathscr D^{\mathrm{ch}}_{\mathfrak R,\hbar,\lambda})\;,
\end{align}
where, for $U\subset \widetilde\CM$ a good open affine subset, 
\begin{align}\label{eq:sheaf on affine module}
    \mathscr D^{\mathrm{ch}}_{\widetilde\CM,\hbar,\lambda}(U)=H^{\infty/2+0}_{\hbar\mathrm{BRST}}(\mathfrak{gl}(\mathbf v),\mathscr D^{\mathrm{ch}}_{\mathfrak R,\hbar,\lambda}(\widetilde U))\;.
\end{align}
The classical limit $\mathscr D^{\mathrm{ch}}_{\widetilde\CM,\hbar,\lambda}/\hbar \mathscr D^{\mathrm{ch}}_{\widetilde\CM,\hbar,\lambda}$ is isomorphic to a tautological line bundle $\mathcal{L}_\lambda$ on $J_\infty\widetilde{\mathcal{M}}(\mathbf v , \mathbf w | \mathbf u)$. Finally, the $\hbar=1$ specialization of the $\mathbb S$-finite part of $\mathscr D^{\mathrm{ch}}_{\widetilde\CM, \hbar,\lambda}$ defines cf. \cite[Section 13.2]{Coman:2023xcq} the module 
\begin{align}\label{eqdef:quiverVSA-Local-module}
\mathsf D^{\mathrm{ch}}_\lambda(\widetilde\CM):=\mathscr D^{\mathrm{ch}}_{ \widetilde\CM, \hbar,\lambda}(\widetilde\CM)_{\mathbb S\text{-fin}}/(\hbar-1)\;
\end{align}
of the vertex superalgebra $\mathsf D^{\mathrm{ch}}(\widetilde\CM)$ \eqref{eqdef:quiverVSA-Local}. We claim that there exists a natural map of vertex superalgebra modules, which is the analog of \eqref{naturalVSAmap-1}
\begin{equation}\label{naturalVSAModulemap-1}
    \CV_\lambda[Q]\to\mathsf D^{\mathrm{ch}}_\lambda(\widetilde\CM)~.
\end{equation}

\section{Vertex Superalgebra of Global Sections}\label{sec:VSA-globalsections}

Having recalled in Sections \ref{sec:VSA-quiverSetting} and \ref{sec:VSA-quiverSetting-Local} the construction of the vertex superalgebras $\CV(Q,\mathbf v,\mathbf w)$ \eqref{eqdef:quiverVSA} and $\mathsf D^{\mathrm{ch}}(\widetilde\CM)$ \eqref{eqdef:quiverVSA-Local} from \cite{Coman:2023xcq}, which are both chiralizations of the extended Nakajima quiver variety $\widetilde\CM$ \eqref{eqdef:ExtNakajimaQuiverSuperVariety}, we will now focus on the natural vertex superalgebra map \eqref{naturalVSAmap-1}
\begin{equation}\label{naturalVSAmap-2}
    \CV[Q]\to\mathsf D^{\mathrm{ch}}(\widetilde\CM)~.
\end{equation}
There exist, in particular, quivers and dimension data $(Q,\mathbf v,\mathbf w)$, for which the BRST cohomology $H^{\infty/2+n}_{\mathrm{BRST}}(\mathfrak{gl}(\mathbf v),\CVs(Q,\mathbf v,\mathbf w))$ vanishes for degree $n <0$, and where natural vertex superalgebra map $\CV[Q]\to\mathsf D^{\mathrm{ch}}(\widetilde\CM)$ is injective. This section makes this statement precise.

\subsection{Vanishing Theorem for the Negative Degree BRST Cohomologies}

\noindent{To ease formulae, throughout this section we will use the abbreviated notation $\mathfrak{g}=\gl (\mathbf v)$ and $\mathfrak{R}=T^*\mathrm{Rep}(\mathbf v,\mathbf w | \mathbf u)\times \mathcal{Z}$ \eqref{def:prequotient}. The following theorem is an adaptation of Theorem 2.3.3.1 \cite{arakawa2015localization}, where 
\begin{equation}
    \mathcal{O}_{J_\infty \mathfrak{R}} \otimes \Lambda^{\mathrm{vert}}(\mathfrak{g}\oplus\mathfrak{g}^\ast)
\end{equation}
is a vertex Poisson algebra; $\mathcal{O}_{J_\infty \mathfrak{R}}$ is a vertex Poisson algebra associated to the Poisson scheme $\mathfrak{R}$, and $\Lambda^{\mathrm{vert}}(\mathfrak{g}\oplus\mathfrak{g}^\ast)$ is the vertex Poisson algebra $\mathbb{C}[J_\infty T^\ast\Pi\mathfrak{g}]$. This contains is a globally defined element 
\begin{equation}\label{def:Arakawa:dchi}
    d_{\mathbf z} = \sum_i \tilde\mu^\ast(g_i)\otimes\phi^\ast_i - \frac{1}{2}\sum_{i,j,k}{\bf{1}}\otimes f_{ij}{}^k \phi_k\phi^\ast_i\phi^\ast_j
\end{equation}
where $i$ runs over basis elements $\{g_i\}$ of $\mathfrak{g}$, while $\{\phi_i\}$ and $\{\phi_i^\ast\}$ are a copy and dual copy of this basis  inside the exterior algebra $\Lambda(\mathfrak{g}\oplus\mathfrak{g}^\ast)\cong \mathbb{C}[T^\ast\Pi\mathfrak{g}]$, with $f_{ij}{}^k$ the structure constants relative to this basis. The operator $d_{\mathbf z}$ is a derivation and $(d_{\mathbf z})_{(0)}^2=0$. The vertex Poisson algebra $\mathcal{O}_{J_\infty \mathfrak{R}} \otimes \Lambda^{\mathrm{vert}}(\mathfrak{g}\oplus\mathfrak{g}^\ast)$ has a bi-grading where $\Pi\mathfrak{g}^\ast$ has bi-degree $(1,0)$, and $\Pi\mathfrak{g}$ has bi-degree $(0,-1)$. $\mathrm H^{\infty/2+n}_{d_{\mathbf z}} ( \mathcal{O}_{J_\infty \mathfrak{R}} (U) \otimes \Lambda^{\mathrm{vert}}(\mathfrak{g}\oplus\mathfrak{g}^\ast))$ is a (presheaf) cohomology vertex Poisson algebra for open affine subsets $U\in\mathfrak{R}$, and $\mathcal H^{\infty/2+n}_{d_{\mathbf z}} ( \mathcal{O}_{J_\infty \mathfrak{R}} \otimes \Lambda^{\mathrm{vert}}(\mathfrak{g}\oplus\mathfrak{g}^\ast))$ is the associated sheaf\footnote{With regard to the Theorem \ref{lem:vanishingHigherDegCoh-SemiClassical}, see \cite[Section 5.3]{kuwabara2017vertex} for Poisson BRST reduction in the case of vertex algebras associated to hypertoric varieties.}.  
}

\bigskip

{
\begin{thm}[cf. Theorem 2.3.3.1 \cite{arakawa2015localization}]\label{lem:vanishingHigherDegCoh-SemiClassical}
If $\widetilde\mu$ is jet-flat (property $(\widetilde P_2)$ in Proposition \ref{prop: various jet flatness}), then 
\begin{equation} 
\mathrm H^{\infty/2+n}_{\mathrm{BRST}}(\gl(\mathbf v),\mathcal{O}_{J_\infty \mathfrak{R}})=0 ~,  \qquad n<0~,
\end{equation}
and 
\begin{equation} 
\mathrm H^{\infty/2+0}_{\mathrm{BRST}}(\gl(\mathbf v),\mathcal{O}_{J_\infty \mathfrak{R}})\cong{\mathbb C}[J_\infty \widetilde{\mu}_{\mathsf s}^{-1}(0)]^{J_\infty\mathrm{GL}(\mathbf v)}\,.
\end{equation}
\end{thm}

\bigskip

\begin{proof}
In the notation of Section \ref{sec:BRST reduction}, the cohomology vertex Poisson algebra is
\begin{equation} 
\mathrm H^{\infty/2+n}_{\mathrm{BRST}}(\gl(\mathbf v),\mathcal{O}_{J_\infty \mathfrak{R}})=\mathrm H^{\infty/2+n}_{d_{\mathbf z}} ( \mathcal{O}_{J_\infty \mathfrak{R}} \otimes \Lambda^{\mathrm{vert}}(\mathfrak{g}\oplus\mathfrak{g}^\ast))  ~.
\end{equation}
The jet-flatness of $\tilde\mu$ implies that $\{\partial^n(\tilde\mu^\ast (g_i))\}$ is a regular sequence for all $n\in\mathbb{Z}_{\geq 0}$. The complex $( \mathcal{O}_{J_\infty \mathfrak{R}} \otimes \Lambda^{\mathrm{vert}}(\mathfrak{g}\oplus\mathfrak{g}^\ast),(d_{\mathbf z})_{(0)})$ has a bi-grading with $\Pi\mathfrak{g}^\ast$ degree $(1,0)$, and $\Pi\mathfrak{g}$ degree $(0,-1)$, and gives rise to a spectral sequence in terms of the Koszul complex $K_\bullet\left(\{\partial^n(\tilde\mu^\ast(g_i)\}, \mathcal{O}_{J_\infty\mathfrak{R}}(U)\right)=\mathcal{O}_{J_\infty\mathfrak{R}}(U)\otimes \Lambda^\bullet(\mathfrak{g}^\ast[t])$ and Koszul cohomology $H_\bullet\left(\{\partial^n(\tilde\mu^\ast(g_i)\}, \mathcal{O}_{J_\infty\mathfrak{R}}(U)\right)$ of $\mathcal{O}_{J_\infty\mathfrak{R}}(U)$ with respect to the regular sequence $\{\partial^n(\tilde\mu^\ast(g_i))\}$
\begin{itemize}
    \item $E^{p,q}_0(U)=K_q\left(\{\partial^n(\tilde\mu^\ast(g_i)\}, \mathcal{O}_{J_\infty\mathfrak{R}}(U)\right)\otimes\Lambda^p(\mathfrak{g}^\ast[t^{-1}]t^{-1})$\,;
    \item $E^{p,q}_1(U)=C^p \left( \mathfrak{g}[t], K_q\left(\{\partial^n(\tilde\mu^\ast(g_i)\}, \mathcal{O}_{J_\infty\mathfrak{R}}(U)\right) \right)$ is the Chevalley complex of $\mathfrak{g}[t]$ with coefficients in the Koszul cohomology;
    \item $E^{p,q}_2(U)=H^p \left( \mathfrak{g}[t], K_q\left(\{\partial^n(\tilde\mu^\ast(g_i)\}, \mathcal{O}_{J_\infty\mathfrak{R}}(U)\right) \right)$ is the cohomology of $\mathfrak{g}[t]$ with coefficients in the Koszul cohomology.
\end{itemize}
The sheaf $\mathcal{O}_{J_\infty \mathfrak{R}} \otimes \Lambda^{\mathrm{vert}}(\mathfrak{g}\oplus\mathfrak{g}^\ast)$ is a sheaf of $\mathcal{O}_{ \mathfrak{R} \times T^\ast\Pi\mathfrak{g}}$-modules graded by the order of the jet (i.e.\  the number of times $\partial$ is applied \cite{arakawa2015localization}). Each graded component has a bounded cohomological bi-degree, and the spectral sequence converges to $\mathrm H^{\infty/2+p+q}_{d_{\mathbf z}} ( \mathcal{O}_{J_\infty \mathfrak{R}} \otimes \Lambda^{\mathrm{vert}}(\mathfrak{g}\oplus\mathfrak{g}^\ast))$, supported on $J_\infty \tilde\mu^{-1}(0)\cap U$. This proves the  vanishing of $\mathrm H^{\infty/2+n}_{\mathrm{BRST}}(\gl(\mathbf v),\mathcal{O}_{J_\infty \mathfrak{R}})$ in negative degree $n$. 

The second page of spectral sequence degenerates and becomes the Chevalley-Eilenberg complex for the $\mathfrak{g}[t]$ action on $\widetilde{\mu}^{-1}(0)$. Taking $0^{th}$ cohomology of $E_2$, we find the $\mathfrak{g}[t]$ invariant part of $\widetilde{\mu}^{-1}(0)$.
\end{proof}
}

\bigskip

{
Theorem \ref{lem:vanishingHigherDegCoh-SemiClassical} is a semiclassical analog of the following theorem, with a key role in its proof. The chiralization of $\mathcal{O}_{J_\infty \mathfrak{R}}$ is the vertex superalgebra $\mathcal D^{\mathrm{ch}}(\mathfrak R)$ \eqref{eqdef:prequotientquiverVOA}, as defined in Section \ref{sec:VSA-quiverSetting}, where 
\begin{equation}
    d^{ch}:=\sum_{\ell} \left(\tilde\mu^{ch}_\ell +\hbar^{-1} h_\ell \cdot\mathrm{id} \right)\otimes c_\ell(z) -\frac{1}{2} \hbar^{-1}\sum_{i,j,k} {\mathbf 1}\otimes f_{ij}{}^k (b_{\ell,k})_{(-1)}((c_{\ell,i})_{(-1)}(c_{\ell,j}))(z)
\end{equation}
is the quantum counterpart of $d_{\mathbf z}$ \eqref{def:Arakawa:dchi}, and $\ell\in Q_0$. Following \cite[Theorem 2.3.5.1]{arakawa2015localization}, the BRST cohomology $\mathrm H^{\infty/2+n}_{\mathrm{BRST}}(\gl(\mathbf v),\mathcal D^{\mathrm{ch}} \mathfrak{R})$ vanishes in negative degree $n$.

\bigskip 

\begin{thm}[cf. Theorem 2.3.5.1 \cite{arakawa2015localization}]\label{thm:vanishingHigherDegCoh}
If $\widetilde\mu$ is jet-flat (property $(\widetilde P_2)$ in Proposition \ref{prop: various jet flatness}), then 
\begin{equation} 
\mathrm H^{\infty/2+n}_{\mathrm{BRST}}(\gl(\mathbf v),\mathcal D^{\mathrm{ch}}(\mathfrak{R}))=0~, \qquad n<0~.
\end{equation}
\end{thm}
\begin{proof}
By definition, for open affine subsets $U\in J_\infty \mathfrak{R}$
\begin{equation} 
\mathrm H^{\infty/2+n}_{\mathrm{BRST}}(\gl(\mathbf v),\mathscr D^{\mathrm{ch}}_{\mathfrak{R},\hbar}(U))=H^{\infty/2+n}_{d^{ch}}(\mathscr D^{\mathrm{ch}}_{\mathfrak{R},\hbar}(U)\otimes\CV_{\mathsf b\mathsf c}(\mathfrak g)_\hbar) 
\end{equation}
is an asymptotic vertex superalgebra, with a filtration by powers of $\hbar$ preserved by the differential $d^{ch}$. This gives rise to a spectral sequence, which does not converge to $H^{\infty/2+n}_{d^{ch}}(\mathscr  D^{\mathrm{ch}}_{\mathfrak{R},\hbar}(U)\otimes\CV_{\mathsf b\mathsf c}(\mathfrak g)_\hbar)$, but where the $\hbar$-adic completeness of $\mathscr D^{\mathrm{ch}}_{\mathfrak{R},\hbar}(U)\otimes\CV_{\mathsf b\mathsf c}(\mathfrak g)_\hbar$ together with Theorem \ref{lem:vanishingHigherDegCoh-SemiClassical} imply the analogous statements
\begin{equation}
    \mathrm H^{\infty/2+n}_{d_{\mathbf z}} ( \mathcal{O}_{J_\infty \mathfrak{R}}(U) \otimes \Lambda^{\mathrm{vert}}(\mathfrak{g}\oplus\mathfrak{g}^\ast)) =0 ~\Rightarrow ~ H^{\infty/2+n}_{d^{ch}}(\mathscr  D^{\mathrm{ch}}_{\mathfrak{R},\hbar}(U)\otimes\CV_{\mathsf b\mathsf c}(\mathfrak g)_\hbar) =0 ~.
\end{equation}
\end{proof}
}

\subsection{Vertex superalgebra maps}

{Recall from equation \eqref{def:BRST-cohomology} that the relative BRST cohomology of a vertex superalgebra $V$ is 
\begin{align}
    H^{\infty/2+\bullet}_{\mathrm{BRST}}(\mathfrak g,V):=H^{\bullet}(C(\mathfrak g, V), \mathcal{Q})\;,
\end{align}
where $C(\mathfrak g, V)$ is a canonically defined chain complex. In this context, the vertex superalgebra $V$ is associated to a quiver, as described in Section \ref{sec:generalconstructionSVA}. Li's filtration\footnote{This filtration was introduced in Section \ref{sec:MapPoissonAlgebra}, compatible with the differential $\mathcal{Q}(F^pC)\subset F^p(C)$.} $F^p C(\mathfrak g, V):= F^p C \subset C(\mathfrak g, V)$ on the complex $C(\mathfrak g, V)$ induces, in the notation of \cite{arakawa2017introduction}, a filtration on the BRST cohomology
\begin{equation}\label{def:inducedfiltrationoncohom}
    G^p H^i(C(\mathfrak g, V), \mathcal{Q}):= \mathrm{Im} \left( H^i (F^p C,\mathcal{Q}) \rightarrow H^i (C(\mathfrak g, V), \mathcal{Q}) \right) ~,
\end{equation}
which we denote by $G^p H^i(C)$ in a compact manner, and where 
\begin{equation}
    \mathrm{gr}_G^p H^0(C(\mathfrak g, V), \mathcal{Q}) := G^p H^0(C(\mathfrak g, V), \mathcal{Q}) / G^{p+1} H^0(C(\mathfrak g, V), \mathcal{Q}) ~. 
\end{equation}

\bigskip

\begin{lem}\label{lem:grH-vs-Hgr}
    If $\widetilde\mu$ is flat, and $\widetilde\mu^{-1}(0)$ is reduced and irreducible and has rational singularities (i.e.\  property $(\widetilde P_1)$ in Proposition \ref{prop: various jet flatness}), then the following natural vertex Poisson algebra map is injective
    \begin{equation}
        \mathrm{gr}_G\, H^{\frac{\infty}{2}+0}_{\mathrm{BRST}} (\mathfrak{g},V) \rightarrow H^{\frac{\infty}{2}+0}_{\mathrm{BRST}} ( \mathfrak{g}, \mathrm{gr}_F\, V ) ~.
    \end{equation}
\end{lem}

\bigskip

\begin{proof}
    Using the above notation, the short exact sequence 
    \begin{equation}
        0 \rightarrow F^{p+1} C \rightarrow F^{p} C \rightarrow \mathrm{gr}^p_F C \rightarrow 0
    \end{equation}
    induces a long exact sequence on cohomology
    \begin{equation}
        \cdots H^{-1}(\mathrm{gr}^{p}_F C,\mathcal{Q}) \rightarrow H^{0}(F^{p+1}C,\mathcal{Q}) \rightarrow H^{0}(F^{p}C,\mathcal{Q}) \rightarrow H^{0}(\mathrm{gr}^{p}_F C,\mathcal{Q}) \rightarrow H^{1}(F^{p+1}C,\mathcal{Q}) \rightarrow \cdots ~.
    \end{equation}
    From the definition \eqref{def:inducedfiltrationoncohom} of the filtration induced by Li's filtration on the relative BRST cohomology, {the map $H^0(F^pC,\mathcal{Q})\rightarrow G^p H^0(C(\mathfrak{g},V),\mathcal{Q})$ is surjective} for all $p\in\mathbb{Z}_{\geq 0}$. At the same time, $\mathrm{Ker}\left( H^i(F^pC,\mathcal{Q})\rightarrow H^i (C(\mathfrak g, V), \mathcal{Q}) \right)$ is trivial for all $p\in\mathbb{Z}_{\geq 0}$, so the map $H^0(F^pC,\mathcal{Q}) \rightarrow G^p H^0 (C(\mathfrak{g},V),\mathcal{Q})$ is an isomorphism for all $p\in\mathbb{Z}_{\geq 0}$. 

    \bigskip

    In Proposition \ref{prop: various jet flatness}, the property $(\widetilde P_1)$ implies that the moment map $\widetilde\mu$ is jet-flat, which is property $(\widetilde P_2)$. Consequently, Theorem \ref{lem:vanishingHigherDegCoh-SemiClassical} implies that the negative degree cohomology $H^{n<0}(\mathrm{gr}_F C,\mathcal{Q})$ vanishes. We therefore have the following long exact sequence in cohomology, where the left-most vertical maps are isomorphisms

    \bigskip

\[\adjustbox{scale=0.87,center}{
\begin{tikzcd}
	{0=H^{-1}(\mathrm{gr}^{p}_F C,\mathcal{Q})} & {H^{0}(F^{p+1}C,\mathcal{Q})} & {H^{0}(F^{p}C,\mathcal{Q})} & {H^{0}(\mathrm{gr}^{p}_F C,\mathcal{Q})} & {H^{1}(F^{p+1}C,\mathcal{Q})} & \cdots \\
	0 & {G^{p+1}H^{0}(C,\mathcal{Q})} & {G^{p}H^{0}(C,\mathcal{Q})} & {\mathrm{gr}^{p}_{G} H^{0}(C,\mathcal{Q})} & 0
	\arrow[from=1-2, to=1-3]
	\arrow[from=1-3, to=1-4]
	\arrow[from=1-1, to=1-2]
	\arrow[from=1-5, to=1-6]
	\arrow[from=1-4, to=1-5]
	\arrow[from=2-1, to=2-2]
	\arrow[shift right=2, two heads, from=1-2, to=2-2]
	\arrow[shift right=2, two heads, from=1-3, to=2-3]
	\arrow[from=2-2, to=2-3]
	\arrow[from=2-3, to=2-4]
	\arrow[shift right=2, hook, from=2-4, to=1-4]
	\arrow[shift right=2, hook, from=2-3, to=1-3]
	\arrow[shift right=2, hook, from=2-2, to=1-2]
    \arrow[from=2-4, to=2-5]
\end{tikzcd}}\]

\bigskip

and the second line is a short exact sequence. The map $\alpha^p$ in the following part of the top sequence 
\begin{equation}
        0 \longrightarrow H^{0}(F^{p+1}C,\mathcal{Q}) \stackrel{\alpha^p}{\longrightarrow} H^{0}(F^{p}C,\mathcal{Q}) \longrightarrow H^{0}(\mathrm{gr}^{p}_F C,\mathcal{Q}) \longrightarrow  \cdots ~
\end{equation}
is injective, and therefore 
\begin{equation}
    \mathrm{gr}^{p}_{G} H^{0}(C,\mathcal{Q})\cong H^{0}(F^{p}C,\mathcal{Q}) / H^{0}(F^{p+1}C,\mathcal{Q})\cong \mathrm{coker}(\alpha^p) \hookrightarrow  H^{0}(\mathrm{gr}^{p}_F C,\mathcal{Q}) ~.
\end{equation}
Taking the direct sum over $p\in\mathbb{Z}_{\geq 0}$ induces an injection 
\begin{equation}
\mathrm{gr}_{G} H^{0}(C(\mathfrak{g},V),\mathcal{Q})\hookrightarrow  H^{0}(\mathrm{gr}_{F}C,\mathcal{Q}) = H^{\frac{\infty}{2}+0}_{\textrm{BRST}}(\mathfrak{g},\mathrm{gr}_{F}V) ~.
\end{equation}
\end{proof}
}

\bigskip

\begin{theorembox}
\begin{thm}\label{thm:VOAinjectivity}
If $\widetilde\mu$ is flat, and $\widetilde\mu^{-1}(0)$ is reduced and irreducible and has rational singularities (i.e.\  property $(\widetilde P_1)$ in Proposition \ref{prop: various jet flatness}), then the natural vertex superalgebra map \eqref{naturalVSAmap-2} $$\mathcal V(\mathbf v,\mathbf w | \mathbf u)\stackrel{f}{\longrightarrow}\mathsf D^{\mathrm{ch}}(\widetilde{\mathcal M}(\mathbf v,\mathbf w | \mathbf u))$$ is injective.
\end{thm}
\end{theorembox}

\begin{proof}
    If $\mathcal V(\mathbf v,\mathbf w | \mathbf u)\stackrel{f}{\longrightarrow}\mathsf D^{\mathrm{ch}}(\widetilde{\mathcal M}(\mathbf v,\mathbf w | \mathbf u))$ is a vertex superalgebra homomorphism, then it respects the canonical filtration induced by Li's filtration and induces a homomorphism 
    \begin{equation}
        \mathrm{gr}_G\mathcal V(\mathbf v,\mathbf w | \mathbf u)\stackrel{\mathrm{gr}f}{\longrightarrow} \mathrm{gr}_G \mathsf D^{\mathrm{ch}}(\widetilde{\mathcal M}(\mathbf v,\mathbf w | \mathbf u))
    \end{equation}
    on the associated graded algebras introduced above. {In order to prove the injectivity of the map $\mathcal V(\mathbf v,\mathbf w | \mathbf u)\stackrel{f}{\longrightarrow}\mathsf D^{\mathrm{ch}}(\widetilde{\mathcal M}(\mathbf v,\mathbf w | \mathbf u))$, it suffices to prove the injectivity of $\mathrm{gr}f$ at the level of the associated graded algebras.} By the definition of the vertex superalgebra $\mathcal V(\mathbf v,\mathbf w | \mathbf u)=H^{\frac{\infty}{2}+0}_{\mathrm{BRST}}(\mathfrak{g}, \mathcal D^{\mathrm{ch}}\mathfrak{R})$, 
    \begin{equation}
        \mathrm{gr}_G \mathcal V(\mathbf v,\mathbf w | \mathbf u) = \mathrm{gr}_G H^{\frac{\infty}{2}+0}_{\mathrm{BRST}}(\mathfrak{g}, \mathcal D^{\mathrm{ch}}\mathfrak{R})\hookrightarrow H^{\frac{\infty}{2}+0}_{\textrm{BRST}}(\mathfrak{g},\mathrm{gr}_{F} \mathcal D^{\mathrm{ch}}\mathfrak{R}) \cong H^{\frac{\infty}{2}+0}_{\textrm{BRST}}(\mathfrak{g},\mathcal{O} (J_\infty \mathfrak{R}) )~,
    \end{equation}
    where the injection into $ H^{\frac{\infty}{2}+0}_{\textrm{BRST}}(\mathfrak{g},\mathrm{gr}_{F} \mathcal D^{\mathrm{ch}}\mathfrak{R})$ follows from Lemma \ref{lem:grH-vs-Hgr}, and for the right-most identification we have used the vertex Poisson algebra isomorphism $\mathrm{gr}_{F} \mathcal D^{\mathrm{ch}}\mathfrak{R}\cong \mathcal{O} (J_\infty \mathfrak{R})$, which is due, cf. \cite[Lemma 4.4]{arakawa2024arc}, to $\mathcal D^{\mathrm{ch}}\mathfrak{R}$ having a PBW basis\footnote{ \cite{2018arXiv180206533A} has in fact shown that the surjective homomorphism of differential graded rings $\mathcal{O}(J_\infty \tilde{X}_\mathcal{V})\to \mathrm{gr}_F\mathcal{V}$, for $\mathcal{V}$ any vertex algebra and $\tilde{X}_\mathcal{V}$ its associated scheme, is an isomorphism at the level of varieties if $\mathcal{V}$ is quasi-lisse, that is if its associated variety has finitely many symplectic leaves. In the case of Nakajima quiver varieties and the representation spaces $\mathfrak{R}$ of quivers that we consider, these have finitely many symplectic leaves \cite{bellamy2016symplectic}.
    }. 
    Under the conditions $(\widetilde{P}_1)$ of Proposition \ref{prop: various jet flatness} for the flatness of $\tilde\mu$ and $\tilde\mu^{-1}(0)$ reduced, irreducible and with rational singularities, with the isomorphism $ {H^{\frac{\infty}{2}+0}_{\mathrm{BRST}}}({\mathfrak{g}},\mathcal{O} (J_\infty \mathfrak{R})) \cong {\mathbb C}[J_\infty \widetilde{\mu}_{\mathsf s}^{-1}(0)]^{J_\infty\mathrm{GL}(\mathbf v)}$ from Theorem  \ref{lem:vanishingHigherDegCoh-SemiClassical}, the following natural map 
    \begin{equation}
     {H^{\frac{\infty}{2}+0}_{\mathrm{BRST}}}({\mathfrak{g}},\mathcal{O} (J_\infty \mathfrak{R})) \cong {\mathbb C}[J_\infty \widetilde{\mu}_{\mathsf s}^{-1}(0)]^{J_\infty\mathrm{GL}(\mathbf v)}\longrightarrow \Gamma(J_\infty\widetilde{\mathcal M}(\mathbf v,\mathbf w|\mathbf u),\mathcal O_{J_\infty \widetilde{\mathcal M}(\mathbf v,\mathbf w|\mathbf u)})~
    \end{equation}
    was demonstrated by Theorem \ref{thm:semiclassical-injectivity} to be injective, where\footnote{When a vertex algebra $\mathcal{V}$ has a PBW basis, $\Gamma(J_\infty \tilde{X}_\mathcal{V},\mathcal{O}_{J_\infty \tilde{X}_\mathcal{V}} )\cong \mathrm{gr}_F\mathcal{V}$.}   
    \begin{equation} 
    \Gamma(J_\infty\widetilde{\mathcal M}(\mathbf v,\mathbf w|\mathbf u),\mathcal O_{J_\infty \widetilde{\mathcal M}(\mathbf v,\mathbf w|\mathbf u)}) \cong \mathrm{gr}_G\mathsf D^{\mathrm{ch}}(\widetilde{\mathcal M}(\mathbf v,\mathbf w | \mathbf u))~.
    \end{equation}
\end{proof}

\medskip

Theorem \ref{thm:VOAinjectivity} is a quantum counterpart of Theorem \ref{thm:semiclassical-injectivity}, which holds at the semiclassical level. It is key to proving Conjecture 10.1 of \cite{Coman:2023xcq}, which is Theorem \ref{thm:vanishing and embedding} below, and has a counterpart Theorem \ref{thm:VOA-module-injectivity} for vertex superalgebra modules. 

\medskip

\begin{theorembox}
\begin{thm}\label{thm:VOA-module-injectivity}
If $\widetilde\mu$ is flat, and $\widetilde\mu^{-1}(0)$ is reduced and irreducible and has rational singularities (i.e.\  property $(\widetilde P_1)$ in Proposition \ref{prop: various jet flatness}), then the natural map \eqref{naturalVSAModulemap-1} of vertex superalgebra modules $$\mathcal V_\lambda(\mathbf v,\mathbf w | \mathbf u)\stackrel{f_\lambda}{\longrightarrow}\mathsf D_\lambda^{\mathrm{ch}}(\widetilde{\mathcal M}(\mathbf v,\mathbf w | \mathbf u))$$ is injective.
\end{thm}
\end{theorembox}

\bigskip

\begin{proof}
    The proof is largely similar to that of Theorem \ref{thm:VOAinjectivity}. Taking the associated graded of both modules $V_\lambda(\mathbf v,\mathbf w | \mathbf u)$ and $D_\lambda^{\mathrm{ch}}(\widetilde{\mathcal M}(\mathbf v,\mathbf w | \mathbf u))$ gives on the left hand side the injection $\mathrm{gr}_G \mathcal V(\mathbf v,\mathbf w | \mathbf u) \hookrightarrow {\mathbb C}[J_\infty \widetilde{\mu}_{\mathsf s}^{-1}(0)\, \big|\, h^i_{(0)}\to \lambda_i \, , i\in Q_0 ]^{J_\infty\mathrm{GL}(\mathbf v)}$, where the image of the zero modes of the Heisenberg fields $h^i$ is fixed to the values $\lambda_i$ in the arc space $J_\infty \widetilde{\mu}_{\mathsf s}^{-1}(0)$, and the resulting ring of functions is isomorphic to ${\mathbb C}[J_\infty \widetilde{\mu}_{\mathsf s}^{-1}(0)]^{J_\infty\mathrm{GL}(\mathbf v)}$. Similarly, on the right hand side, 
    \begin{equation} 
    \Gamma(J_\infty\widetilde{\mathcal M}(\mathbf v,\mathbf w|\mathbf u),\mathcal O_{J_\infty \widetilde{\mathcal M}(\mathbf v,\mathbf w|\mathbf u)}) \cong  \Gamma(J_\infty\widetilde{\mathcal M}(\mathbf v,\mathbf w|\mathbf u),\mathcal{L}_\lambda) \cong \mathrm{gr}_G\mathsf D_\lambda^{\mathrm{ch}}(\widetilde{\mathcal M}(\mathbf v,\mathbf w | \mathbf u))~,
    \end{equation}
    where $\mathcal{L}_\lambda$ is a line bundle over $J_\infty\widetilde{\mathcal{M}}$ locally identified with $\mathcal{O}_{J_\infty\widetilde{\mathcal{M}}}$,  whereby the injectivity of $\mathrm{gr}_G\mathcal V_\lambda(\mathbf v,\mathbf w | \mathbf u)\stackrel{\mathrm{gr}f_\lambda}{\longrightarrow} \mathrm{gr}_G \mathsf D_\lambda^{\mathrm{ch}}(\widetilde{\mathcal M}(\mathbf v,\mathbf w | \mathbf u))$ follows from Theorem \ref{thm:semiclassical-injectivity}.
\end{proof}

\bigskip

\begin{theorembox}
\begin{thm}\label{thm:vanishing and embedding}
Let $Q$ be a quiver and choose the gauge and framing dimensions $\mathbf v,\mathbf w$ such that $\mathbf w-\mathsf C\mathbf v\in \mathbb Z^{Q_0}_{\ge 0}$. Assume moreover that one of the following conditions holds:
\begin{itemize}
    \item $Q$ is totally negative, i.e.\  $\mathsf C_{ij}<0$ for every pair $i,j\in Q_0$,
    \item $Q$ is a Dynkin quiver, and $\sum_{i\in Q_0}( v_i-1)\le 2$,
    \item $Q$ is the $A_1$ quiver,
    \item $Q$ is the Jordan quiver, and $ w_i\ge  v_i$,
    \item $\forall i\in Q_0$, $ v_i=1$,
\end{itemize}
then we have
\begin{enumerate}
    \item $H^{\infty/2+n}_{\mathrm{BRST}}(\mathfrak{gl}(\mathbf v),\mathcal{D}^{\mathrm{ch}}(\mathfrak{R}))$ vanishes for $n<0$,
    \item the natural vertex superalgebra map $\CV(Q,\mathbf v,\mathbf w)\to \mathsf D^{\mathrm{ch}}(\widetilde\CM)$ is injective.
\end{enumerate}
\end{thm}
\end{theorembox}
\begin{proof}
    From Theorem \ref{thm:vanishingHigherDegCoh}, the first condition  
    $$H^{\infty/2+n}_{\mathrm{BRST}}(\mathfrak{gl}(\mathbf v),\mathcal{D}^{\mathrm{ch}}(\mathfrak{R}))=0~, \quad n<0$$ is satisfied if the quiver and dimension data $(Q,\mathbf v,\mathbf w)$ are such that the extended moment map $\widetilde\mu$ has property $(\widetilde P_2)$ from Proposition \ref{prop: various jet flatness} ($\widetilde\mu$ is jet-flat). From Theorem \ref{thm:VOAinjectivity}, the second condition, namely the injectivity of the natural vertex superalgebra map $\CV(Q,\mathbf v,\mathbf w)\to \mathsf D^{\mathrm{ch}}(\widetilde\CM)$, holds under the stronger condition $(\widetilde P_1)$ from Proposition \ref{prop: various jet flatness} (that $\widetilde\mu$ is flat, and $\widetilde\mu^{-1}(0)$ is reduced and irreducible and has rational singularities), which in turn implies property $(\widetilde P_2)$. Therefore, in order to prove this theorem, it suffice to demonstrate that property $(\widetilde P_1)$ holds in each of the cases:
    \begin{itemize}
    \item  $Q$ totally negative (see Example \ref{eq:totally negative quiver}) has property $(P_1)$ from \cite{2022arXiv220914791V}, which implies property $(\widetilde P_1)$ in Proposition \ref{prop: various jet flatness};
    \item  $Q$ a Dynkin quiver, with $\sum_{i\in Q_0}( v_i-1)\le 2$, has property $(\widetilde P_1)$ from Theorem \ref{thm: Dynkin quiver tilde P_1, II};
    \item $Q$ the $A_1$ quiver has property $(\widetilde P_1)$ from Example \ref{ex: A_1 quiver tilde P_1};
    \item $Q$ the Jordan quiver, with $ w_i\ge  v_i$ has property $(\widetilde P_1)$ from Theorem \ref{thm: Jordan quiver tilde P_1, III}
    \item $\forall i\in Q_0$, with $ v_i=1$, represents the abelian case where the Nakajima quiver variety is a hypertoric variety; this case has been studied in \cite{kuwabara2017vertex} and has property $(P_1)$ cf. Lemma 3.3 in loc. cit. (Lemma 4.7 \cite{Bellamy2010OnDQ}), together with Propositions \ref{prop: flat moment map} and \ref{prop: flat moment map 2}.  
\end{itemize}
\end{proof}

\bigskip

\begin{remark}
    Under the assumption that $\mu^{-1}(0)^{\mathrm{reg}}$ is nonempty, Theorem \ref{thm:semiclassical-injectivity2} proves that the natural map  
\begin{align*}
    \mathbb C[\widetilde{\mathcal M}^0(\mathbf v,\mathbf w|\mathbf u)]\longrightarrow \Gamma(\widetilde{\mathcal M}(\mathbf v,\mathbf w|\mathbf u),\mathcal O_{\widetilde{\mathcal M}(\mathbf v,\mathbf w|\mathbf u)})
\end{align*}
is an isomorphism. In turn, this proves that \cite[Conjecture 10.2]{Coman:2023xcq} holds under this assumption, specifically that the map \eqref{eq:map from C_2 alg to function ring of quiver variety} 
\begin{align}
    \CR(\mathsf D^{\mathrm{ch}}(\widetilde\CM))\longrightarrow \mathscr O_{\widetilde\CM(Q,\mathbf v,\mathbf w|\mathbf u)}(\widetilde\CM) \;
\end{align}
is an isomorphism of Poisson algebras.
\end{remark}

\newpage

\appendix

\section{Étale Slice and Auxiliary Quivers}\label{sec:étaleSlice}

In this appendix, we review and generalize the étale slice construction in \cite{crawley2003normality}, which we employ in Section \ref{Conditions for moment map jet-flatness}. From a framed quiver data $(Q,\mathbf v,\mathbf w)$, we canonically associate an unframed quiver data $(Q^+,\mathbf v^+)$ using the construction in Remark \ref{rmk: CB trick}, then \cite[Section 3]{crawley2003normality} gives a quadruple $(A,M,\omega,\zeta)$, where $A$ is the semisimple algebra 
\begin{align*}
    A=\bigoplus_{i\in Q^+_0}\End(V^+_i)=\mathbb C\oplus\bigoplus_{i\in Q_0}\End(V_i),
\end{align*}
$M$ is the $A$-$A$ bimodule $T^*\mathrm{Rep}(\mathbf v,\mathbf w)$, $\omega$ is the standard symplectic form on the cotangent bundle $T^*\mathrm{Rep}(\mathbf v,\mathbf w)$, and $\zeta\in A^*$ is a trace function i.e.\  $\zeta([a,b])=0$ for $a,b\in A$. $\omega$ is balanced in the sense that $\omega(xa,y)=\omega(x,ay)$ for $a\in A$ and $x,y\in M$. The moment map $\mu:M\to A^\ast$ is defined by $\mu(m)(a)=\tfrac{1}{2}\omega(m,[a,m])$, where $m\in M$ and $a\in A$. For a subspace $U\subset M$, we denote $U^\perp = \{m\in M| \omega(m,u)=0, ~ \forall u\in U\}$; if $\omega$ is nondegenerate on $U$, then $\dim U^\perp = \dim M - \dim U$ and $U^{\perp \perp}=U$. 

\smallskip Fix $x\in \mu^{-1}(\zeta)$ such that it has a closed $G$-orbit where $G=\GL(\mathbf v)$, then $[A,x]$ is an isotropic subspace of $M$ \cite[Lemma 4.1]{crawley2003normality} and $$A_x:=\{a\in A: ax=xa\}$$ is a semisimple subalgebra of $A$ \cite[Lemma 4.2]{crawley2003normality}. Denote by $G_x$ the stabilizer of $x$ in $G$; this is the group of units of $A_x$. The construction of an étale slice can be summarized as follows.
\begin{itemize}
    \item Choose an $A_x$-$A_x$-bimodule complement $L$ to $A_x$ in $A$, so that $A=A_x\oplus L$.
    \item Choose a coisotropic $A_x$-$A_x$-bimodule complement $C$ to $[A, x]$ in $M$, so $M = [A, x] \oplus C$, the existence of $C$ is guaranteed by \cite[Corollary 2.3]{crawley2003normality}.
    \item Let $W:=C\cap [A,x]^\perp$. Let $\omega_x$ be the restriction of $\omega$ to $W$. Then $\omega_x$ is nondegenerate, and $W$ is a $A_x$-$A_x$ bimodule with balanced symplectic form $\omega_x$ \cite[Lemma 4.3]{crawley2003normality}.
    \item Let $\mu_x: M\to A^*_x$ be the composition of $\mu:M\to A^*$ with the restriction map $A^*\to A^*_x$. Let $\hat\mu: W\to A^*_x$ be the restriction of $\mu_x$ to $W$, then $\hat\mu$ is the moment map for $W$ \cite[Lemma 4.3]{crawley2003normality}.
    \item  Define $\nu:C\to L^*$ by $\nu(c)(\ell)=\omega(c,[\ell, x])+\omega(c,\ell c)$. Thus
    \begin{align*}
        \mu(x+c)(a+\ell)=\zeta(a+\ell)+\mu_x(c)(a)+\nu(c)(\ell)
    \end{align*}
    for $c\in C$, $a\in A_x$ and $\ell\in L$.
\end{itemize}

\begin{defn}
    For $G$ a reductive group acting on an affine variety $X$, with $x\in X$ a point with closed $G$-orbit and stabilizer $G_x$, an étale slice is a $G_x$-invariant, locally closed, affine subvariety $S\subset X$ containing $x$, such that the induced $G$-equivariant morphism $$G\times_{G_x} S\to X$$ is strongly étale onto its image. A morphism of schemes is étale if it is flat and unramified. A $G$-equivariant morphism $f:X\to Y$  is strongly étale if $f/G:X\sslash G\to Y\sslash G$ is étale and $f$, $f/G$, and the quotient morphism induce a $G$-isomorphism $X\cong Y\times_{Y\sslash G} (X\sslash G)$), which is a $G$-saturated affine open subset $U\subset X$.  
\end{defn}

\begin{proposition}[{\cite[Lemma 4.4-4.8, Theorem 4.9]{crawley2003normality}}]\label{prop: étale slice}
$\nu:C\to L^*$ is smooth at $0$, thus $\nu^{-1}(0)$ is smooth at $0$. The tangent space of $\nu^{-1}(0)$ at $0$ is $W$. Moreover, we have the following.
\begin{itemize}
    \item[(1)] For $G\times_{G_x} C = (G\times C)\sslash G_x$, the map $\phi:G\times_{G_x} C\to M$ such that $\phi(g,c)=g(x+c)$ is étale at $(1,0)$, and $\phi^{-1}(\mu^{-1}(\zeta))=G\times_{G_x}(\mu_x^{-1}(0)\cap \nu^{-1}(0))$.
    \item[(2)] There exists an $A_x$-$A_x$-bimodule decomposition $C=C^\perp\oplus W$. Denote by $\varphi: \nu^{-1}(0)\to W$ the restriction of the projection $C\to W$ to $\nu^{-1}(0)$, then $\varphi$ is étale at $0$ and it is $G_x$-equivariant, and $\varphi^{-1}(\hat\mu^{-1}(0))=\mu_x^{-1}(0)\cap \nu^{-1}(0)$.
    \item[(3)] It follows from (1) and (2) that the maps $\phi$ and $\mathrm{id}_G\times\varphi$ in the following diagram are étale at $(1,0)$
    \begin{equation*}
    \begin{tikzcd}
        & G\times _{G_x}(\mu_x^{-1}(0)\cap \nu^{-1}(0)) \ar[dl,"\phi"]\ar[dr,"\mathrm{id}_G\times\varphi" '] &\\
        \mu^{-1}(\zeta) & &  G\times _{G_x}\hat\mu^{-1}(0)
    \end{tikzcd}~.
    \end{equation*}
    The induced maps between quotients are étale at $0$
    \begin{equation*}
    \begin{tikzcd}
        & \mu_x^{-1}(0)\cap \nu^{-1}(0)\sslash G_x \ar[dl,"\phi/ G_x"]\ar[dr,"\varphi/ G_x" '] &\\
        \mu^{-1}(\zeta)\sslash G & & \hat\mu^{-1}(0)\sslash G_x
    \end{tikzcd}~.
    \end{equation*}
    Therefore the neighborhood of $x\in \mu^{-1}(\zeta)$ is isomorphic to the neighborhood of $(1,0)\in  G\times _{G_x}\hat\mu^{-1}(0)$ in étale topology, and the neighborhood of $x\in \mu^{-1}(\zeta)\sslash G$ is isomorphic to the neighborhood of $0\in \hat\mu^{-1}(0)\sslash G_x$ in étale topology.
\end{itemize}
\end{proposition}

We say that $x$ has representation type $\tau=(k_0,\beta^{(0)}; \cdots; k_r,\beta^{(r)})$ if the $\overline{Q}^+$-representation $x$ decomposes into $$x=y_0^{\oplus k_0}\oplus\cdots\oplus y_r^{\oplus k_r}$$ such that $y_s$ are simple $\overline{Q}^+$-representations of dimension $\beta^{(s)}$ and $\zeta(\beta^{(s)})=0$, for $s\in\{0,\ldots, r\}$. Note that there exists $0\le s\le r$ such that $\beta^{(s)}_{\infty}=1$ and $\beta^{(t)}_{\infty}=0$ if $t\neq s$. Let us assume that $\beta^{(0)}_{\infty}=1$, thus $k_0=1$, then $\beta^{(0)}=\mathbf v^{(0)+}$ for a unique $\mathbf v^{(0)}\in \mathbb Z^{Q_0}_{\ge 0}$, and $\beta^{(t)}\in \mathbb Z^{Q_0}_{\ge 0}$ for $t>0$.

\begin{defn}\label{defn: auxiliary quiver}
The auxiliary quiver data $(Q_\tau,\mathbf v_\tau,\mathbf w_\tau)$ associated to the representation type $$\tau=(1,\mathbf v^{(0)+}; k_1,\beta^{(1)}; \cdots; k_r,\beta^{(r)})$$ is defined to have the set of nodes $Q_{\tau,0}:=\{1,\cdots,r\}$, and the Cartan matrix $\mathsf C_\tau$ of $Q_\tau$ is determined by
\begin{align}\label{Cartan matrix of aux quiver}
    (\mathsf C_\tau)_{ij}=\beta^{(i)}\cdot \mathsf C\beta^{(j)}.
\end{align}
The dimension vectors $\mathbf v_\tau,\mathbf w_\tau$ are determined by 
\begin{align}
     v_{\tau,i}=k_i,\quad  w_{\tau,i}=\beta^{(i)}\cdot\left(\mathbf w -\mathsf C\mathbf v^{(0)}\right).
\end{align}
\end{defn}

According to the discussion in the end of \cite[Section 4]{crawley2003normality}, we have the following characterization of $W$.

\begin{proposition}\label{prop: auxiliary quiver}
Let $x\in \mu^{-1}(\zeta)$ be a point with closed $G$-orbit, and its representation type is denoted by $\tau=(1,\mathbf v^{(0)+};k_1,\beta^{(1)};\cdots;k_r,\beta^{(r)})$. Let $(Q_\tau,\mathbf v_\tau,\mathbf w_\tau)$ be the auxiliary quiver data associated to $\tau$. Then $G_x\cong \GL(\mathbf v_\tau)$, and there is an isomorphism of $G_x$-modules $$W\cong T^*\left(\mathrm{Rep}(\mathbf v_{\tau},\mathbf w_{\tau})\oplus F_\tau\right),$$ where $F_\tau$ is a vector space with trivial $G_x$-action, and it has dimension
\begin{align*}
    \dim F_\tau=\mathbf v^{(0)}\cdot\left(\mathbf w -\frac{1}{2}\mathsf C\mathbf v^{(0)}\right).
\end{align*}
Moreover, $\hat\mu$ agrees with the composition of the projection $W\to T^*\mathrm{Rep}(\mathbf v_{\tau},\mathbf w_{\tau})$ with the moment map for $T^*\mathrm{Rep}(\mathbf v_{\tau},\mathbf w_{\tau})$. In particular we have a $\GL(\mathbf v_\tau)$-equivariant isomorphism
\begin{align}
    \hat\mu^{-1}(0)\cong \mu_\tau^{-1}(0)\times T^*F_\tau,
\end{align}
where $\mu_\tau :T^*\mathrm{Rep}(\mathbf v_\tau,\mathbf w_\tau)\to \gl(\mathbf v_\tau)^*$ is the moment map for $(Q_\tau,\mathbf v_\tau,\mathbf w_\tau)$. Thus we have an isomorphism
\begin{align}
    \hat\mu^{-1}(0)\sslash G_x\cong \mathcal M^0(\mathbf v_{\tau},\mathbf w_{\tau})\times T^*F_\tau.
\end{align}
\end{proposition}

The auxiliary quiver of a Dynkin quiver is again a Dynkin quiver, more precisely we have the following lemma.

\begin{lem}\label{lem: aux quiver of Dynkin quiver}
Let $Q$ be a Dynkin quiver, then any auxiliary quiver $Q_\tau$ is isomorphic to a sub-quiver of $Q$, in particular $Q_\tau$ is also a Dynkin quiver. Moreover the transition map $\mathcal Z\to \mathcal Z_\tau$ is surjective.
\end{lem}

\begin{proof}
Let $x\in \widetilde{\mathcal M}^0(\mathbf v_\tau,\mathbf w_\tau)$ be of representation type $\tau=(1,\mathbf v^{(0)+};k_1,\beta^{(1)};\cdots;k_r,\beta^{(r)})$, and denote by $\zeta=\mu(x)$, then $\beta^{(1)},\cdots,\beta^{(r)}$ are roots of $Q$ such that $\beta^{(s)}\cdot\zeta=0$ \cite[Theorem 1.2]{crawley2001geometry}. Note that $\beta^{(1)},\cdots,\beta^{(r)}$ are real roots since $Q$ is a Dynkin quiver, whence $\beta^{(1)},\cdots,\beta^{(r)}$ are distinct because there is a unique (up to isomorphism) simple representation of a given dimension which equals to a real root.

Denote by $\mathfrak q$ the simple Lie algebra associated to $Q$. Identify $\zeta$ with an element of a Cartan subalgebra $\mathfrak h$ of $\mathfrak q$, then we define $\mathfrak l$ to be a Levi subalgebra of $\mathfrak q$. It is easy to see that $\mathfrak l$ is reductive with a Cartan subalgebra $\mathfrak h$. The root system of the semi-simple part of $\mathfrak l$ is a sub-system of the root system of $\mathfrak q$ and therefore annihilated $\zeta$, so the quiver associated to its simple roots, denoted by $Q'$, is isomorphic to a sub-quiver of $Q$. For an arbitrary $s\in \{1,\cdots,r\}$, the root $\beta^{(s)}$ is also a root of $\mathfrak l$, and moreover $\beta^{(s)}$ can not be written as sum of positive roots of $\mathfrak l$ which are annihilated by $\zeta$ according to \cite[Theorem 1.2]{crawley2001geometry}, thus $\beta^{(s)}$ is a simple root of $\mathfrak l$. It follows that $Q_\tau$ is a sub-quiver of $Q'$, thus $Q_\tau$ is isomorphic to a sub-quiver of $Q$. 

Since $\beta^{(s)},s=1,\cdots,r$ are distinct simple roots of $\mathfrak l$, it follows that the linear map $\mathbb C^{Q_{\tau,0}}\to \mathbb C^{Q_0}$ such that it maps the $s$-th unit vector $\mathbf e_s$ to $\beta^{(s)}$, is injective. The transition map $\mathcal Z\to \mathcal Z_\tau$ is the dual to the linear map $\mathbb C^{Q_{\tau,0}}\hookrightarrow \mathbb C^{Q_0}$, so it is surjective.
\end{proof}

Some properties are hereditary from a quiver to its auxiliary quiver.

\begin{lem}\label{lem: inherit flatness}
If $\mu$ is flat, then for every auxiliary quiver $(Q_\tau,\mathbf v_\tau,\mathbf w_\tau)$, the moment map $\mu_\tau$ is flat.
\end{lem}

\begin{proof}
It suffices to show that $\dim_0\mu_\tau^{-1}(0)=\dim T^*\mathrm{Rep}(\mathbf v_{\tau},\mathbf w_{\tau})-\dim\gl(\mathbf v_\tau)$, where $\dim_x X$ of a point $x$ in a variety $X$ is defined as the dimension of a sufficiently small Zariski open neighborhood of $x$. By the Proposition \ref{prop: auxiliary quiver}, it is equivalent to showing that $\dim_0\hat\mu^{-1}(0)=\dim W-\dim\gl(\mathbf v_\tau)$. Since $\dim L=\dim G-\dim G_x$ and $\dim W+\dim L=\dim C$, it boils down to showing that $\dim_{(1,0)} G\times _{G_x}\hat\mu^{-1}(0)=\dim C+\dim L-\dim\gl(\mathbf v)$. Finally, since $\dim L=\dim [A,x]$ and $M=[A,x]\oplus C$ and an étale neighborhood of $(1,0)$ in $G\times _{G_x}\hat\mu^{-1}(0)$ is isomorphic to an étale neighborhood of $x$ in $\mu^{-1}(\zeta)$, we only need to show that $\dim_x\mu^{-1}(\zeta)=\dim M-\dim\gl(\mathbf v)$, which holds by the flatness of $\mu$.
\end{proof}

\begin{lem}\label{lem: inherit simple locus}
If the simple locus $\mu^{-1}(0)^{\mathrm{reg}}$ is non-empty, then for every auxiliary quiver $(Q_\tau,\mathbf v_\tau,\mathbf w_\tau)$, the simple locus $\mu_\tau^{-1}(0)^{\mathrm{reg}}$ is non-empty.
\end{lem}

\begin{proof}
We use the equivalent characterization of non-emptiness of simple locus (Proposition \ref{prop: simple locus criterion}). Namely, let $\mathbf v_\tau=\sum_{l=0}^t\mathbf v_\tau^{(l)}$ be a decomposition such that $t\ge 1$ and $\mathbf v_\tau^{(l)}\in \mathbb Z^{Q_{\tau,0}}_{\ge 0}$, $l=0,1,\cdots,t$, and $\mathbf v_\tau^{(s)}\neq 0$ if $s>0$. We need to show that the inequality \eqref{key inequality} holds strictly, i.e.
\begin{align}\label{key inequality_temp}
    \mathbf v_\tau\cdot(2\mathbf w_\tau-\mathsf C_\tau\mathbf v_\tau)- \mathbf v_\tau^{(0)}\cdot(2\mathbf w_\tau-\mathsf C_\tau\mathbf v_\tau^{(0)})>\sum_{l=1}^t (2-\mathbf v_\tau^{(l)}\cdot\mathsf C_\tau\mathbf v_\tau^{(l)}).
\end{align}
Write $\tau=(1,\mathbf v^{(0)+};k_1,\beta^{(1)};\cdots;k_r,\beta^{(r)})$, and set
\begin{align*}
    \bar{\mathbf v}^{(0)}:=\mathbf v^{(0)}+\sum_{i=1}^r v_{\tau,i}^{(0)}\beta^{(i)},\quad \bar{\mathbf v}^{(l)}:=\sum_{i=1}^r  v_{\tau,i}^{(l)}\beta^{(i)},\; l=1,\cdots,t.
\end{align*}
Then $\mathbf v=\sum_{l=0}^t\bar{\mathbf v}^{(l)}$ is a decomposition such that $\bar{\mathbf v}_\tau^{(l)}\in \mathbb Z^{Q_{0}}_{\ge 0}$, $l=0,1,\cdots,t$, and $\bar{\mathbf v}_\tau^{(s)}\neq 0$ if $s>0$. Moreover we have
\begin{align}
\text{LHS of \eqref{key inequality_temp}}=\mathbf v\cdot(2\mathbf w-\mathsf C\mathbf v)-\bar{\mathbf v}^{(0)}\cdot(2\mathbf w-\mathsf C\bar{\mathbf v}^{(0)}),
\label{key inequality_tempV2}
\end{align}
and 
\begin{align*}
\text{RHS of \eqref{key inequality_temp}}=\sum_{l=1}^t (2-\bar{\mathbf v}^{(l)}\cdot\mathsf C\bar{\mathbf v}^{(l)}).
\end{align*}
Applying Proposition \ref{prop: simple locus criterion} to the decomposition $\mathbf v=\sum_{l=0}^t\bar{\mathbf v}^{(l)}$, we have
\begin{align*}
    \mathbf v\cdot(2\mathbf w-\mathsf C\mathbf v)-\bar{\mathbf v}^{(0)}\cdot(2\mathbf w-\mathsf C\bar{\mathbf v}^{(0)})>\sum_{l=1}^t (2-\bar{\mathbf v}^{(l)}\cdot\mathsf C\bar{\mathbf v}^{(l)}),
\end{align*}
whence the inequality \eqref{key inequality_temp} follows.
\end{proof}

\begin{lem}\label{lem: inherit positivity}
If $\mathbf w-\mathsf C\mathbf v\in \mathbb Z^{Q_0}_{\ge 0}$, then for every auxiliary quiver $(Q_\tau,\mathbf v_\tau,\mathbf w_\tau)$, $\mathbf w_\tau-\mathsf C_\tau\mathbf v_\tau\in \mathbb Z^{Q_{\tau,0}}_{\ge 0}$.
\end{lem}

\begin{proof}
Write $\tau=(1,\mathbf v^{(0)+};k_1,\beta^{(1)};\cdots;k_r,\beta^{(r)})$, then the $i$-th coordinate $(\mathbf w_\tau-\mathsf C_\tau\mathbf v_\tau)_i$ can be written as
\begin{align*}
    (\mathbf w_\tau-\mathsf C_\tau\mathbf v_\tau)_i=\beta^{(i)}\cdot (\mathbf w-\mathsf C\mathbf v)\ge 0.
\end{align*}
In other words $\mathbf w_\tau-\mathsf C_\tau\mathbf v_\tau\in \mathbb Z^{Q_{\tau,0}}_{\ge 0}$. 
\end{proof}

\subsection{Étale slice for \texorpdfstring{$\widetilde\mu^{-1}(0)$}{}}

Proposition \ref{prop: étale slice} gives a characterization of the étale neighborhood of a point in $\mu^{-1}(\zeta)\sslash \GL(\mathbf v)$. In this subsection we generalize the construction to give a characterization of the étale neighborhood of a point in $x\in \widetilde\mu^{-1}(0)\sslash \GL(\mathbf v)$.

Using the same notation as the previous subsection, except that now $x\in \widetilde\mu^{-1}(0)$ such that, for $G=\GL(\mathbf v)$, the orbit $Gx$ is closed in $\widetilde\mu^{-1}(0)$, and we define 
\begin{align*}
    \widetilde M:=M\oplus\mathcal Z,
\end{align*}
where $\mathcal Z$ is the linear subspace of $A^*$ which annihilates $[A,A]$, see \eqref{the center Z}. $\widetilde M$ is naturally a $A$-$A$-bimodule where the $A$-$A$-bimodule structure on $\mathcal Z$ is trivial. Then we modify the construction of étale slice in the previous subsection as follows.
\begin{itemize}
    \item We have the same $A_x$-$A_x$-bimodule decomposition $A=A_x\oplus L$.
    \item Define $\widetilde C:=C\oplus \mathcal Z$, so that we have an $A_x$-$A_x$-bimodule decomposition $\widetilde M = [A, x] \oplus \widetilde C$.
    \item Let $\widetilde W:=W\oplus \mathcal Z$, then we have $A_x$-$A_x$-bimodule decomposition $\widetilde C=C^\perp\oplus \widetilde W$.
    \item Let $\widetilde\mu_x: \widetilde M\to A^*_x$ be the composition of $\widetilde\mu:\widetilde M\to A^*$ with restriction map $A^*\to A^*_x$. Let $\hat{\widetilde\mu}: \widetilde W\to A^*_x$ be the restriction of $\widetilde\mu_x$ to $\widetilde W$, then $\hat{\widetilde\mu}$ is the extended moment map for $\widetilde W$.
    \item  Define $\widetilde\nu:\widetilde C\to L^*$ by $\widetilde\nu(c+z)(\ell)=\omega(c,[\ell, x])+\omega(c,\ell c)+\mathrm{pr}_{L^*}(z)$, where $\mathrm{pr}_{L^*}: A^*\to L^*$ is the restriction map. Thus
    \begin{align*}
        \widetilde\mu(x+c+z)(a+\ell)=\widetilde\mu_x(c+z)(a)+\widetilde\nu(c+z)(\ell)
    \end{align*}
    for $c\in C$, $z\in \mathcal Z$, $a\in A_x$ and $\ell\in L$.
\end{itemize}

The same argument in the proof of {\cite[Lemma 4.4-4.8, Theorem 4.9]{crawley2003normality}} can be applied to the extended version, and we have the following analog of Proposition \ref{prop: étale slice}.

\begin{proposition}\label{prop: étale slice for extended moment map}
$\widetilde\nu:\widetilde C\to L^*$ is smooth at $0$, thus $\widetilde\nu^{-1}(0)$ is smooth at $0$. The tangent space of $\widetilde\nu^{-1}(0)$ at $0$ is $\widetilde W$. Moreover, we have the following.
\begin{itemize}
    \item[(1)] The map $\widetilde\phi:G\times_{G_x} \widetilde C\to \widetilde M$ such that $\widetilde\phi(g,c+z)=g(x+c+z)$ is étale at $(1,0)$, and  $\widetilde\phi^{-1}(\widetilde\mu^{-1}(0))=G\times_{G_x}(\widetilde\mu_x^{-1}(0)\cap \widetilde\nu^{-1}(0))$.
    \item[(2)] Denote by $\widetilde\varphi: \widetilde\nu^{-1}(0)\to \widetilde W$ the restriction of the projection $\widetilde C\to \widetilde W$ to $\widetilde\nu^{-1}(0)$, then $\widetilde\varphi$ is étale at $0$ and it is $G_x$-equivariant, and $\widetilde\varphi^{-1}(\hat{\widetilde\mu}^{-1}(0))=\widetilde\mu_x^{-1}(0)\cap \widetilde\nu^{-1}(0)$.
    \item[(3)] It follows from (1) and (2) that the maps $\widetilde\phi$ and $\mathrm{id}_G\times\widetilde\varphi$ in the following diagram are étale at $(1,(0,0))$
    \begin{equation*}
    \begin{tikzcd}
        & G\times _{G_x}(\widetilde\mu_x^{-1}(0)\cap \widetilde\nu^{-1}(0)) \ar[dl,"\widetilde\phi"]\ar[dr,"\mathrm{id}_G\times\widetilde\varphi" '] &\\
        \widetilde\mu^{-1}(0) & &  G\times _{G_x}\hat{\widetilde\mu}^{-1}(0)
    \end{tikzcd}
    \end{equation*}
    The induced maps between quotients are étale at $0$ 
    \begin{equation*}
    \begin{tikzcd}
        & \widetilde\mu_x^{-1}(0)\cap \widetilde\nu^{-1}(0)\sslash G_x \ar[dl,"\widetilde\phi/ G_x"]\ar[dr,"\widetilde\varphi/ G_x" '] &\\
        \widetilde\mu^{-1}(0)\sslash G & & \hat{\widetilde\mu}^{-1}(0)\sslash G_x
    \end{tikzcd}
    \end{equation*}
    Therefore the neighborhood of $x\in \widetilde\mu^{-1}(0)$ is isomorphic to the neighborhood of $(1,0)\in  G\times _{G_x}\hat{\widetilde\mu}^{-1}(0)$ in étale topology, and the neighborhood of $x\in \widetilde\mu^{-1}(0)\sslash G$ is isomorphic to the neighborhood of $0\in \hat{\widetilde\mu}^{-1}(0)\sslash G_x$ in étale topology.
\end{itemize}
\end{proposition}

The analog of Proposition \ref{prop: auxiliary quiver} is as follows.

\begin{proposition}
If $x$ has representation type $\tau=(1,\mathbf v^{(0)+};k_1,\beta^{(1)};\cdots;k_r,\beta^{(r)})$, and denote by $(Q_\tau,\mathbf v_\tau,\mathbf w_\tau)$ the auxiliary quiver in Proposition \ref{prop: auxiliary quiver}, then there is a $\GL(\mathbf v_\tau)$-equivariant isomorphism
\begin{align}
    \hat{\widetilde\mu}^{-1}(0)\cong \left(\widetilde\mu_\tau ^{-1}(0)\times _{\mathcal Z_\tau}\mathcal Z\right)\times T^*F_\tau,
\end{align}
where $\widetilde\mu_\tau :T^*\mathrm{Rep}(\mathbf v_\tau,\mathbf w_\tau)\times \mathcal Z_\tau\to \gl(\mathbf v_\tau)^*$ is the extended moment map for $(Q_\tau,\mathbf v_\tau,\mathbf w_\tau)$, and the map $\mathcal Z\to \mathcal Z_\tau$ is induced by the restriction map $A^*\to A^*_x$. Thus we also have isomorphism
\begin{align}
    \hat{\widetilde\mu}^{-1}(0)\sslash G_x\cong \left(\widetilde{\mathcal M}^0(\mathbf v_\tau,\mathbf w_\tau)\times _{\mathcal Z_\tau}\mathcal Z\right)\times T^*F_\tau.
\end{align}
\end{proposition}

\subsection{Super analog of étale slice}

Next, we extend the construction of étale slice to the super setting. Let us fix a coordinate vector $\mathbf u\in \mathbb Z^{Q_0}_{\ge 0}$ and a $Q_0$ graded vector space $U$ such that $\dim U=\mathbf u$, and define the super symplectic vector space $$M_{\mathsf s}:=M\oplus\Pi S,\quad S:=\bigoplus_{i\in Q_0}T^*\Hom(U_i,V_i).$$ $M_{\mathsf s}$ is equipped with standard super symplectic structure $\omega_{\mathsf s}$ such that it is balanced with respect to the $A$-$A$-bimodule structure, i.e.\  $\omega_{\mathsf s}(x,ay)=\omega_{\mathsf s}(xa,y)$ for $a\in A, x,y\in M_{\mathsf s}$. We have defined a moment map $\mu_{\mathsf s}:M_{\mathsf s}\to A^*$ in \eqref{super moment map}. Let us fix $\zeta\in \mathcal Z$ and $x\in \mu^{-1}(\zeta)$ such that it has closed $G$-orbit, where $\mu$ is the usual moment map. Then $x$ is automatically a point in $\mu_{\mathsf s}^{-1}(\zeta)$. We would like to give a characterization of étale neighborhood of $x$ in $\mu_{\mathsf s}^{-1}(\zeta)$, and also étale neighborhood of its image in $\mu_{\mathsf s}^{-1}(\zeta)\sslash G$. 
\begin{itemize}
    \item We have the same $A_x$-$A_x$-bimodule decomposition $A=A_x\oplus L$.
    \item Define $C_{\mathsf s}:=C\oplus \Pi S$, so that we have a $\mathbb Z_2$-graded $A_x$-$A_x$-bimodule decomposition $M_{\mathsf s} = [A, x] \oplus C_{\mathsf s}$.
    \item Let $W_{\mathsf s}:=W\oplus \Pi S$, then we have $A_x$-$A_x$-bimodule decomposition $C_{\mathsf s}=C^\perp\oplus W_{\mathsf s}$.
    \item Let $\mu_{\mathsf s,x}:  M_{\mathsf s}\to A^*_x$ be the composition of $\mu_{\mathsf s}:M_{\mathsf s}\to A^*$ with restriction map $A^*\to A^*_x$. Let $\hat{\mu}_{\mathsf s}: W_{\mathsf s}\to A^*_x$ be the restriction of $\mu_{\mathsf s,x}$ to $W_{\mathsf s}$, then $\hat{\mu}_{\mathsf s}$ is the moment map for $W_{\mathsf s}$.
    \item  Define $\nu_{\mathsf s}:C_{\mathsf s}\to L^*$ by $\nu_{\mathsf s}(c+s)(\ell)=\omega_{\mathsf s}(c,[\ell, x])+\omega_{\mathsf s}(c,\ell c)+\omega_{\mathsf s}(s,\ell s)$. Thus
    \begin{align*}
        \mu_{\mathsf s}(x+c+s)(a+\ell)=\zeta(a+\ell)+\mu_{\mathsf s,x}(c+s)(a)+\nu_{\mathsf s}(c+s)(\ell)
    \end{align*}
    for $c\in C$, $s\in \Pi S$, $a\in A_x$ and $\ell\in L$.
\end{itemize}

\begin{proposition}\label{prop: étale slice for super moment map}
$\nu_{\mathsf s}:C\to L^*$ is smooth at $0$, thus $\nu_{\mathsf s}^{-1}(0)$ is smooth at $0$. The tangent space of $\nu_{\mathsf s}^{-1}(0)$ at $0$ is $W_{\mathsf s}$. Moreover, we have the following.
\begin{itemize}
    \item[(1)] The map $\phi_{\mathsf s}:G\times_{G_x} C_{\mathsf s}\to M_{\mathsf s}$ such that $\phi_{\mathsf s}(g,c+s)=g(x+c+s)$ is étale at $(1,0)$, and $\phi_{\mathsf s}^{-1}(\mu_{\mathsf s}^{-1}(\zeta))=G\times_{G_x}(\mu_{\mathsf s,x}^{-1}(0)\cap \nu_{\mathsf s}^{-1}(0))$.
    \item[(2)] Denote by $\varphi_{\mathsf s}: \nu_{\mathsf s}^{-1}(0)\to W_{\mathsf s}$ the restriction of the projection $C_{\mathsf s}\to W_{\mathsf s}$ to $\nu_{\mathsf s}^{-1}(0)$, then $\varphi_{\mathsf s}$ is étale at $0$ and it is $G_x$-equivariant, and $\varphi_{\mathsf s}^{-1}(\hat\mu_{\mathsf s}^{-1}(0))=\mu_{\mathsf s,x}^{-1}(0)\cap \nu_{\mathsf s}^{-1}(0)$.
    \item[(3)] It follows from (1) and (2) that the maps $\phi_{\mathsf s}$ and $\mathrm{id}_G\times\varphi_{\mathsf s}$ in the following diagram are étale at $(1,0)$
    \begin{equation*}
    \begin{tikzcd}
        & G\times _{G_x}(\mu_{\mathsf s,x}^{-1}(0)\cap \nu_{\mathsf s}^{-1}(0)) \ar[dl,"\phi_{\mathsf s}"]\ar[dr,"\mathrm{id}_G\times\varphi_{\mathsf s}" '] &\\
        \mu_{\mathsf s}^{-1}(\zeta) & &  G\times _{G_x}\hat\mu_{\mathsf s}^{-1}(0)
    \end{tikzcd}
    \end{equation*}
    The induced maps between quotients are étale at $0$
    \begin{equation*}
    \begin{tikzcd}
        & \mu_{\mathsf s,x}^{-1}(0)\cap \nu_{\mathsf s}^{-1}(0)\sslash G_x \ar[dl,"\phi_{\mathsf s}/ G_x"]\ar[dr,"\varphi_{\mathsf s}/ G_x" '] &\\
        \mu_{\mathsf s}^{-1}(\zeta)\sslash G & & \hat\mu_{\mathsf s}^{-1}(0)\sslash G_x
    \end{tikzcd}
    \end{equation*}
    Therefore the neighborhood of $x\in \mu_{\mathsf s}^{-1}(\zeta)$ is isomorphic to the neighborhood of $(1,0)\in  G\times _{G_x}\hat\mu_{\mathsf s}^{-1}(0)$ in étale topology, and the neighborhood of $x\in \mu_{\mathsf s}^{-1}(\zeta)\sslash G$ is isomorphic to the neighborhood of $0\in \hat\mu_{\mathsf s}^{-1}(0)\sslash G_x$ in étale topology.
\end{itemize}
\end{proposition}

\begin{proof}
The smoothness of $\nu_{\mathsf s}$ follows from the smoothness of $\nu$ since the tangent map of the latter is already surjective. The tangent space of $\nu_{\mathsf s}^{-1}(0)$ at $0$ is the kernel of the tangent map at $0$, which is $W\oplus\Pi S=W_{\mathsf s}$.

The tangent map of $\phi_{\mathsf s}$ at $(1,0)$ is surjective, thus it is an isomorphism by dimensions, therefore $\phi_{\mathsf s}$ is étale at $(1,0)$. Notice that $\phi_{\mathsf s}(1,c+s)\in \mu_{\mathsf s}^{-1}(\zeta)$ if and only if $\mu_{\mathsf s}(c+s)=0$ and $\nu_{\mathsf s}(c+s)=0$; moreover, $\phi_{\mathsf s}$ is $G$-equivariant, thus $\phi_{\mathsf s}^{-1}(\mu_{\mathsf s}^{-1}(\zeta))=G\times_{G_x}(\mu_{\mathsf s,x}^{-1}(0)\cap \nu_{\mathsf s}^{-1}(0))$. This proves (1).

The tangent map at $0$ induced by the natural projection $\varphi_{\mathsf s}:\nu_{\mathsf s}^{-1}(0)\to W_{\mathsf s}$ is an isomorphism, thus $\varphi_{\mathsf s}$ is étale at $0$. $\varphi_{\mathsf s}$ is $G_x$-equivariant because $C_{\mathsf s}=C^\perp\oplus W_{\mathsf s}$ is an $A_x$-$A_x$-bimodule decomposition. The restriction of $\mu_{\mathsf s,x}$ to $C_{\mathsf s}$ factors into $\hat\mu_{\mathsf s}\circ \varphi_{\mathsf s}$, thus $\varphi_{\mathsf s}^{-1}(\hat\mu_{\mathsf s}^{-1}(0))=\mu_{\mathsf s,x}^{-1}(0)\cap \nu_{\mathsf s}^{-1}(0)$. This proves (2).

Finally, we prove (3). It remains to show that the quotient morphisms $\phi_{\mathsf s}/G$ and $\varphi_{\mathsf s}/G_x$ are étale. We notice that the morphism $\phi_{\mathsf s}$ (resp. $\varphi_{\mathsf s}$), after modulo the ideals on the source and the target which are generated by odd elements, agree with the morphism $\phi$ (resp. $\varphi$) in the Proposition \ref{prop: étale slice}. To show that the quotient morphism $\phi/G$ is étale, \cite{crawley2003normality} uses Luna's Fundamental Lemma \cite[p94]{luna1973slices} together with the Lemma in \cite[p95]{luna1973slices} to conclude that there exist $G$-invariant affine open subsets $U\subset G\times _{G_x}(\mu_x^{-1}(0)\cap \nu^{-1}(0))$ and $V\subset \mu^{-1}(\zeta)$ such that
\begin{itemize}
    \item $U$ and $V$ are $G$-saturated, i.e.\  $q^{-1}(q(V))=V$ and $q'^{-1}(q'(U))=U$, where $q:\mu^{-1}(\zeta)\to \mu^{-1}(\zeta)\sslash G$ and $q':G\times _{G_x}(\mu_x^{-1}(0)\cap \nu^{-1}(0))\to G\times _{G_x}(\mu_x^{-1}(0)\cap \nu^{-1}(0))\sslash G$ are quotient maps,
    \item $\phi|_U$ is a strongly étale $G$-morphism of $U$ onto $V$, i.e.\  $\phi|_U:U\to V$ is étale, and $\phi|_U/G: U\sslash G\to V\sslash G$ is étale, and moreover the natural diagram is Cartesian
    \begin{equation*}
    \begin{tikzcd}
    U \ar[d] \ar[r,"\phi|_U"] & V \ar[d]\\
    U\sslash G \ar [r,"\phi|_U/G"] & V\sslash G
    \end{tikzcd}.
    \end{equation*}
\end{itemize}
We use the same notation $U$ and $V$ to denote the corresponding open subschemes in the super-schemes $G\times _{G_x}(\mu_{\mathsf s,x}^{-1}(0)\cap \nu_{\mathsf s}^{-1}(0))$ and $\mu_{\mathsf s}^{-1}(\zeta)$ respectively, then the statement (3) follows from the next lemma.
\end{proof}

\begin{lem}\label{lem: strongly étale morphism}
Let $f:U\to V$ be an étale $G$-morphism between affine $G$-schemes. Suppose that $f^{\mathrm{bos}}:U^{\mathrm{bos}}\to V^{\mathrm{bos}}$ is a strongly étale $G$-morphism, then $f$ is also a strongly étale $G$-morphism.
\end{lem}

\begin{proof}
Let $U=\Spec B,V=\Spec A$, where $A$ and $B$ are super-commutative algebras, and we denote $A=A^0\oplus A^1$ and $B=B^0\oplus B^1$ the decomposition of $A$ and $B$ into $\mathbb Z_2$-homogeneous components. Since $f$ is étale, this implies that $f^0:\Spec B^0\to \Spec A^0$ is étale and $B^1\cong A^1\otimes _{A^0} B^0$ by \cite[Lemma 7.10]{moosavian2019existence}. We denote by $\mathfrak n_A$ (resp. $\mathfrak n_B$) the nilpotent ideal in $A$ (resp. in $B$) generated by $A^1$ (resp. by $B^1$). Consider the filtration on $A$ (resp. $B$) induced by the powers of $\mathfrak n_A$ (resp. $\mathfrak n_B$), and denote by $\overline{f}^*$ the induced map between associated graded super-commutative algebras, i.e.
\begin{align*}
    \overline{f}^*:\mathrm{gr}A=\bigoplus_{i=0}^k\mathfrak n_A^i/\mathfrak n_A^{i+1}\longrightarrow \bigoplus_{i=0}^k\mathfrak n_B^i/\mathfrak n_B^{i+1}=\mathrm{gr}B,
\end{align*}
where $k$ is the maximal integer such that $\mathfrak n_A^{k}\neq 0$ or $\mathfrak n_B^{k}\neq 0$. Note that $(f^{\mathrm{bos}})^*$ is identified with the degree zero map $\overline{f}^*_0:\mathrm{gr}_0A=A/\mathfrak n_A\to B/\mathfrak n_B=\mathrm{gr}_0B$. Since $\mathfrak n_B\cong \mathfrak n_A\otimes_A B$, it follows that $B/\mathfrak n_B\cong A/\mathfrak n_A\otimes_A B$, thus $\overline{f}_0: \Spec B/\mathfrak n_B\to \Spec A/\mathfrak n_A$ is étale. By \cite[Lemma 7.7]{moosavian2019existence}, the natural map
\begin{align*}
    \mathfrak n_A^i/\mathfrak n_A^{i+1}\otimes_{\mathrm{gr}_0 A} \mathrm{gr}_0 B\longrightarrow \mathfrak n_B^i/\mathfrak n_B^{i+1}
\end{align*}
is an isomorphism for all $i$. Now according to the assumption that $\overline{f}_0$ is strongly étale, we have isomorphism $\mathrm{gr}_0 B\cong \mathrm{gr}_0 A\otimes_{(\mathrm{gr}_0 A)^G}(\mathrm{gr}_0 B)^G$, thus the natural map 
\begin{align*}
    \mathfrak n_A^i/\mathfrak n_A^{i+1}\otimes_{(\mathrm{gr}_0 A)^G} (\mathrm{gr}_0 B)^G\longrightarrow \mathfrak n_B^i/\mathfrak n_B^{i+1}
\end{align*}
is an isomorphism for all $i$. Taking $G$-invariant, we get isomorphism $(\mathfrak n_A^i/\mathfrak n_A^{i+1})^G\otimes_{(\mathrm{gr}_0 A)^G} (\mathrm{gr}_0 B)^G\cong (\mathfrak n_B^i/\mathfrak n_B^{i+1})^G$. 

Note that the $G$-invariant of the filtration $A\supset \mathfrak n_A\supset \mathfrak n_A^2\supset \cdots \mathfrak n_A^k\supset 0$ induces a filtration $A^G\supset \mathfrak n_A^G\supset (\mathfrak n_A^2)^G\supset \cdots (\mathfrak n_A^k)^G\supset 0$ such that $(\mathfrak n_A^i)^G/(\mathfrak n_A^{i+1})^G=(\mathfrak n_A^i/\mathfrak n_A^{i+1})^G$ because $G$ is reductive. In particular, we have $\mathrm{gr}_i (A^G)=(\mathrm{gr}_i A)^G$. Similarly we also have $\mathrm{gr}_i (B^G)=(\mathrm{gr}_i B)^G$. Then it follows that the natural map
\begin{align*}
    \mathrm{gr} (A^G)\otimes_{\mathrm{gr}_0 (A^G)} \mathrm{gr}_0 (B^G)\longrightarrow \mathrm{gr} (B^G)
\end{align*}
is an isomorphism. $\mathrm{gr}_0 (A^G)\to \mathrm{gr}_0 (B^G)$ is étale by the assumption that $\overline{f}_0$ is étale, then it follows that for every point $x\in U\sslash G$, the induced map $\mathrm{gr}_0 (A^G)_y^{\wedge}\to \mathrm{gr}_0 (B^G)_x^{\wedge}$ between completed local rings is an isomorphism, where $y$ is the image of $x$ in $V\sslash G$. Since localization and taking completion are exact for finitely generated modules of a Noetherian ring, we have
\begin{align*}
    (\mathrm{gr} (A^G))_y^{\wedge}\cong \bigoplus_{i=0}^k ((\mathfrak n_A^G)^i)_y^{\wedge}/((\mathfrak n_A^G)^{i+1})_y^{\wedge},\quad (\mathrm{gr} (B^G))_x^{\wedge}\cong \bigoplus_{i=0}^k ((\mathfrak n_B^G)^i)_x^{\wedge}/((\mathfrak n_B^G)^{i+1})_x^{\wedge}.
\end{align*}
Hence the map $(A^G)_y^{\wedge}\to (B^G)_x^{\wedge}$ induces isomorphism between graded algebras associated to the filtrations $(A^G)_y^{\wedge}\supset \mathfrak (n_A^G)_y^{\wedge}\supset ((\mathfrak n_A^G)^2)_y^{\wedge}\supset \cdots ((\mathfrak n_A^G)^k)_y^{\wedge}\supset 0$ and $(B^G)_x^{\wedge}\supset \mathfrak (n_B^G)_x^{\wedge}\supset ((\mathfrak n_B^G)^2)_x^{\wedge}\supset \cdots ((\mathfrak n_B^G)^k)_x^{\wedge}\supset 0$ respectively. Thus $A^G\to B^G$ is étale, i.e.\  $f/G: U\sslash G\to V\sslash G$ is étale.

Finally, the induced map $h:U\to U':=(U\sslash G)\times_{V\sslash G}V$ is étale because both $f/G$ and $f$ are étale. $h$ induces isomorphism of underlying topological spaces of $U$ and $U'$ because $f^{\mathrm{bos}}$ is strongly étale. Moreover the map between local completions $h^*:\mathcal O_{U',p}^{\wedge}\to \mathcal O_{U,p}^{\wedge}$ is an isomorphism for all $p\in U$, thus $h$ is an isomorphism. This proves that $f$ is a strongly étale $G$-morphism.
\end{proof}

\begin{proposition}\label{prop: auxiliary super quiver}
Let $x\in \mu^{-1}(\zeta)$ be a point with closed $G$-orbit, and its representation type is denoted by $\tau=(1,\mathbf v^{(0)+};k_1,\beta^{(1)};\cdots;k_r,\beta^{(r)})$. Let $(Q_\tau,\mathbf v,\mathbf w_\tau)$ be the auxiliary quiver data associated to $\tau$ in the Definition \ref{defn: auxiliary quiver}. Then $G_x\cong \GL(\mathbf v_\tau)$, and there is an isomorphism of $G_x$-modules $$W_{\mathsf s}\cong T^*\left(\mathrm{Rep}(\mathbf v_{\tau},\mathbf w_{\tau}|\mathbf u_{\tau})\oplus F_\tau\oplus\Pi E_\tau\right),$$ where $F_\tau$ is the vector space in Proposition \ref{prop: auxiliary quiver}, and $E_\tau$ is a vector space with trivial $G_x$-action, and it has dimension
\begin{align*}
    \dim E_\tau=\mathbf v^{(0)}\cdot\mathbf u,
\end{align*}
and $\mathbf u_\tau$ is determined by
\begin{align*}
    \mathbf u_{\tau,i}=\beta^{(i)}\cdot \mathbf u.
\end{align*}
Moreover, $\hat\mu_{\mathsf s}$ agrees with the composition of the projection $W\to T^*\mathrm{Rep}(\mathbf v_{\tau},\mathbf w_{\tau}|\mathbf u_{\tau})$ with the moment map for $T^*\mathrm{Rep}(\mathbf v_{\tau},\mathbf w_{\tau}|\mathbf u_{\tau})$. In particular we have a $\GL(\mathbf v_\tau)$-equivariant isomorphism
\begin{align}
    \hat\mu_{\mathsf s}^{-1}(0)\cong \mu_{\tau,\mathsf s}^{-1}(0)\times T^*F_\tau\times \Pi(T^*E_\tau),
\end{align}
where $\mu_{\tau,\mathsf s} :T^*\mathrm{Rep}(\mathbf v_\tau,\mathbf w_\tau|\mathbf u_{\tau})\to \gl(\mathbf v_\tau)^*$ is the moment map for $(Q_\tau,\mathbf v_\tau,\mathbf w_\tau|\mathbf u_{\tau})$. Thus we have an isomorphism
\begin{align}
    \hat\mu^{-1}(0)\sslash G_x\cong \mathcal M^0(\mathbf v_{\tau},\mathbf w_{\tau}|\mathbf u_{\tau})\times T^*F_\tau\times \Pi(T^*E_\tau).
\end{align}
\end{proposition}

\section{Vertex algebras and Poisson vertex algebras}\label{appendix:VA-PVA} 

In this appendix we collect some general background definitions and background relevant for vertex algebras (quantum chiral algebras) and Poisson vertex algebras (the classical limit) \cite{li2004vertex}. 

\begin{definition}
    A vertex algebra is a collection of data $(V,|0\rangle,Y(-,z),T)$ with
        \begin{itemize}
            \item a $\mathbb{C}$-vector space $V$ as the vacuum module of the vertex algebra,
            \item a distinguished vector (the vacuum) $|0\rangle \in V$,
            \item a linear map (state--field correspondence) $Y(-,z): V \to (\mathrm{End}_{\mathbb{C}}\,V)[[z,z^{-1}]]~,$ sending $a \in V$ to the field
            $$Y(a,z) = \sum_{n \in \mathbb{Z}} a_{(n)} z^{-n-1}~$$ with modes $a_{(n)}$, where $a_{(-1)}|0\rangle=a$ and $a_{(n>0)}|0\rangle =0$, with $a_{(n)}b=0~,\forall ~ a,b\in V$ if $n\gg 0$, 
            \item a linear translation operator $T: V \to V$,
        \end{itemize}
    subject to a set of axioms --- vacuum, translation, and locality (or Borcherds / Jacobi identity) 
        \begin{itemize}
            \item (Vacuum axiom):  $Y(|0\rangle ,z) = \mathrm{id}_V,$ the identity operator on $V$,  
            \item (Creation property): $Y(a,z)|0\rangle \in a + zV[[z]]$,
            \item (Translation covariance): $[T, Y(a,z)] = Y(T a , z) = \partial_z Y(a,z),~ \forall a\in V$, with  $T|0\rangle = 0$, 
            \item (Locality): for all $a,b \in V$, there exists $N \gg 0$ such that $(z-w)^N [Y(a,z), Y(b,w)] = 0$, whereby the OPE of these fields takes the form
            $$Y(a,z)Y(b,w) \sim \sum_{n \ge 0} \frac{Y(a_{(n)}b,w)}{(z-w)^{n+1}}~.$$
        \end{itemize}
\end{definition}

\begin{definition}
    A vertex Poisson algebra is a collection of data $(V,|0\rangle,\partial,~_{(n)} ~\mathrm{with}~n\geq -1)$ with 
        \begin{itemize}
            \item $\mathbb{C}$-vector space $V$,  
            \item distinguished vacuum vector $|0\rangle\in V$, 
            \item derivation $\partial\in\mathrm{End}_{\mathbb{C}}\,V$, 
            \item product $~_{(n)}:V\otimes V\to V$, $a\otimes b\mapsto a_{(n)}b$,
        \end{itemize}
    where $(V,|0\rangle,\partial,~_{(-1)})$ is a commutative associative unital algebra with derivation $\partial$, and $$(V,|0\rangle,\partial,~_{(n)} ~\mathrm{with}~n\geq 0)$$ is a vertex Lie algebra with
        $$a_{(n)}b=0~,~~\forall~a,b\in V~, ~~ n\gg0~,$$
        $$a_{(n)}b=(-1)^{n+1}\sum_{j=0}^\infty \frac{(-1)^j}{j!}\partial^j(b_{(n+1)}a)~,$$
    with Leibnitz rule $a_{(n)}(bc) = (a_{(n)}b)c + b(a_{(n)}c)$,
        $$a_{(m)}b_{(k)}c-b_{(k)}a_{(m)}c = \sum_{j=0}^\infty \left( \begin{smallmatrix}
            m \\ j
        \end{smallmatrix} \right) (a_{(j)}b)_{(m+k-j)}~, ~~ \forall~ a,b,c\in V,~n\geq 0~,$$
    which implies the Poisson bracket 
    $$\{a(z), b(w)\}=\sum_{n \ge 0} \frac{(a_{(n)}b)(w)}{(z-w)^{n+1}}~,$$
    and
        $$[\partial,a_{(n)}]=(\partial a)_{(n)}=-na_{(n-1)}~,~~\forall~n\geq 0~.$$
\end{definition}

For a scheme $X$, the jet scheme $J_\infty X$ is a $D$-scheme \cite{Beilinson_2004} and its structure sheaf $\mathcal{O}_{J_\infty X}$ carries a canonical derivation $\partial$. In particular, if $X$ is a Poisson scheme \cite{MR2875849}, then $\mathcal{O}_{J_\infty X}$ carries a unique vertex Poisson algebra structure such that, $\forall~f,g\in\mathcal{O}_X\subset \mathcal{O}_{J_\infty X}$, $f_{(0)}g=\{f,g\}$ and $f_{(n>0)}g=0$. Therefore the assignment $X\mapsto J_\infty X$ defines a functor from the category $\mathscr{C}_X$ of Poisson schemes $X$ to the category of vertex Poisson schemes.

\medskip

Given a filtered vertex algebra $V$, the associated graded algebra $\mathrm{gr}_F V$ defined cf. \eqref{def:associated graded algebra} is a vertex Poisson algebra. 

\subsection*{Canonical example: affine vertex algebra \texorpdfstring{$V^k(\mathfrak{g})$}{V} associated to $\mathfrak{g}$}

Let $\mathfrak{g}$ be a finite-dimensional simple Lie algebra with invariant bilinear form $\kappa$. Associated to this Lie algebra is the affine Kac--Moody algebra $$\hat{\mathfrak{g}} =\mathfrak{g} \otimes \mathbb{C}[t,t^{-1}] \oplus \mathbb{C}K$$ with bracket $$[x \otimes t^m, y \otimes t^n]=[x,y] \otimes t^{m+n}+ m \delta_{m+n,0}\,\kappa(x,y)\,K~.$$ The vacuum module at level $k$ for $k \in \mathbb{C}$ is defined as the induced representation $$V^k(\mathfrak{g})=\mathrm{Ind}_{\mathfrak{g}[[t]] \oplus \mathbb{C}K}^{\hat{\mathfrak{g}}} \mathbb{C}_k ~,$$ where $\mathfrak{g}[[t]]$ acts trivially and $K$ acts by $k$. For a basis $\{J^i\}_{i=1,\ldots , \dim\mathfrak{g}}$ of $\mathfrak{g}$, the associated universal affine vertex algebra is the vertex algebra structure on $V^k(\mathfrak{g})$ which is strongly generated by the fields $J^i(z)$ with operator product expansion (OPE)
$$J^i(z)J^j(w)\sim\frac{k\,\kappa(J^i,J^j)}{(z-w)^2}+\frac{[J^i,J^j](w)}{z-w}~.$$

\newpage

\bibliographystyle{alpha}
\bibliography{Bibliography}

\end{document}